\documentclass[reqno,11pt,final]{amsart} 
\usepackage{showlabels}
\usepackage[margin=1in,marginpar=2cm]{geometry}
\usepackage[foot]{amsaddr} 
\usepackage{amssymb}
\usepackage{verbatim} 
\usepackage{amssymb}
\usepackage{amsthm}
\usepackage{amsmath}
\usepackage{amsfonts}
\usepackage{latexsym}
\usepackage{mathtools} 
\usepackage{changepage}
\usepackage{enumitem} 
\usepackage[english]{babel}
\usepackage{esint}
\usepackage{cite}
\usepackage{calrsfs}
\usepackage{hyperref} 
\usepackage{color}
\usepackage{enumitem}

\usepackage{kotex}
\usepackage{autobreak}

\def\Xint#1{\mathchoice
{\XXint\displaystyle\textstyle{#1}}%
{\XXint\textstyle\scriptstyle{#1}}%
{\XXint\scriptstyle\scriptscriptstyle{#1}}%
{\XXint\scriptscriptstyle\scriptscriptstyle{#1}}%
\!\int}
\def\XXint#1#2#3{{\setbox0=\hbox{$#1{#2#3}{\int}$ }
\vcenter{\hbox{$#2#3$ }}\kern-.6\wd0}}

\theoremstyle{plain}

\newtheorem{theorem}{Theorem}
\newtheorem{lemma}[theorem]{Lemma}
\newtheorem{definition}[theorem]{Definition}

\newtheorem{remark}[theorem]{Remark}

\newtheorem{corollary}[theorem]{Corollary}

\usepackage{todonotes}
\numberwithin{equation}{section}
\numberwithin{theorem}{section}

\newcommand{\eqdef }{\overset{\mbox{\tiny{def}}}{=}}

\title[Local Existence and Singularity Formation for Vlasov–Poisson–Isotropic Landau]{Local Existence and Finite-Time Singularity Formation in the Vlasov-Poisson-Isotropic Landau System}

\author[J. W. Jang]{Jin Woo Jang$^\dagger$}
\address{$^\dagger$Department of Mathematics, POSTECH (Pohang University of Science and Technology), Pohang 37673, Republic of Korea. \href{mailto:jangjw@postech.ac.kr}{jangjw@postech.ac.kr} }

\author[J. Kim]{Junsung Kim$^\ddagger$}
\address{$^\ddagger$Department of Mathematics, POSTECH (Pohang University of Science and Technology), Pohang 37673, Republic of Korea. \href{mailto:junsung998@postech.ac.kr}{junsung998@postech.ac.kr} }

\AtBeginDocument{%
   \def\MR#1{}
}
\begin{document}


\date{\today}\begin{abstract}
The isotropic Landau (Coulomb) operator was introduced in kinetic theory by Krieger and Strain \cite{iL_KS_2012}. In this work, we study the spatially inhomogeneous Vlasov--Poisson--isotropic Landau system. We first establish a local--in--time existence theory for the Cauchy problem: for initial data satisfying a suitable smallness condition in an appropriate norm, there exists a non--negative solution on a time interval $[0,T]$, where the lifespan $T$ depends on the size of the initial data.

Beyond the local theory, we investigate a mechanism that may lead to the breakdown of global existence. We show that finite--time singularity formation can occur in the gravitationally attractive case, provided that the weak solution satisfies certain a priori regularity and decay assumptions, the initial gravitational field energy exceeds the kinetic energy, and the resulting energy gap dominates the diffusive effect of the collision operator. As a consequence, if the solution is further assumed to belong to a suitable measure space up to the maximal existence time, it collapses to a single point in physical space at that time. The proof of finite--time singularity formation is based on deriving an upper bound for the second spatial moment, which becomes negative in finite time.
\end{abstract}

\setcounter{tocdepth}{2}

\maketitle
\tableofcontents
\thispagestyle{empty}
\section{Introduction}
\label{Introduction}
In this article, we study the Vlasov--Poisson--isotropic Landau (VPiL) system in $\mathbb{R}^3$,
\begin{align}\label{VPiL}
    \partial_t f + v \cdot \nabla_x f + E_f \cdot \nabla_v f
    = Q_{\mathrm{iso}}(f,f),
    \qquad
    E_f = (\mp)\,\frac{x}{4\pi |x|^3} \star_x \rho_f .
\end{align}
Here, $(-)$ and $(+)$ correspond to the gravitational and plasma cases, respectively.
The symbol $\star_x$ denotes convolution over $\mathbb{R}^3$ in the $x$-variable, and the mass/charge density is defined by
$$
    \rho_f \eqdef \int_{\mathbb{R}^3} f\,dv ,
$$
where the particle mass and charge have been normalized so that $\textup{m}=\textup{e}=1$.
The collision operator $Q_{\mathrm{iso}}$ is an isotropic Landau operator associated with a Coulombic cross section and is given by
\begin{align}\label{def.Qiso}
    Q_{\mathrm{iso}}(f,f)
    = Q_{\mathrm{KS}}(f,f)
    \eqdef (-\Delta_v)^{-1} f \, \Delta_v f + f^2 .
\end{align}
Here, the nonlinear diffusion coefficient term
$$
    (-\Delta_v)^{-1} f(v)
    \eqdef \frac{1}{4\pi}
    \int_{\mathbb{R}^3} \frac{f(u)}{|u-v|}\,du
$$
corresponds to the Newtonian potential of $f$, which may be understood as the (weak) solution to the Poisson equation.
This definition is natural whenever $f$ is locally H\"older continuous in the classical sense or is $L_v^1(\mathbb{R}^3)$ in the sense of distribution.
The subscript ``KS'' refers to the formulation introduced by Krieger and Strain~\cite{iL_KS_2012}.

\subsection{Background on Landau-type Collisional Kinetic Equations}\label{Background}

The mathematical theory of kinetic equations has seen extensive development across both collisional and collisionless regimes. In the collisional regime, the Boltzmann equation plays a central role: it models dilute gases composed of neutral atoms or molecules undergoing short--range binary interactions and governs the evolution of the velocity distribution function through a first--order integro--differential operator. The Landau equation--often regarded as the Boltzmann equation for Coulombic interactions--arises as the grazing--collision limit of the Boltzmann equation and is used widely to model plasmas where particles interact through long--range repulsive Coulomb potentials. These two equations are closely related: repulsive long--range interactions in plasmas generate scattering effects analogous to grazing collisions in rarefied gases, and indeed the rigorous grazing limit of the Boltzmann equation leads to the Landau equation \cite{Limit_Desvillettes_1992}.

It is well known that Maxwellian distributions of the form $Ce^{-c|v|^2}$ are stationary solutions of the Landau equation. Supported by experimental observations \cite{Phy.L_Tokuda_2021, Phy.L_Mannion_2023} as well as classical kinetic theory \cite{HL_Guo_2002, VPL_Guo_2012, VPL_Strain_2012, VPL_Wang_2012}, Maxwellians are understood as thermal equilibrium states of dilute plasmas. Consequently, early mathematical investigations primarily focused on small perturbations around Maxwellians as the first step toward understanding nonequilibrium statistical phenomena. The global well--posedness of such perturbations was first established by Guo \cite{HL_Guo_2002}, who later proved their exponential convergence to equilibrium in \cite{VPL_Guo_2012}. Subsequent works by Strain \cite{VPL_Strain_2012}, Wang \cite{VPL_Wang_2012} and others have extended this perturbative framework and broadened the class of admissible solutions.

For the spatially homogeneous Landau equation, large--data global well--posedness is known for a wide range of interaction potentials \cite{HL_Silvestre_2023, HL_Arsenev_1977, HL_Villani_1998, HL_Villani_2000_1, HL_Villani_2000_2, HL_Villani_2000_3}. In sharp contrast, for the inhomogeneous Landau equation, the well--posedness theory for evolutionary solutions remains largely open. Within the vacuum perturbative framework, progress toward local well--posedness has been made along several directions. The existence of local weak solutions was established in \cite{IHL_Henderson_2019, IHL_Sanchit_2019, IHL_Snelson_2019, IHL_Snelson_2024}. To further obtain properties such as regularity and uniqueness, these works require additional mesoscopic assumptions on the hydrodynamic fields--typically a uniform positive lower bound on the local charge density, together with uniform upper bounds on the charge, energy and entropy. Under such assumptions, one can derive $L^\infty$ a priori bounds for solutions, depending on the potential exponent \cite{IHL_Cameron_2018, IHL_Snelson_2018, IHL_Snelson_2023}. For very soft potentials, further moment bounds on appropriate velocity moments are also needed. These estimates allow the use of Harnack-type inequalities, leading to global Hölder continuity of weak solutions under quantitative control of the hydrodynamic quantities (see Theorem~1.2 in \cite{IHL_Cameron_2018}).

Another line of research exploits Gaussian decay in the velocity variable \cite{IHL_Snelson_2018} to compensate for the polynomial growth of the derivatives of the Landau coefficients. Henderson et al.~\cite{IHL_Henderson_2019} demonstrated that combining H\"older continuity with Schauder estimates yields local Hölder bounds on derivatives, which subsequently imply global smoothness of weak solutions. As a consequence, any finite--time blow--up of $f$ or its derivatives must be accompanied by the loss of at least one uniform boundedness condition on the hydrodynamic quantities. An alternative method was proposed by Sanchit and Henderson et al.~\cite{IHL_Snelson_2019,IHL_Sanchit_2019}, who proved local--in--time existence by assuming weighted regularity of the initial data, without requiring a priori boundedness of hydrodynamic fields.

The perturbative global theory near vacuum has also seen progress across different interaction regimes. In the moderately soft potential case, Luk \cite{IHL_Luk_2019} obtained global existence and stability by combining dispersive decay from the transport operator with weighted energy estimates, while exploiting a favorable null structure in the collision operator. Sanchit \cite{IHL_Sanchit_2020} later extended this strategy to the hard potential regime, where enhanced coercivity allows a technically simpler yet conceptually parallel argument. In the very soft potential regime and the Coulombic case, the global existence problem remains unresolved. Nevertheless, Bedrossian et al.~\cite{VPL_Bedrossian_2022} ruled out the presence of self--similar blow--up solutions, thereby eliminating one natural mechanism for singularity formation.

From a modeling perspective, plasmas also generate long--range self--consistent electromagnetic fields. An alternative physically comprehensive formulation is therefore the Vlasov--Poisson--Landau (VPL) system, in which collisional effects are coupled with transport induced by the self--consistent electric field. Existing mathematical studies of the VPL system have focused largely on perturbations near Maxwellians; see \cite{VPL_Yu_2004, VPL_Guo_2012, VPL_Strain_2012, VPL_Wang_2012}. In contrast, near-vacuum regimes in the whole space remain entirely open for the VPL system. The well-posedness issues differ fundamentally from those of the inhomogeneous Landau equation on $\mathbb{R}^3$: in particular, the definition of the electric field requires spatial integrability that cannot be derived from the conditional regularity framework for the Landau equation, as that theory crucially assumes a uniform positive lower bound on the charge density.

\subsection{The Vlasov–Poisson–Isotropic Landau System and Main Results}
In this paper, we focus on the near--vacuum regime of the Vlasov--Poisson--isotropic Landau (VPiL) system for the Coulomb potential:
\begin{equation}
    \begin{gathered}
        \partial_tf + v\cdot\nabla_xf + E_f\cdot\nabla_vf = (-\Delta_v)^{-1}f\Delta_vf + f^2,\tag{VPiL}\\
        \rho_f = \int_{\mathbb{R}^3} f\,dv, \quad E_f = \mp\frac{x}{4\pi|x|^3} \star_x\rho_f, \nonumber
    \end{gathered}
\end{equation}for $(t,x,v)\in[0,T]\times\mathbb{R}^3\times\mathbb{R}^3$.
A distinctive feature of this model is that the total energy is not conserved and may even increase, in sharp contrast to the classical Vlasov-Poisson-Landau system. More precisely, observe by Lemma \ref{I'''} in our case or \cite{iL_Gualdani_2016} in the homogeneous case that if $f\ge0$,
\begin{align*}
    \frac{d}{dt}\left(\frac{1}{2}\iint_{\mathbb{R}^6} |v|^2f\,dvdx  \pm \frac{1}{2}\int_{\mathbb{R}^3} |E_f|^2\,dx\right) = 2\iint_{\mathbb{R}^6} (-\Delta_v)^{-1}ff\,dvdx\ge0.
\end{align*}The signs $(+)$ and $(-)$ on the left-hand side correspond to the plasma and gravitational cases, respectively. This lack of a dissipative energy law precludes the use of standard global energy methods. Then, this property suggests that the dynamics may amplify concentration rather than promote smoothing, leaving open the scenario of finite--time singularities. As another feature of this model, since in $\mathbb{R}^3$ the electric field must remain well--defined, one cannot assume an \textit{a priori} lower bound on the charge/mass density that is uniform in $x$. This prevents a direct application of existing conditional regularity theories, such as global $L^\infty$ bounds \cite{IHL_Cameron_2018}, $C^\infty$ smoothing \cite{IHL_Henderson_2019}, or Gaussian bounds \cite{IHL_Snelson_2018}.

To the best of our knowledge, there has been no mathematical theory developed for the VPiL system, leaving a gap in the mathematical theory of collisional plasmas. In this work, we establish the local existence of a slowly decaying weak solution to the Cauchy problem \eqref{VPiL}. We further show that any non--negative weak solution to \eqref{VPiL} in the gravitational case, satisfying certain physically relevant conditions, has a finite maximal existence time. In particular, any such non--negative global--in--time solution (if it exists) necessarily loses regularity after some finite time.

Our main theorems can be (informally) summarized as follows:
\begin{itemize}
    \item (\textit{Local existence}) In both gravitational and plasma cases, if the initial data exhibits sufficient polynomial decay together with its derivative, and is small in an appropriate weighted $C^2$ norm, then there exists a polynomially decaying solution to \eqref{VPiL} on the interval $[0,T]$, where $T$ depends on the size of $f^\textup{in}$ in a suitable weighted $C^2$ norm.
    \item (\textit{Finite-time singularity}) We consider the purely gravitationally attractive case. If the gravitational field energy exceeds the kinetic energy and the resulting energy gap is sufficiently large compared to the diffusive effect induced by the collision operator, then the solution to the collisional system develops a finite-time singularity. Moreover, if we further assume that a weak solution exists up to the maximal existence time in an appropriate measure space, then the solution collapses to a single point in space at that time.
\end{itemize}

\subsection{Relationship between Landau and Isotropic Landau Operator and these Basic Properties}

The classical Landau collisional operator $Q(\cdot,\cdot)$ is given by
\begin{align}\label{Landau op}
    Q(f,f) = \sum_{i,j=1}^3\partial_{v_i}\int_{\mathbb{R}^3} \Pi_{ij}(v-v_\star)\left((\partial_{v_j}f)(v)f(v_\star) - (\partial_{v_j}f)(v_\star)f(v)\right)\,dv_\star,
\end{align}where the kernel matrix $(\Pi_{ij}(v))_{1\le i,j\le3}$ in \eqref{Landau op} is given by
\begin{align}\label{kernel}
    \Pi(v) = \frac{1}{8\pi} \left(\mathbb{I}_3 -\frac{v\otimes v}{|v|^2}\right)|v|^{\gamma+2}, \quad \gamma\in[-3,1].
\end{align}Here, the interaction exponent $\gamma$ characterizes different physical regimes: 
$\gamma \in (0,1]$ corresponds to hard potentials, 
$\gamma \in (-2,0)$ to moderately soft potentials, 
and $\gamma \in (-3,-2]$ to very soft potentials. 
In particular, $\gamma = 0$ is known as the Maxwellian molecule case, 
while $\gamma = -3$ represents the Coulomb potential, 
which is regarded as the most physically relevant regime for plasma dynamics. One can simplify the Landau operator \eqref{Landau op} as
\begin{align*}
    Q(f,f) = a_{ij}(f)\partial_{v_iv_j}f - c(f)f,
\end{align*}where for some $a_\gamma, c_\gamma>0$,
\begin{align}\label{leading coeff}
    a_{ij}(f) = a_\gamma\int_{\mathbb{R}^3} |w|^{\gamma+2}\left(\delta_{ij} - \frac{w_iw_j}{|w|^2}\right)f(v-w)\,dw,
\end{align}and
\begin{align*}
    c(f) =
    \begin{cases}
    c_{\gamma} \int_{\mathbb{R}^3} |w|^\gamma f(v-w)\,dw, \quad &\gamma\in(-3,1],\\
    f, &\gamma = -3.
    \end{cases}
\end{align*}

Notice that since $f$ formally satisfies
\begin{align*}
    \iint_{\mathbb{R}^6} \varphi Q(f,f)\,dvdx = 0, \qquad (\varphi = 1,v_i, |v|^2),
\end{align*}we have for $t>0$
\begin{align}\label{Landau mass cons}
    \textup{M}(t)\eqdef \iint_{\mathbb{R}^6} f(t,x,v)\,dvdx = \textup{M}(0),\\ \label{Landau momentum cons}
    \textup{P}(t) \eqdef \iint_{\mathbb{R}^6} vf(t,x,v)\,dvdx = \textup{P}(0),\\ \label{Landau energy cons}
    \textup{KE}(t) \eqdef \iint_{\mathbb{R}^6} \frac{|v|^2}{2}f(t,x,v)\,dvdx = \textup{KE}(0).
\end{align}

In this paper, our main focus is the collisional Vlasov--Poisson system
\eqref{VPiL}, where the collision operator is given by the isotropic variant
\eqref{def.Qiso} of the Coulombic Landau operator. This operator can be viewed
as a simplified model obtained by replacing the anisotropic matrix appearing in
\eqref{kernel} with its velocity-averaged (radially symmetric) counterpart, in
particular neglecting the $v\otimes v$ contribution to the diffusion
structure. Although such an approximation removes the physically correct
directional degeneracy of the Coulomb kernel, it retains the key feature that
the potential term in velocity is governed by the inverse Laplacian
$(-\Delta_v)^{-1}$.

We also have the formal properties of the physical quantities for VPiL.
\begin{remark}[Formal conservation and monotonicity properties for VPiL]
    \label{formal conservation laws}Notice that the operator $Q_\textup{iso}$ exhibits several properties with respect to physical quantities. More precisely, we observe that
\begin{align*}
\int_{\mathbb{R}^3} Q_\textup{iso}(f,f)\,dv = \int_{\mathbb{R}^3} \left((-\Delta_v)^{-1}f\Delta_vf + f^2\right)\,dv = \int_{\mathbb{R}^3} \left(-f^2 + f^2\right) \,dv = 0.
\end{align*}This identity yields the following continuity equation:
\begin{align*}
\partial_t\rho + \nabla_x\cdot j = 0,
\end{align*}where the mass/charge density $\rho(t,x)$ and current density $j(t,x)$ are defined in terms of $f$ by
\begin{align*}
\rho(t,x) \eqdef \int_{\mathbb{R}^3} f(t,x,v)\,dv, \quad  j(t,x) \eqdef \int_{\mathbb{R}^3} vf(t,x,v)\,dv,
\end{align*}where we work in normalized units in which the particle charge and mass satisfy $\textup{e}=\textup{m}=1$.

In addition, the total mass/charge is formally conserved:
\begin{align*}
\frac{d}{dt}\textup{M}(t) \eqdef \frac{d}{dt}\iint_{\mathbb{R}^6} f(t,x,v)\,dvdx = 0.
\end{align*}
On the other hand, we note that the total energy increases in time. This was first observed in the spatially homogeneous case by \cite{iL_Gualdani_2021}. In the spatially inhomogeneous case, we observe that 
\begin{equation}
\begin{gathered}\label{iL energy cons}
    \frac{d}{dt}\textup{E}_{\textup{tot}}(t) \eqdef \frac{d}{dt}\left(\frac{1}{2} \iint_{\mathbb{R}^6} |v|^2f(t,x,v)\,dvdx \pm \frac{1}{2}\int_{\mathbb{R}^3} |\nabla(K\star \rho)(t,x)|^2\,dx\right)\\
    \eqdef \frac{d}{dt}\left(\textup{KE}(t) \pm \mathcal{E}_E(t)\right) =2\iint_{\mathbb{R}^6} (-\Delta_v)^{-1}f f\,dvdx\ge0,
    \end{gathered}
\end{equation}where $K(x) = \pm\frac{1}{4\pi|x|}$. Unlike the classical Vlasov–Poisson–Landau system, in which the total energy is conserved, the present model exhibits an increasing energy. 

Lastly, for $f\ge0$, the entropy dissipation relation formally holds as
\begin{align*}
\frac{d}{dt}\textup{H}(t) &\eqdef \frac{d}{dt}\iint_{\mathbb{R}^6} f(t)\log f(t)\,dxdv\\
&= \iint_{\mathbb{R}^6} \{-v\cdot\nabla_xf - E_f \cdot\nabla_vf + Q_\textup{iso}(f,f)\}\log f\,dvdx\\
&= \iint_{\mathbb{R}^6} Q_\textup{iso}(f,f) \log f\,dvdx = \iint_{\mathbb{R}^6} \{(-\Delta_v)^{-1}f\Delta_vf + f^2\}\log f\,dvdx\\
&= \iint_{\mathbb{R}^6} f\left\{2(-\Delta_v)^{-1}(\nabla_vf)\cdot \frac{\nabla_vf}{f} + (-\Delta_v)^{-1}f \left(\frac{\Delta_vf}{f} - \frac{|\nabla_vf|^2}{f^2}\right)\right\}\,dvdx\\
&= \iint_{\mathbb{R}^6} (-\Delta_v)^{-1}(\nabla_vf)\cdot\nabla_vf - \frac{(-\Delta_v)^{-1}f|\nabla_vf|^{2}}{f}\,dvdx\\
&= -\frac{1}{8\pi}\iint_{\mathbb{R}^9} \frac{f(u)f(v)}{|u-v|} |\nabla_u\log f(u) - \nabla_v\log f(v)|^2\,dudvdx\le0,
\end{align*}where $(-\Delta_v)^{-1}f(v) = \frac{1}{4\pi}\int_{\mathbb{R}^3}\frac{f(u)}{|u-v|}\,du$.
\end{remark}

\subsection{Singularity Formation in the Attractive Case}

Before turning to the nonlinear dynamics of the Vlasov--Poisson--isotropic Landau
model, it is useful to recall a classical diagnostic tool in kinetic theory for detecting
the breakdown of solutions: the \emph{mixed-moment argument}. This method tracks
the evolution of low-order spatial moments and attempts to force the second spatial
moment to become negative in finite time, thereby producing a contradiction. The
mechanism is closely tied to the presence of the nonlinear (gravitational) Poisson
term, which plays an essential role.

For instance, consider the spatially inhomogeneous classical Landau equation
without a self-consistent field. In this case, the mixed-moment argument alone never
produces finite-time breakdown. Indeed, if
$f(t,x,v)\ge 0$ is a regular solution, the second spatial moment
$$
\textup{I}(t) \eqdef \frac12 \iint_{\mathbb{R}^6} |x|^2 f(t,x,v)\, dvdx
$$
is an upward-opening quadratic polynomial in $t$ and therefore remains nonnegative
for all time. To see this, recall the total energy conservation law \eqref{Landau energy cons}. By direct
    differentiation of $\textup{I}(t)$, we find
    \begin{align*}
        \textup{I}'(t) &= \iint_{\mathbb{R}^6} (x \cdot v)\, f(t,x,v)\, dv\, dx,\\
        \textup{I}''(t) &= 2\,\mathrm{KE}(0)>0,\\
        \textup{I}'''(t) &= 0,
    \end{align*}
    where $\mathrm{KE}(0)$ denotes the initial kinetic energy as defined in \eqref{iL energy cons}, which is conserved in
    time. Hence $\textup{I}(t)$ is an upward-opening quadratic polynomial in $t$.

    For a quadratic polynomial to become negative at some finite time, its minimum
    value must be negative. The minimum of $\textup{I}$ occurs at
    $$
        t_* = -\frac{\textup{I}'(0)}{\textup{I}''(0)},
    $$
    and therefore
    $$
        \textup{I}(t_*) = \textup{I}(0) - \frac{\textup{I}'(0)^2}{2\,\textup{I}''(0)}.
    $$
    Thus, a necessary condition for $\textup{I}(t)$ to attain negative values is
    $$
        \textup{I}'(0)^2 - 2\,\textup{I}(0)\,\textup{I}''(0) > 0.
    $$
    However, by the Cauchy--Schwarz inequality we have
    $$
        \textup{I}'(0)^2 \le 2\,\textup{I}(0)\,\textup{I}''(0),
    $$
    so the above condition can never be satisfied. Consequently, $\textup{I}(t)$
    remains nonnegative for all $t$.

This shows that the mixed-moment argument cannot yield any mechanism
for finite-time singularity formation in the classical Landau equation without an attractive self-consistent field.

The situation changes markedly for the classical Vlasov--Poisson--Landau system with
attractive self-interaction. If the initial field energy exceeds twice the kinetic energy,
$\mathcal{E}_E(0) > 2\mathrm{KE}(0)$, then the conserved quantity $\textup{I}''(t)$ becomes
proportional to $(2\mathrm{KE}(0) - \mathcal{E}_E(0))<0$, and the quadratic polynomial
$\textup{I}(t)$ opens downward. In this regime, the mixed-moment argument does permit
finite-time singularity formation.

In contrast, the isotropic Landau equation with a gravitational self-consistent field does not conserve energy (see
\eqref{iL energy cons}), and the Taylor expansion of $\textup{I}(t)$ consequently contains a nontrivial cubic term.  The coefficient of this term is positive, reflecting the strictly increasing energy. This makes the analysis more delicate and requires a refined mixed-moment argument. In particular, we assume that the absolute value of the negative quadratic coefficient $|-(\mathcal{E}_E(0) - \textup{KE}(0))|$ is sufficiently large--equivalently, that $\mathcal{E}_E(0)\gg1$--relative to the cubic coefficient. Under this assumption, the cubic upper bound for $\textup{I}(t)$ can still be forced to become negative within a finite time interval. This refined argument will be carried out for the Vlasov--Poisson--isotropic Landau model in Section~\ref{Singularity formation}.

\section{Main Theorems and Strategies}
Now, we are in position to state the main theorems precisely. We begin by introducing the necessary function spaces. As in \cite{IHL_Luk_2019, IHL_Sanchit_2020}, the solution space $X_T^m$ is defined using $\langle x-vt \rangle$ weight.

\begin{definition}For each $T>0$, we denote the local phase space as
    \begin{align*}
            \mathcal{Q}_T \eqdef (0,T)\times\mathbb{R}^6.
     \end{align*}For $m\in\mathbb{R}$, $1<p_0<\frac{3}{2}$, $3<q_0<\frac{11}{3}$ and $w=\langle x\rangle\langle v\rangle^8$, the solution space $X_T^m$ is defined by
    \begin{align*}
            X_T^m  \eqdef \left\{f \in W_{p_0}^{1,2}(\mathcal{Q}_T)\cup W_{q_0,w}^{1,2}(\mathcal{Q}_T) \ \bigg|\  \sup_{(t,x,v)\in \overline{\mathcal{Q}_T}} \langle x-vt\rangle^m\langle v\rangle^m|f(t,x,v)|<\infty\right\},
    \end{align*}endowed with a norm
    \begin{align*}
            \lVert f\rVert_{X_T^m} \eqdef \lVert \langle x-vt\rangle^m\langle v\rangle^m f(t,x,v)\rVert_{L^\infty(\overline{\mathcal{Q}_T})}.
    \end{align*}Here, $W_{p_0}^{1,2}$ and $W_{q_0,w}^{1,2}$ denote the (weighted) parabolic Sobolev spaces as defined in \eqref{def.para.sob}. The Sobolev conditions in the definition of $X_T^m$ are imposed only to ensure sufficient regularity, while the norm $\|\cdot\|_{X_T^m}$ measures the weighted $L^\infty$ size of the solution.
    
To introduce the space for initial datum, for $m\in\mathbb{R}$, we define the weighted $L^\infty$ space by
\begin{align*}
    L^\infty_{m,2m}(\mathbb{R}^6) \eqdef \biggl\{f : \mathbb{R}^6\to\mathbb{R}\ \bigg|\ \sup_{(x,v)\in\mathbb{R}^6}\langle x\rangle^m \langle v\rangle^{2m}|f(x,v)|<\infty\biggr\},
\end{align*}where $\langle z\rangle \eqdef (1+|z|^2)^{1/2}$ for $z\in\mathbb{R}^3$. The associated norm is given by
\begin{align*}
    \lVert f\rVert_{L^\infty_{m,2m}(\mathbb{R}^6)} \eqdef \sup_{(x,v)\in\mathbb{R}^6} \langle x\rangle^m\langle v\rangle^{2m} |f(x,v)|.
\end{align*}In accordance with this space, we define the initial space, denoted by $\mathcal{C}^2_{m,0}(\mathbb{R}^6)$ as the set of functions $f\in C^2(\mathbb{R}^6)$ such that
\begin{align*}
    f,v\cdot\nabla_xf,\nabla_vf, \Delta_xf, \Delta_vf\in L^\infty_{m,2m}(\mathbb{R}^6),
\end{align*}and
\begin{align*}
    \langle x\rangle^m\langle v\rangle^{2m}\big(|f|+|v\cdot\nabla_xf| + |\nabla_vf| + |\Delta_xf| + |\Delta_vf|\big)\to0 \quad \text{ as }|x|,|v|\to\infty,
\end{align*}endowed with a norm
\begin{align*}
    \|f\|_{\mathcal{C}^2_{m,0}(\mathbb{R}^6)} \eqdef \|(|f| + |v\cdot\nabla_xf| + |\nabla_vf| + |\Delta_xf| + |\Delta_vf|)\|_{L^\infty_{m,2m}(\mathbb{R}^6)}.
\end{align*}

In fact, on any finite time interval [0,T], the quantity $\langle x-vt\rangle^{m}\langle v\rangle^{m}$ can be bounded in terms of $\langle x\rangle^{m}\langle v\rangle^{2m}$ (see Lemma \ref{embed}), which motivates the above choice of weights for the initial data.
\end{definition}

Using this functional setting, we make precise the definition of a weak solution. Let $m>3$, and suppose that $f : \overline{\mathcal{Q}_T} \to \mathbb{R}$ is a classical solution to the initial-value problem associated with \eqref{VPiL}, with initial data $f^\textup{in}$. Then, since the solution is locally integrable, for each $T>0$ and for any test function $\varphi\in C_c^{2}([0,T)\times\mathbb{R}^3\times\mathbb{R}^3)$, the left-hand side of \eqref{VPiL} can be written by
\begin{multline*}
\int_0^T\iint_{\mathbb{R}^3\times\mathbb{R}^3} (\partial_tf + v\cdot\nabla_xf + E_f\cdot\nabla_vf)\varphi\,dxdvds\\
= -\iint_{\mathbb{R}^3\times\mathbb{R}^3} f^\textup{in}\varphi(0)\,dxdv - \int_0^T\iint_{\mathbb{R}^3\times\mathbb{R}^3} f(\partial_t\varphi + v\cdot\nabla_x\varphi + E_f\cdot\nabla_v\varphi)\,dxdvds.
\end{multline*}Using integration by parts and the self-adjointness of $(-\Delta_v)^{-1}$, the right-hand side of \eqref{VPiL} can be written as
\begin{gather*}
\int_0^T \iint_{\mathbb{R}^3\times\mathbb{R}^3} f\Delta_v[(-\Delta_v)^{-1}f\varphi] + f^2\varphi\,dxdvds \\
=\int_0^T \iint_{\mathbb{R}^3\times\mathbb{R}^3} f[2\nabla_v\{(-\Delta_v)^{-1}f\}\cdot\nabla_v\varphi + (-\Delta_v)^{-1}f\cdot\Delta_v\varphi]\,dxdvds.
\end{gather*}
Here, we rely on the sufficient regularity of $f$, specifically its local H\"older continuity in $v$-variable, to justify $\Delta_v(-\Delta_v)^{-1}f = -f$. Moreover, assuming that $f$ is sufficiently integrable with appropriate polynomial decay at infinity in the $v$-variable, we obtain the field representation:
\begin{align}\label{D potential}
\nabla_v\{(-\Delta_v)^{-1}f\}(v) = \frac{1}{4\pi} \int_{\mathbb{R}^3} \frac{u-v}{|u-v|^3}f(u)\,du.
\end{align}Combining the results above, we define the integral solution to \eqref{VPiL} in the following sense:
\begin{definition}[Weak solution]\label{Weak formulation}
    Let $T>0$ and $m>3$. A function $f\in X_T^m$ is said to be a weak solution to the initial-value problem for \eqref{VPiL} in $\overline{\mathcal{Q}_T}$ if, for any test function $\varphi\in C_c^2([0,T)\times\mathbb{R}^3\times\mathbb{R}^3)$, the following identity holds:
\begin{multline}\label{weak formulation}
\int_0^T \iint_{\mathbb{R}^6} f\left[\partial_t\varphi + v\cdot\nabla_x\varphi + E_f \cdot\nabla_v\varphi + 2\nabla_v\{(-\Delta_v)^{-1}f\}\cdot\nabla_v\varphi\right] \,dxdvds\\
+ \int_0^T\iint_{\mathbb{R}^6}f(-\Delta_v)^{-1}f\cdot \Delta_v\varphi\,dxdvds
=-\iint_{\mathbb{R}^6}f^\textup{in} \varphi(0)\,dxdv,
\end{multline}where we use the explicit formula for $\nabla_v\{(-\Delta_v)^{-1}f\}$ as derived in \eqref{D potential}.
\end{definition}

Then our first main theorem on local existence follows.

\begin{theorem}[Local existence with polynomially decaying initial data]\label{thm.main1}
     We simultaneously consider the plasma and gravitational cases $E_f = -(\nabla K\star_x\rho_f)$, for which $K(x)=\pm\frac{1}{4\pi|x|}$. Let $m>3$, $1<p_0<\frac{3}{2}$ and $3<q_0<\frac{11}{3}$, where $p_0,q_0$ denote the integrability exponent associated with the (weighted) Sobolev space in $X_T^m$. Then there exists $C_0 = C_0(m) > 0$ such that, if the non-negative initial data 
$f^\textup{in} \in  \mathcal{C}^2_{m,0}(\mathbb{R}^6)$ satisfies
\begin{align}\label{smallness cond initial}
    \|f^\textup{in}\|_{\mathcal{C}^2_{m,0}(\mathbb{R}^6)} \le C,
\end{align}
for some $C \in [0,C_0]$, then there exists a time 
$T = T\!\left(\|f^\textup{in}\|_{\mathcal{C}^2_{m,0}},\, m\right) > 0$
for which the initial-value problem \eqref{VPiL} admits a non-negative weak solution $f \in X_T^m$.
\end{theorem}
\begin{remark}\label{thm.main1 maximal existence time}
    An existence time $T$ in Theorem \ref{thm.main1} is determined by Lemma \ref{uniform bound} and Remark \ref{T bd} as 
    \begin{align*}
    T \le 
    \left\{
        \frac{1}{M}
        \log\!\left(
            \frac{1}{
                e^{-e}M
                \|f^{\mathrm{in}}\|_{\mathcal{C}^2_{m,0}}C(m)
            }
        \right)
    \right\}^{\!\frac{1}{3}}
    - 1.
\end{align*}
\end{remark}

Our second main result is on the finite-singularity formation and the collapse of mass in the gravitational case. Hence, we restrict our attention to the gravitational attractive case from now on; namely, we put $K(x) = -\frac{1}{4\pi|x|}$. Before we state the finite-time singularity for \eqref{VPiL}, we define the relevant physical quantities associated with $f$:
\begin{definition}\label{phys quantity}
    Let $0\le f\in L^\infty([0,T]; (L^1_2\cap L^\infty)_{x}(L^1_2\cap L^\infty)_v)$ be a solution to \eqref{VPiL}. Then we define the scalar quantities:
    \begin{align}\label{x moment}
        \textup{I}(t) &\eqdef \frac{1}{2} \iint_{\mathbb{R}^6} |x|^2f(t,x,v)\,dvdx,\\
        \textup{KE}(t) &\eqdef \frac{1}{2} \iint_{\mathbb{R}^6} |v|^2f(t,x,v)\,dvdx,\label{tot kin energy}\\
        \mathcal{E}_E(t) &\eqdef -\frac{1}{2} \int_{\mathbb{R}^3} (K\star_x\rho_f)\rho_f\,dx = \frac{1}{2}\int_{\mathbb{R}^3} |\nabla_xK\star\rho_f|^2\,dx\ge0,\label{tot E energy}
    \end{align}where $K(x) = -\frac{1}{4\pi|x|}$ and $\rho_f(t,x) = \int_{\mathbb{R}^3}f(t,x,v)\,dv$.
\end{definition}

Moreover, we introduce a collection of velocity distribution, where the second spatial moment \eqref{x moment} and its derivatives are well-defined if $0\le f\in X_T^m$ is a weak solution to \eqref{VPiL}.

\begin{definition}\label{adm class}
    We say a velocity distribution function $f$ belongs to an admissible class $A_2$ if $f$ satisfies the following:
    \begin{itemize}
        \item The distributional derivatives $\partial_t f$ and $\nabla_x f$ admit representatives (still denoted by $\partial_t f$ and $\nabla_x f$) that are defined pointwise on $\mathcal Q_T$.
        \item There exist the real-valued dominating functions $g_1\in (L^1_2\cap L_2^\infty)_{x,v}\cap (L^1\cap L^\infty)_xL^1_v$ and $g_2\in (\dot{L}_1^\infty\cap \dot{L}^1_1)_v$ such that
        \begin{align}
            |\partial_tf(t,x,v)| &\le g_1(x,v), \quad \forall (t,x,v)\in\mathcal{Q}_T,\label{Dtf dom}\\
            |\nabla_xf(t,x,v)| &\le g_2(v), \quad\forall (t,x,v)\in\mathcal{Q}_T.\label{Dxf dom}
        \end{align}
        \item For each $t\in[0,T]$,
        \begin{align}
            |v\cdot\nabla_xf(t)| + |\nabla_vf(t)| \in (L^1_2)_{x,v}.\label{vDxf Dvf L1w}
        \end{align}
    \end{itemize}The subscript 2 in $A_2$ refers to the control of moments up to second order. 
\end{definition}

The finite-time singularity formation for \eqref{VPiL} is described in the following theorem.

\begin{theorem}[Finite-time singularity formation for the gravitationally attractive case]
\label{thm.main2}
Consider the gravitational case $E_f = -(\nabla K\star_x\rho_f)$, for which
$$
K(x) = -\frac{1}{4\pi|x|}.
$$
Let $m>6$, and let $f \in X_T^m\cap A_2$ be the nonnegative weak solution to
\eqref{VPiL} and \eqref{smallness cond initial} constructed in
Theorem~\ref{thm.main1}. Assume that the initial datum $f^{\textup{in}}$
satisfies $\textup{I}'(0)<0$ and
\begin{align}\label{inertia < energy}
\mathcal{E}_E(0) - 2\textup{KE}(0)
\ge
C\left(3\frac{\textup{I}(0)}{|\textup{I}'(0)|} +k\right)
\frac{\lVert f^{\textup{in}}\rVert_{L^1_{x,v}}}{\|f^{\textup{in}}\|_{\mathcal{C}^2_{m,0}}},
\end{align}
for some positive constants $C=C(m)$ and $k=k(m)$.
Then the maximal non-negativity time is finite, namely
$$
t_{\max}
\eqdef
\sup\{\tau>0:\ f(\tau,x,v)\ge 0
\ \text{for all }(x,v)\in\mathbb{R}^6\}
<\infty .
$$
In particular, the nonnegative solution constructed in
Theorem~\ref{thm.main1} cannot be extended beyond $t_{\max}$ while
preserving nonnegativity, which signals the onset of singular behavior
in finite time.
\end{theorem}

As a corollary of Theorem~\ref{thm.main2}, we show that the solution
collapses to the spatial origin at time $t=t_{\max}$, provided that the
solution admits a measure-valued continuation in $\mathcal{M}^1_2$.
We first introduce the measure space used in the statement of the
corollary.

\begin{definition}
We say that a Radon measure $\mu$ on $\mathbb{R}^3_x\times\mathbb{R}^3_v$
(with respect to the Lebesgue $\sigma$-algebra) belongs to
$\mathcal{M}^1_2$ if
$$
\int_{\mathbb{R}^3\times\mathbb{R}^3}
(1+|x|^2+|v|^2)\,d\mu <\infty .
$$
\end{definition}

We then obtain the following collapse result.

\begin{corollary}[Collapse of mass]\label{collapse}
Suppose that the nonnegative solution $f$ of \eqref{VPiL} considered in
Theorem~\ref{thm.main2} exists on the time interval $[0,t_{\max})$.
Assume further that the associated finite Radon measure
$t\mapsto g_t\in \mathcal{M}^1_2$ for \eqref{VPiL} is defined on
$[0,t_{\max}]$, that
$$
g_t(dx\,dv)=g(t,x,v)\,dx\,dv
$$
admits a density on $[0,t_{\max})\times\mathbb{R}^6$, and that
$f=g$ on $[0,t_{\max})\times\mathbb{R}^6$.
Then the terminal measure $g_{t_{\max}}$ is singular with respect to the
Lebesgue measure on $\mathbb{R}^6$, and its support satisfies
$$
\operatorname{supp} g_{t_{\max}}
\subset
\{x=0\}\times\mathbb{R}^3_v .
$$
\end{corollary}
\begin{remark}
It is worth emphasizing that the condition
$$
\mathcal{E}_E(0)> 2\textup{KE}(0)
$$
in Theorem~\ref{thm.main2} plays a crucial role in the onset of
finite-time singularity formation. This inequality ensures that the
attractive interaction dominates the kinetic dispersion and thereby
initiates the collapse mechanism identified in our analysis.

Since the singularity criterion requires
$$
0 \le 2\textup{KE}(0) < \mathcal{E}_E(0),
$$
our result applies only to the gravitationally attractive case
$$
E_f = -(\nabla K\star_x\rho_f),
\qquad
K(x)=-\frac{1}{4\pi|x|},
$$
and excludes the plasma case whose potential kernel is
$$
K(x)=\frac{1}{4\pi|x|}.
$$
\end{remark}

\begin{remark}
The concentration point $x=0$ appearing in Corollary~\ref{collapse} is
not intrinsic to the equation \eqref{VPiL}. Rather, it arises from the
specific choice of the test quantity
$$
\textup{I}(t)=\frac{1}{2}\iint_{\mathbb{R}^6}|x|^2 f(t,x,v)\,dx\,dv
$$
and the conditions imposed to guarantee that
$\textup{I}(t_{\max})\le 0$.

More generally, if one instead considers the shifted spatial moment
$$
\textup{I}_{x_0}(t)
\eqdef
\frac{1}{2}\iint_{\mathbb{R}^6}|x-x_0|^2 f(t,x,v)\,dx\,dv
$$
together with the corresponding conditions, then the same argument
yields concentration at the point $x=x_0$ (possibly at a different
maximal time).
\end{remark}


\subsection{Main Strategies}

\subsubsection{Construction of the Approximate Sequence}
Our solution construction is based on the method of continuity applied to a suitably designed approximation scheme~\eqref{semi}. To ensure uniform ellipticity and coefficient regularity necessary for global parabolic estimates, we introduce artificial diffusion and cutoff modifications in both the spatial and velocity variables, together with a time–mollified characteristic function. This enables us to employ Schauder- and $L^p$-type estimates (see Theorems~\ref{Schauder}, \ref{L^p}, and \ref{Weighted mixed L^p estimate}). The corresponding functional framework is the intersection space $Y_{T,\alpha,p_0,q_0,w}$ (defined in Section~\ref{reg space}), which simultaneously provides boundedness, Hölder regularity, and integrability required for the comparison arguments developed later. Detailed steps are given in Section~\ref{Construction of a sequence}.

To control the iterative sequence, we utilize the nonlinear estimate established in Section~\ref{boundedness}. In particular, the classical maximum principle for parabolic equations yields polynomial decay of the solution in both $x$ and $v$. Such decay is crucial to ensure that nonlinear quantities, including $E_f$ and $(-\Delta_v)^{-1}f$, are well-defined within the weak formulation. For example, if $G$ satisfies sufficient polynomial decay, then $E_G$ and $(-\Delta_v)^{-1}G$ remain well-defined and $G(t,\cdot,\cdot)\in L^1_{x,v}$ provided
$$
G \lesssim \langle x\rangle^{-3-\varepsilon}\langle v\rangle^{-3-\varepsilon}.
$$
A challenge arises because the standard weight $\langle x\rangle^{-m}\langle v\rangle^{-m}$ is not invariant under the transport operator $v\cdot\nabla_x$. To overcome this, we employ the transport-invariant weight $\langle x-vt\rangle^{-m}\langle v\rangle^{-m}$. Although this introduces additional terms when differentiating with respect to $v$, the resulting coefficients grow at most polynomially in $T$, which remains manageable within our analysis.

\subsubsection{To Study the Finite-Time Singularity Formation}
We focus only on the gravitational case, in which the attractive interaction may lead the particles to collapse onto a set of measure zero. For the singularity analysis, we employ a mixed-moment method inspired by
\cite{VP_Horst_1982, VMV_Bobylev_1997, VR_Choi_2024}. The central object is the
second spatial moment
$$
\textup{I}(t) \eqdef \frac12 \iint_{\mathbb{R}^6} |x|^2 f(t,x,v)\,dvdx,
$$
whose evolution is governed by the fundamental identity
$$
\textup{I}'(t) = \iint_{\mathbb{R}^6} (x\cdot v) f(t)\,dvdx.
$$
Using the equation and successive differentiations of $\textup{I}(t)$, we establish the
following exact formulas:
\begin{align*}
    \textup{I}''(t) &= 2\mathrm{KE}(t)- \mathcal{E}_E(t),\\
\textup{I}'''(t) &= \frac{d}{dt}(2\mathrm{KE}(t) - \mathcal{E}_E(t))
= 4\iint_{\mathbb{R}^6} (-\Delta_v)^{-1}f\, f\,dvdx,
\end{align*}
where $\mathrm{KE}(t)$ and $\mathcal{E}_E(t)$, defined in \eqref{tot kin energy} and \eqref{tot E energy}, denote the total kinetic energy and gravitational field energy, respectively. Since the right-hand side
is non-negative, we obtain $\textup{I}'''(t)\ge 0$, ensuring that $\textup{I}''(t)$ is
non-decreasing.

To extract a quantitative differential inequality, we estimate the Newtonian
potential term via weighted polynomial decay of $f$, propagated from the local
existence theory. In particular, the bootstrap assumptions imply
$$
\|(-\Delta_v)^{-1} f(t)\|_{L^\infty_{x,v}} \lesssim \|f(t)\|_{L^\infty_{x,v}},
$$
from which we derive
$$
\textup{I}'''(t) \le C \|f(t)\|_{L^\infty_{x,v}} \|f^\textup{in}\|_{L^1_{x,v}}
\lesssim_m \frac{\|f^\textup{in}\|_{L^1_{x,v}}}{\|f^{\mathrm{in}}\|_{\mathcal{C}^2_{m,0}}},
$$
where we have used the uniform bounds on the weighted $C^2$ norms of order $m$ obtained from the
nonlinear estimates in Section \ref{boundedness}.

Combining the above estimates, we show that $\textup{I}(t)$ satisfies
$$
\textup{I}(t) \le \textup{I}(0) + \textup{I}'(0)t + (2\mathrm{KE}(0) - \mathcal{E}_E(0))t^2
+ C \frac{\|f^\textup{in}\|_{L^1_{x,v}}}{\|f^{\mathrm{in}}\|_{\mathcal{C}^2_{m,0}}}t^3.
$$
Thus, $\textup{I}(t)$ is bounded above by a cubic polynomial whose coefficients depend
only on the initial data.

The hypotheses in Theorem~\ref{thm.main2} impose a negative initial mixed moment
$\textup{I}'(0) < 0$, large initial gravitational field energy relative to the kinetic energy $\mathcal{E}_E(0)>2\textup{KE}(0)$; more precisely, we assume the condition
$$
(\mathcal{E}_E(0) - 2\mathrm{KE}(0))\,\|f^{\mathrm{in}}\|_{\mathcal{C}^2_{m,0}}
\gg \|f^{\mathrm{in}}\|_{L^1_{x,v}},
$$
which ensures that the negative quadratic term dominates the small cubic term. As a consequence, the cubic polynomial attains a negative local minimum
at a finite time $t=t_\ast>0$. Since the non-negativity of $f$ enforces
$|x|^2 f \ge 0$ almost everywhere, the spatial moment $\textup{I}(t)$ cannot become
negative without violating positivity of the solution. Therefore, the solution
must lose at least one of the regularity hypotheses before $t_\ast$, proving that the maximal
lifespan is finite.

\subsection{Related Results in the Literature} \subsubsection{On the isotropic Landau equation. }Although the full system \eqref{VPiL} has not yet been addressed in the literature, the \textit{spatially homogeneous} isotropic Landau equation has been investigated in several works. Indeed, the generic non-divergence form of the isotropic Landau equation can be written as
\begin{align*}
    \partial_t f = (-\Delta_v)^{-\frac{\gamma+5}{2}} f\, \Delta_v f + \alpha f(-\Delta_v)^{-\frac{\gamma+3}{2}}f,
\end{align*} for the potential $|\cdot|^{\gamma+2}$ with $\gamma\in [-3,-2]$ and the reaction coefficient with $\alpha>0$ in 3 dimensions.
The global existence of the spatially homogeneous isotropic Landau equation was established by Krieger and Strain \cite{iL_KS_2012}, when $\alpha<\frac{2}{3}$, assuming radially symmetric and non-increasing initial data. In the same year, the same authors, together with Gressman, introduced a sharp nonlocal inequality \cite{iL_KS_2012_2} involving the non-constant diffusion coefficient $(-\Delta)^{-1}f$, which improved the global existence result up to the threshold, $\alpha<\frac{74}{75}$. In 2016, Gualdani and Guillen \cite{iL_Gualdani_2016} proved the existence of smooth solutions for $\alpha=1$, under the assumption of radial symmetry and monotonicity of the initial data. We note that the subsequent works have focused exclusively on the case with unit coefficient, $\alpha=1$. Later, in 2018, Gualdani and Zamponi \cite{iL_Gualdani_2018} established the existence of weak even solutions for even initial data in $L^1$. Regarding the \textit{very soft potential} range of $\gamma\in(-\gamma_*,2] $ with $\gamma_*\approx -2.458$ in 3 dimensions, Gualdani and Guillen in 2022 established $L^p$ propagation and an $L^\infty $ estimate.    Most recently, in 2024, Bowman and Sehyun \cite{iL_Bowman_2024} established the conditionally global existence of smooth solution for large $m\ge1$, without requiring symmetry or monotonicity assumptions.

A summary of these results on the \textbf{spatially homogeneous} equation, highlighting the assumptions on initial data, the range of the parameter $\alpha$ associated with the strength of the reaction term, and the corresponding regularity, is presented in Table \ref{tab:landau-summary}.
\begin{table}[ht!]
\resizebox{\linewidth}{!}{
\begin{tabular}{|c|c|c|c|c|c|}
\hline
\textbf{Authors}                                                                        & \begin{tabular}[c]{@{}c@{}}\textbf{Existence/}\\ \textbf{Uniqueness}\end{tabular} & \textbf{Regularity}                                                       & \textbf{Range of} $\alpha,\gamma$        & \textbf{Initial Data}                                                     & \textbf{Year} \\ \hline
\begin{tabular}[c]{@{}c@{}}\textbf{J. Krieger,}\\ \textbf{R. M. Strain}\cite{iL_KS_2012}\end{tabular}                & $\exists!$                                                    &\begin{tabular}[c]{@{}c@{}} $C([0,\infty); L^1\cap L^{2+}(\mathbb{R}^3))\cap C([0,\infty);H^2(\mathbb{R}^3))$,\\ $\langle v\rangle^{\frac{1}{2}} (-\Delta_v)f\in C([0,\infty); L^{2}(\mathbb{R}^3)) $\end{tabular}                            & $0<\alpha<\frac{2}{3},\ \gamma=-3$  &\begin{tabular}[c]{@{}c@{}}Radial,\\ Non-increasing, \\Non-negative,\\ $f^\textup{in}\in L^1\cap L^{2+}$,\\$\langle v\rangle^{\frac{1}{2}} (-\Delta_v)f^\textup{in}\in L^2$\end{tabular} & '12  \\ \hline
\begin{tabular}[c]{@{}c@{}}\textbf{P. Gressman,}\\ \textbf{J. Krieger,}\\ \textbf{R. M. Strain}\cite{iL_KS_2012_2}\end{tabular} & $\exists!$                                                    & \begin{tabular}[c]{@{}c@{}} $C([0,\infty); L^1\cap L^{2+}(\mathbb{R}^3))\cap C([0,\infty);H^2(\mathbb{R}^3))$,\\ $\langle v\rangle^{\frac{1}{2}} (-\Delta_v)f\in C([0,\infty); L^{2}(\mathbb{R}^3))$ \end{tabular}                                  & $0<\alpha<\frac{74}{75},\ \gamma=-3$ & \begin{tabular}[c]{@{}c@{}}Radial,\\ Non-increasing,\\Non-negative,\\ $f^\textup{in}\in L^1\cap L^{2+}$,\\$\langle v\rangle^{\frac{1}{2}} (-\Delta_v)f^\textup{in}\in L^2$ \end{tabular} & '12  \\ \hline
\begin{tabular}[c]{@{}c@{}}\textbf{M. Gualdani,}\\ \textbf{N. Guillen}\cite{iL_Gualdani_2016}\end{tabular}              & $\exists$                                                     & $C^\infty((0,\infty)\times\mathbb{R}^3)$                       & $\alpha=1,\ \gamma=-3$             & \begin{tabular}[c]{@{}c@{}}Radial,\\ Non-increasing,\\Non-negative,\\  $L^1_2\cap L\log L\cap L^p_{\text{weak}}$,\\ \text{ for some }$p>6$\end{tabular} & '16  \\ \hline
\begin{tabular}[c]{@{}c@{}}\textbf{M. Gualdani,}\\ \textbf{N. Zamponi}\cite{iL_Gualdani_2018}\end{tabular}               & $\exists$                               & \begin{tabular}[c]{@{}c@{}}$\sqrt{f}\in L^2((0,T); H^1(\mathbb{R}^3, \langle v\rangle^{-1}\,dv))$, \\ $f, f\log f\in L^\infty((0,T); L^1(\mathbb{R}^3)),$ \\$f$ even in $v$,\$conditional) $L^\infty([\tau,T]; L^\infty(B_R)),\tau,R>0$\end{tabular}& $\alpha=1,\ \gamma=-3$            & \begin{tabular}[c]{@{}c@{}}Even,\\ Non-negative,\\ $L^1_2\cap L\log L$\end{tabular}            & '18 \\\hline
\begin{tabular}[c]{@{}c@{}}\textbf{M. Gualdani,}\\ \textbf{N. Guillen }\cite{iL_Gualdani_2021}\end{tabular}   
 &                                    & $L^\infty([\tau,T]; L^\infty(B_R))$, $\tau,R>0$ & \begin{tabular}[c]{@{}c@{}}$\alpha=1,\ \gamma\in (\gamma_*,-2]$ \\\text{with } $\gamma_*\approx -2.458$ \end{tabular}            & \begin{tabular}[c]{@{}c@{}}Non-negative\\$L^1_2\cap L^p$,\\\text{with }$p>\frac{3}{\gamma+5}$ \end{tabular}            & '22
  \\ \hline
\begin{tabular}[c]{@{}c@{}}\textbf{D. Bowman,}\\\textbf{S. Ji}\cite{iL_Bowman_2024}\end{tabular}                     & $\exists$                                                     & $C^\infty((0,\infty)\times\mathbb{R}^3)$                          & $\alpha=1,\ \gamma=-3$    & \begin{tabular}[c]{@{}c@{}}Non-negative,\\ $L^\infty\cap L^1_m$,\\ Finite Fisher \\information\end{tabular}                                            & '24  \\ \hline
\end{tabular}
}
\vspace{2mm}
\caption{Summary of existence and regularity results for \textbf{spatially homogeneous} isotropic Landau equations}
\label{tab:landau-summary}
\end{table}

\subsubsection{On the Weighted Elliptic Regularity Theory. }As we can see in Lemma \ref{Holder cont of E_g}, the fact that the sequence element $\{f^n\}$ in the iterative scheme belongs to a weighted $L^p$ space is crucial for establishing the regularity (locally H\"older) of the electric field $E_f$. This property is taken directly from \cite{Ell_Dong_2018}, where the authors proved the global a priori estimate for elliptic and parabolic equations in weighted mixed-norm Sobolev spaces.
Beyond the specific nonlinear framework considered here, substantial progress has also been made on linear degenerate elliptic equations in the setting of weighted regularity theory. Although not directly utilized in our iteration, such results characterize regularizing mechanisms inherent in kinetic models. In particular, the introduction of artificial diffusion in our linearized approximation is consistent with these analytical strategies, as it facilitates the construction of approximate solutions in weighted Sobolev spaces. We refer to \cite{Ell_Fabes_1982, Ell_Trudinger_1971, Ell_Balci_2022, Ell_Bella_2021} for foundational developments on Hölder continuity, notably via De Giorgi–type methods, for degenerate elliptic equations.

\subsubsection{On Singularity Formation in Various Kinetic Equations}
In parallel with advances on local and global existence, several kinetic models are known to develop singularities in finite time under appropriate physical conditions. For the Vlasov–Poisson system, Horst \cite{VP_Horst_1982} established in 1982 that any classical solution in the attractive (gravitational) case with spatial dimension 
$d\ge 4$ must blow up in finite time whenever the initial total energy is negative. For the Vlasov–Manev equation, a kinetic model arising in stellar dynamics, Bobylev et al.~\cite{VMV_Bobylev_1997} proved in 1997 the nonexistence of global classical solutions by employing a moment-based argument. More recently, Choi and Jeong \cite{VR_Choi_2024} showed that solutions to the Vlasov–Riesz equation have finite lifespan when the initial kinetic energy is sufficiently small relative to the initial interaction energy. They further demonstrated finite-time singularity formation for the Vlasov–Riesz–Fokker–Planck equation, provided the initial data are such that the second mixed moment is sufficiently negative in a quantitative sense, as measured in terms of $\textup{I}(0)$ and the initial total energy.

\subsection{Notations}\label{notations}
We now introduce the key notions and notations that will be used throughout the paper.

\begin{itemize}
\item Domain: For each $T>0$, we define the phase space 
\begin{align*}
\mathcal{Q}_T \eqdef (0,T)\times\mathbb{R}^6,
\end{align*}which serves as the domain of the solution to equation \eqref{VPiL}. We also introduce the extended phase space 
\begin{align*}
\Omega_T \eqdef  (-\infty, T)\times\mathbb{R}^6,
\end{align*}which will be used as the domain for the approximate problem described in \eqref{semi}.
\item Parabolic H\"older spaces in $\Omega_T$: Let $0<\alpha<1$. We define the parabolic H\"older seminorm of a function $f$ on $\Omega_T$ by
\begin{align*}
[f]_{C_\textup{para}^\alpha(\Omega_T)}  \eqdef \sup_{\substack{(t,x,v)\neq (s,y,u) \\ (t,x,v), (s,y,u)\in \Omega_T}} \frac{|f(t,x,v) - f(s,y,u)|}{(|t-s|^{1/2} + |x-y| + |v-u|)^\alpha}.
\end{align*}For second-order parabolic regularity, we define
\begin{align*}
[f]_{C_\textup{para}^{2,\alpha}(\Omega_T)}  \eqdef [\partial_tf]_{C_\textup{para}^\alpha(\Omega_T)} + [\nabla_{x,v}^2f]_{C_\textup{para}^\alpha(\Omega_T)},
\end{align*}where $\nabla_{x,v}^2f$ denotes the collection of all second-order derivatives with respect to $x,v$. We say that $f\in C^{2,\alpha}_\textup{para}(\Omega_T)$ if $\lVert f\rVert_{C_\textup{para}^{2,\alpha}(\Omega_T)} <\infty$, 
where the full norm is defined by
\begin{align*}
\lVert f\rVert_{C_\textup{para}^{2,\alpha}(\Omega_T)} &\eqdef \lVert f\rVert_{L^\infty(\Omega_T)} + \lVert \nabla_{x,v} f\rVert_{L^\infty(\Omega_T)} \\
&+\lVert \nabla_{x,v}^2f\rVert_{L^\infty(\Omega_T)} + \lVert \partial_tf \rVert_{L^\infty(\Omega_T)} + [f]_{C_\textup{para}^{2,\alpha}(\Omega_T)}.
\end{align*}
Here, $\nabla_{x,v}$ denotes the gradient with respect to the spatial and velocity variables, and $\nabla_{x,v}^2$ denotes the corresponding collection of second-order derivatives.
\item Muckenhoupt weight class $A_p(\mu, D)$: Let $p\in (1,\infty)$, let $\mu$ be a $\sigma$-finite measure on $\mathbb{R}^n$ and let $D$ be any nonempty open subset of the Euclidean space. A non-negative locally integrable function $w : D \to[0,\infty)$ belongs to the Muckenhoupt $A_p(\mu, D)$ class if
\begin{align*}
\qquad\quad\ [w]_{A_p}  \eqdef \sup_{B_r(x_0)\subset D} \left(\Xint-_{B_r(x_0)} w(x)\,d\mu \right) \left( \Xint-_{B_r(x_0)}(w(x))^{-1/(p-1)}\,d\mu\right)^{p-1}<\infty.
\end{align*}Throughout this paper, if $\mu$ is the Lebesgue measure in $\mathbb{R}^d$, we abbreviate $A_p(\mu, D)$ by $A_p(D)$ for any $1<p<\infty$ and $D\subset \mathbb{R}^d$.
\item Weighted mixed parabolic Sobolev spaces in $\Omega_T$: Let $1<p\le q<\infty$, and consider a product-type weight $w : \Omega_T \to \mathbb{R}_+$ given by 
\begin{align*}
w(t,x,v) = w_1(t)w_2(x,v),
\end{align*}where $w_1\in A_q((-\infty,T))$ and $w_2\in A_p(\mathbb{R}^6)$. Then the associated mixed norm $\|\cdot\|_{L_{p,q,w}(\Omega_T)}$ is defined by
\begin{align*}
\qquad\quad\ \lVert f\rVert_{L_{p,q,w}(\Omega_T)} \eqdef \left( \int_{-\infty}^T w_1(t)\left( \iint_{\mathbb{R}^6} |f(t,x,v)|^p w_2(x,v)\,dxdv\right)^{\frac{q}{p}}\,dt\right)^{\frac{1}{q}}.
\end{align*}
Then, we define the weighted mixed parabolic Sobolev space $W_{p,q,w}^{1,2}(\Omega_T)$ to consist of functions $f$ such that
\begin{align*}
\partial_tf,\,\, \nabla_{x,v}f, \,\, \nabla_{x,v}^2f \in L^1_\textup{loc}(\Omega_T), \quad \textup{ and }\quad \| f\|_{W_{p,q,w}^{1,2}(\Omega_T)}<\infty,
\end{align*}where $\nabla_{x,v}$ and $\nabla_{x,v}^2$ denote the gradient and Hessian, respectively, with respect to $(x,v)$ and the norm in $W_{p,q,w}^{1,2}(\Omega_T)$ is given by
\begin{align}\label{def.para.sob}
\qquad\qquad\lVert f\rVert_{W_{p,q,w}^{1,2}(\Omega_T)} &\eqdef \lVert f\rVert_{L_{p,q,w}(\Omega_T)} + \lVert \partial_tf\rVert_{L_{p,q,w}(\Omega_T)} + \lVert \nabla_{x,v}f\rVert_{L_{p,q,w}(\Omega_T)} + \lVert \nabla_{x,v}^2f\rVert_{L_{p,q,w}(\Omega_T)}.
\end{align}Finally, we define the non-weighted Sobolev space by
\begin{align*}
W_{p,q}^{1,2}(\Omega_T) \eqdef W_{p,q,1}^{1,2}(\Omega_T),\ \ W_p^{1,2}(\Omega_T) \eqdef W_{p,p,1}^{1,2}(\Omega_T),
\end{align*}where the weight is identically one, and the mixed norm reduces to the standard $L^p$-based parabolic Sobolev norm.
\item $L^\infty_{m,2m}$ and $\mathcal{C}^2_{m,0}$ spaces: For $m\in\mathbb{R}$, we define the weighted $L^\infty$ space by
\begin{align*}
    L^\infty_{m,2m}(\mathbb{R}^6) \eqdef \biggl\{f : \mathbb{R}^6\to\mathbb{R} \bigg|\ \sup_{(x,v)\in\mathbb{R}^6}\langle x\rangle^m \langle v\rangle^{2m}|f(x,v)|<\infty\biggr\},
\end{align*}where $\langle z\rangle \eqdef (1+|z|^2)^{1/2}$ for $z\in\mathbb{R}^3$. The associated norm is given by
\begin{align*}
    \lVert f\rVert_{L^\infty_{m,2m}(\mathbb{R}^6)} \eqdef \sup_{(x,v)\in\mathbb{R}^6} \langle x\rangle^m\langle v\rangle^{2m} |f(x,v)|.
\end{align*}In accordance with this space, we define the initial space, denoted by $\mathcal{C}^2_{m,0}(\mathbb{R}^6)$ as the set of functions $f\in C^2(\mathbb{R}^6)$ such that
\begin{align*}
    f,v\cdot\nabla_xf, \nabla_vf, \Delta_xf, \Delta_vf\in L^\infty_{m,2m}(\mathbb{R}^6),
\end{align*}and
\begin{align*}
    \langle x\rangle^m\langle v\rangle^{2m}\big(|f|+|v\cdot\nabla_xf| + |\nabla_vf| + |\Delta_xf| + |\Delta_vf|\big) \to0 \quad\text{as }|x|,|v|\to\infty,
\end{align*}endowed with a norm
\begin{align*}
    \|f\|_{\mathcal{C}^2_{m,0}(\mathbb{R}^6)} \eqdef \|(|f| +|v\cdot\nabla_xf|+ |\nabla_vf| + |\Delta_xf| + |\Delta_vf|)\|_{L^\infty_{m,2m}(\mathbb{R}^6)}. 
\end{align*}
\item Solution space $X_T^m$: For $m\in\mathbb{R}$, $1<p_0<\frac{3}{2}$, $3<q_0<\frac{11}{3}$ and $w=\langle x\rangle\langle v\rangle^8$, we define the solution space 
    \begin{align*}
            X_T^m  \eqdef \left\{f \in W_{p_0}^{1,2}(\mathcal{Q}_T)\cup W_{q_0,w}^{1,2}(\mathcal{Q}_T) \ :\  \sup_{(t,x,v)\in \overline{\mathcal{Q}_T}} \langle x-vt\rangle^m\langle v\rangle^m|f(t,x,v)|<\infty\right\},
    \end{align*}endowed with a norm
    \begin{align*}
            \lVert f\rVert_{X_T^m} \eqdef \lVert \langle x-vt\rangle^m\langle v\rangle^m f(t,x,v)\rVert_{L^\infty(\overline{\mathcal{Q}_T})},
    \end{align*}where $W_{p_0}^{1,2}$ and $W_{q_0,w}^{1,2}$ are the (weighted) parabolic Sobolev spaces defined as in \eqref{def.para.sob}.
\item \label{reg space}Regular space $Y_{T,\alpha,p_0,q_0,w}$: For $0<\alpha< 1$, $1<p_0<\frac{3}{2}$ and $3<q_0<\frac{11}{3} $, we define the regular space
\begin{align*}
Y_{T,\alpha,p_0,q_0,w} \eqdef (C_\textup{para}^{2,\alpha} \cap W_{p_0}^{1,2} \cap W^{1,2}_{q_0, q_0, w})(\mathcal{Q}_T),
\end{align*}where the weight function is given by
\begin{align*}
w(x,v) \eqdef \langle x\rangle \langle v\rangle^8.
\end{align*}
\item Image space $Y'_{T,\alpha,p_0,q_0,w}$: For $\alpha$, $p_0,q_0$ and $w$ as above, we define
\begin{align*}
Y'_{T,\alpha,p_0,q_0,w} \eqdef (C_\textup{para}^\alpha \cap L_{p_0} \cap L_{q_0, q_0, w})(\mathcal{Q}_T).
\end{align*}
\item Moment of inertia of the system $\textup{I}(t)$ is defined by
\begin{align*}
\textup{I}(t) \eqdef \frac{1}{2}\iint_{\mathbb{R}^3\times\mathbb{R}^3} |x|^2f(t,x,v)\,dxdv.
\end{align*}
\item Kinetic energy of the system $\textup{KE}(t)$ is defined by
\begin{align}\label{kinetic energy}
\textup{KE}(t) \eqdef \frac{1}{2}\iint_{\mathbb{R}^3\times\mathbb{R}^3} |v|^2f(t,x,v)\,dxdv.
\end{align}
\item Absolute value of the gravitational field energy $\mathcal{E}_E(t)$ is defined by
\begin{align*}
    \mathcal{E}_E(t) \eqdef \frac{1}{2}\int_{\mathbb{R}^3}|E_f(t,x)|^2\,dx,
\end{align*}where $E_f = -(\nabla K\star_x\rho_f)$, $K(x)=-\frac{1}{4\pi|x|}$.
\item The Japanese brackets are defined as $\langle z\rangle = (1+|z|^2)^{1/2}$.
\item For $1\le p\le \infty$, we denote by $L^p_k(\mathbb{R}^d)$ the weighted $L^p$-space with polynomial weight of order $k\in\mathbb{R}$. It consists of all measurable functions $f : \mathbb{R}^d\to\overline{\mathbb{R}}$ such that
\begin{align*}
    \|f\|^p_{L^p_k(\mathbb{R}^d)} \eqdef \int_{\mathbb{R}^d} \langle x\rangle^k |f(x)|^p\,dx < \infty,
\end{align*}with the usual modification when $p=\infty$. The corresponding homogeneous weighted space is denoted by $\dot{L}^p_k(\mathbb{R}^d)$. It is defined as the set of measurable functions $f : \mathbb{R}^d\to\overline{\mathbb{R}}$ satisfying
\begin{align*}
    \|f\|^p_{\dot{L}^p_k(\mathbb{R}^d)} \eqdef \int_{\mathbb{R}^d} |x|^k|f(x)|^p\,dx <\infty,
\end{align*}again with the usual modification when $p=\infty$.
\item For any open set $\Omega$, $C_c^k(\Omega)$ is denoted by the space of the real-valued $C^k$ functions which are compactly supported in $\Omega$. 
\end{itemize}

\begin{remark}In the definition of the regular space $Y_{T,\alpha,p_0,q_0,w}$, we note that the weight function $w(x,v) = \langle x\rangle\langle v\rangle^8\in A_{q_0}(dxdv, \mathbb{R}^6)$ since both $\langle x\rangle \in A_{q_0}(dx, \mathbb{R}^3)$ and $\langle v\rangle^8 \in A_{q_0}(dv, \mathbb{R}^3)$ hold for $3<q_0<\frac{11}{3}$. The interpolation theorem for Muckenhoupt weights can be found in Chapter 7 of \cite{FA_Grafakos_2014}. The fact that the Japanese bracket function $\langle \cdot \rangle^l$ lies in certain $A_p$ classes is discussed in detail in \cite{FA_Duarte_2023}.
\end{remark}

\subsection{Outline of the Paper} The remainder of this paper is organized as follows. In Section \ref{Construction of a sequence}, after establishing the solvability of the heat equation in the intersection of Hölder and $L^p$ spaces, we construct a sequence of approximate solutions $\{f^n\}_{n\ge 0}$ to the linearized problem \eqref{n approx}, based on the parabolic theory developed in Theorems \ref{Schauder}, \ref{L^p},  and \ref{Weighted mixed L^p estimate}. Section \ref{boundedness} proves that this approximating sequence satisfies uniform polynomial decay in both the $x-tv$ and $v$ directions. Passing to the limit, we show that the limiting function belongs to the solution space $X_T^m$, is non-negative, and satisfies the weak formulation stated in Definition \ref{Weak formulation}. Finally, Section \ref{Singularity formation} establishes finite-time singularity formation for solutions in $X_T^m$ to the initial-value problem for \eqref{VPiL} in the gravitational (attractive) case, where $E_f=-(\nabla K\star\rho)$ with $K(x)=-\frac{1}{4\pi|x|}$, The proof relies on showing that the total moment of inertia of the system $\frac{1}{2}\iint_{\mathbb{R}^6} |x|^2 f\, dv dx$ becomes strictly negative in finite time.

\section{Solution to a Linear Problem}
\label{Construction of a sequence}
In this section, we construct a family of linearly approximate solutions to the initial-value problem for \eqref{VPiL}. For each small $\varepsilon>0$ and, given $g\in Y_{T,\alpha,p_0,q_0,w}$ with $g\ge0$, in $\Omega_T$, we define a related linear operator $\mathcal{L}_{\varepsilon,g}: Y_{T,\alpha,p_0,q_0,w}\rightarrow Y'_{T,\alpha,p_0,q_0,w}$ given by
\begin{align}\label{def.Leg}
\mathcal{L}_{\varepsilon,g}(f) \eqdef  \partial_t f - \varepsilon \Delta_{x,v}f + \chi(\varepsilon v)v\cdot\nabla_x f + J^{\varepsilon}(t)\chi(\varepsilon x) E_g\cdot\nabla_v f - (-\Delta_v)_\varepsilon^{-1}g \Delta_vf - gf,
\end{align}where we define the localized diffusion coefficient
\begin{align*}
    (-\Delta_v)^{-1}_\varepsilon g(v) \eqdef \frac{1}{4\pi}\int_{\mathbb{R}^3} \frac{\chi(\varepsilon u)g(u)}{|u-v|}\,du,
\end{align*}where the function $\chi\in C_c^2(\mathbb{R}^3)$ is a smooth cut-off function that satisfies
\begin{align}\label{def cutoff}
0\le \chi\le 1,\quad \chi\equiv1\,\,\textup{ in }\,\,B_{\frac{1}{2}}(0), \quad \textup{ and }\quad \textup{supp }\chi=B_1(0),
\end{align} and $J^\varepsilon$ is a time-regularized characteristic function defined via the following mollification:
\begin{align}\label{reg interval}
J^\varepsilon(t) &\eqdef \left(1_{\left[-\frac{\varepsilon}{2},T-\frac{\varepsilon}{2}\right]}\star \rho_\varepsilon\right)(t),
\end{align}with a mollifier $\rho_\varepsilon$ supported on $\left[-\frac{\varepsilon}{2}, \frac{\varepsilon}{2}\right]$.

Using this operator, we define the following linear approximate equation for \eqref{VPiL} for each given non-negative $g\in Y_{T,\alpha,p_0,q_0,w}$ as
\begin{align}\label{semi}
\mathcal{L}_{\varepsilon,g}(f) = \mathcal{L}_{\varepsilon,g}(f^\textup{in}_\varepsilon)I^\varepsilon,
\end{align}where
\begin{equation*}
    \mathcal{L}_{\varepsilon,g}(f^\textup{in}_\varepsilon)\eqdef 
     (-\varepsilon\Delta_{x,v}f^\textup{in}_\varepsilon +\chi(\varepsilon v)v\cdot\nabla_xf^\textup{in}_\varepsilon + J^\varepsilon(t)\chi(\varepsilon x)E_g \cdot\nabla_vf^\textup{in}_\varepsilon - (-\Delta_v)_\varepsilon^{-1}g\Delta_vf^\textup{in}_\varepsilon - gf^\textup{in}_\varepsilon),
\end{equation*}
\begin{align}\label{approx initial}
    \{f^\textup{in}_\varepsilon\ge0\}_{\varepsilon>0}\subset C^\infty_c(\mathbb{R}^6)\textup{ s.t. } f^\textup{in}_\varepsilon \to f^\textup{in} \textup{ in }\mathcal{C}^2_{m,0}(\mathbb{R}^6),
\end{align}where the approximates are just obtained from multiplying $f^\textup{in}$ by a $6$-dimensional smooth cut-off. Here, another time-regularized characteristic function near $t=0$, $I^\varepsilon$ is defined by
\begin{align}\label{I def}
    I^\varepsilon(t) \eqdef (1_{\left[-\frac{\varepsilon^\kappa}{2}, \frac{3\varepsilon^\kappa}{2}\right]}\star \rho_{\varepsilon^\kappa})(t),
\end{align}where $0<\kappa \eqdef \min\{\kappa_1, \kappa_2\}$ with
\begin{align*}
    \kappa_1/p_0 \eqdef &\textup{ the maximum of the polynomial growth order in }\lambda^{-1} \textup{ and } (\Lambda/\lambda) \\
    &\textup{ arising from the global estimate } W_{p_0}^{1,2}(\Omega_T)\to L_{p_0}(\Omega_T),\\
    \kappa_2/q_0 \eqdef &\textup{ the maximum of the polynomial growth order in }\lambda^{-1} \textup{ and } (\Lambda/\lambda) \\
    &\textup{ arising from the global estimate } W_{q_0,w}^{1,2}(\Omega_T)\to L_{q_0,w}(\Omega_T).
\end{align*}

Based on linear analysis, we will introduce a sequence of iterated solutions $\{f^n\}_{n\ge 0}$ at the end of the section. Let $\varepsilon>0$ and $T>0$ be given. Let a non-negative function $g\in Y_{T,\alpha,p_0,q_0,w}$, satisfying the following pointwise bound, be given:
\begin{align}\label{bound g}
0\le g\le C_g \langle x-\chi(\varepsilon'v)vt\rangle^{-m} \langle v\rangle^{-m},
\end{align}for some constants $0<C_g$, $0<\varepsilon\le \varepsilon'\le 1$ and exponent $m>3$. This decay estimate will be further established in Section \ref{boundedness}. We first establish boundedness and continuity of the coefficients  in \eqref{semi} for each given $g$.

The following lemma establishes $L^\infty$-boundedness of the coefficients.
\begin{lemma}[$L^\infty$ boundedness of coefficients]
    \label{L^infty bound of coefficients}$\chi(\varepsilon v)v$, $E_g$ and $(-\Delta_v)_\varepsilon^{-1}g$ are essentially bounded with their respective domains.
\end{lemma}

\begin{proof}
The function $\chi(\varepsilon v)v$ is bounded by the definition of $\chi$ as given in \eqref{def cutoff}.

By the upper bound \eqref{bound g} of $g$, the electric field term $E_g(t,x)$ satisfies the estimate
\begin{align*}
|E_g(t,x)| \le C_g\int_{\mathbb{R}^3} \frac{dv}{\langle v\rangle^m}\int_{\mathbb{R}^3} \frac{dy}{|y|^2\langle x-y-\chi(\varepsilon'v)vt\rangle^m} <\infty,
\end{align*}where we used 
\begin{equation}\label{int decay}
    \int_{\mathbb{R}^3}\frac{\langle z\rangle^{-m}}{|x-z|^k}dz \lesssim_{m,k} \frac{1}{\langle x\rangle^k}, \text{ for }m>3\text{ and }0<k<3.
\end{equation}
Finally, for the potential term $(-\Delta_v)_\varepsilon^{-1}g$, we use the upper bound \eqref{bound g} of $g$ and \eqref{int decay} again to obtain
\begin{align*}
|(-\Delta_v)_\varepsilon^{-1}g(t,x,v)| \le \frac{C_g}{4\pi} \int_{\mathbb{R}^3} \frac{du}{|u-v|\langle u\rangle^m}<\infty.
\end{align*}
\end{proof}

The following two lemmas prove the H\"older continuity of the coefficients in \eqref{semi}.

\begin{lemma}[H\"older continuity of $\chi(\varepsilon v)v$ and $(-\Delta_v)_\varepsilon^{-1}g$] \label{lem.Hold.cont}
Let $1\le p_0<\infty$ and let $0<\alpha<1$. Suppose that $g\in (W_{p_0}^{1,2}\cap C_\textup{para}^{2,\alpha})(\mathcal{Q}_T)$ and satisfies \eqref{bound g}. Then the coefficients $\chi(\varepsilon v)v$ and $(-\Delta_v)_\varepsilon^{-1}g$ are H\"older continuous in $(x,v)$ with exponent $\alpha$ and in $t$ with exponent $\alpha/2$. 
In particular, both functions are parabolically H\"older continuous.
\end{lemma}

\begin{proof}
Since $\chi(\varepsilon v)v$ is compactly supported in $v$, it suffices to show that it has a bounded derivative. Indeed, we observe that
\begin{align*}
\nabla_v(\chi(\varepsilon v)v) = \varepsilon(\nabla_v\chi)(\varepsilon v)v + \chi(\varepsilon v) \cdot \textup{Id}.
\end{align*}This yields the uniform bound:
\begin{align*}
\lVert \nabla_v(\chi(\varepsilon v)v)\rVert_{L^\infty_v} \le \varepsilon \lVert \nabla\chi\rVert_{L^\infty} \lVert v\rVert_{L^\infty(\textup{supp}(\chi(\varepsilon v)))} + 3\lVert \chi\rVert_{L^\infty} <\infty.
\end{align*}Thus, $\chi(\varepsilon v)v$ is Lipschitz (and therefore H\"older) continuous in $v$, and trivially in $x$ and $t$ since it depends only on $v$.

We now turn to the H\"older continuity of $(-\Delta_v)_\varepsilon^{-1}g$. First, fix $(x,v)\in\mathbb{R}^3\times\mathbb{R}^3$ and consider $t_1,t_2\in [0, T]$. Since $g\in (W_{p_0}^{1,2}\cap C_\textup{para}^{2,\alpha})(\mathcal{Q}_T)$, using Lemma \ref{Linfty.est.sing.int} and the representation formula for the Newtonian potential, we have
\begin{align*}
|(-\Delta_v)_\varepsilon^{-1}g(t_1,x,v) - (-\Delta_v)_\varepsilon^{-1}g(t_2,x,v)| &\lesssim \varepsilon^{-3}\lVert g(t_1,x,\cdot) - g(t_2,x,\cdot)\rVert_{L^\infty_v}\\
&\lesssim \varepsilon^{-3}[g]_{C_\textup{para}^{\alpha}}|t_1 - t_2|^{\alpha/2}.
\end{align*}Similarly, for fixed $(t,v)\in[0,T]\times\mathbb{R}^3$ and $x_1, x_2\in \mathbb{R}^3$, we have
\begin{align*}
|(-\Delta_v)_\varepsilon^{-1}g(t,x_1,v) - (-\Delta_v)_\varepsilon^{-1}g(t,x_2,v)| & \lesssim \varepsilon^{-3}\lVert g(t,x_1,\cdot) - g(t,x_2,\cdot)\rVert_{L^\infty_v}\\
&\lesssim  \varepsilon^{-3}[g]_{C_\textup{para}^{\alpha}} |x_1 - x_2|^\alpha.
\end{align*}Finally, for fixed $(t,x)\in[0,T]\times\mathbb{R}^3$ and $v_1, v_2\in\mathbb{R}^3$, we estimate
\begin{align*}
|(-\Delta_v)_\varepsilon^{-1}g(t,x,v_1) - (-\Delta_v)_\varepsilon^{-1}g(t,x,v_2)| &\lesssim \varepsilon^{-3}\lVert g(t,x,\cdot+v_1) - g(t,x,\cdot+v_2)\rVert_{L^\infty_v}\\
&\lesssim \varepsilon^{-3}[g]_{C_\textup{para}^{\alpha}}|v_1 - v_2|^\alpha.
\end{align*}Then, recall that for any $\delta>0$, the parabolic H\"older seminorm satisfies the interpolation-type estimate
\begin{align*}
[g]_{C_\textup{para}^\alpha(\mathcal{Q}_T)} \le \delta [g]_{C_\textup{para}^{2,\alpha}(\mathcal{Q}_T)} + C\delta^{-\alpha/2}\lVert g\rVert_{L^\infty(\mathcal{Q}_T)}<\infty,
\end{align*}where the constant $C$ depends only on the spatial dimension. See Chapter 8 of \cite{Ell_Krylov_1996} for details.
\end{proof}

\begin{lemma}[H\"older continuity of $E_g$]\label{Holder cont of E_g} Let
\begin{equation}\label{Y comp condi}
    \begin{gathered}
        1<p_0<\frac{3}{2}, \quad 3<q_0<\frac{11}{3}, \quad0<\alpha<1,\text{ and }
        w =w(x,v)=\langle x\rangle\langle v\rangle^8.
    \end{gathered}
\end{equation}Suppose that $g\in Y_{T,\alpha,p_0,q_0,w}$ and satisfies \eqref{bound g}. Then the coefficient $J^\varepsilon(t)\chi(\varepsilon x) E_g(t,x)$ is H\"older continuous in $x$ with exponent $\alpha$ and in $t$ with exponent $\alpha/2$. 
In particular, it is parabolically H\"older continuous.
\end{lemma}

\begin{proof}
\textbf{Step 1: H\"older continuity in space.}
Since $J^\varepsilon(t)\chi(\varepsilon x)$ is smooth with compact support, it suffices to establish the local H\"older continuity of $E_g(t,x)$ in $x$. We begin with showing the Lipschitz continuity of $E_g(t,x)$ uniformly in $t\in[0,T]$.

Let $x_1,x_2\in\mathbb{R}^3$, fix $t\in[0,T]$, and define $x_0 \eqdef x_1 + \tau(x_2 - x_1)$ for $\tau\in[0,1]$. Then by the fundamental theorem of calculus, we can write the difference as
\begin{align*}
|E_g&(t,x_1) - E_g(t,x_2)| \nonumber
= \left|\frac{1}{4\pi}\int_0^1\int_{\mathbb{R}^3} \int_{\mathbb{R}^3} (\nabla_xg)(t,x_0 - y,v)\cdot(x_2-x_1) \nabla_y\left(\frac{1}{|y|} \right)\,dydvd\tau\right|.
\end{align*}We perform integration by parts with respect to $y$, using the identity:
\begin{multline*}
\int_{|y|>\varepsilon} [(\nabla_xg)(t,x_0 - y, v)\cdot(x_2-x_1)]\nabla_y\left(\frac{1}{|y|}\right)\,dy \\
= -\int_{|y|=\varepsilon} g(t,x_0-y,v) \left(-\frac{y}{|y|^3}\right)(x_2 - x_1) \cdot \frac{y}{|y|}dS_y 
+ \int_{|y|>\varepsilon} g(t,x_0-y,v)(x_2 - x_1)\Delta_y\left(\frac{1}{|y|}\right)\,dy.
\end{multline*}Noting that $\Delta_y \left(\frac{1}{|y|}\right) =0$ for $y\neq0$, the second integral vanishes. The first boundary integral term becomes:
\begin{align*}
\int_{\mathbb{S}^2}g(t,x_0-\varepsilon \omega,v) \omega [(x_2 - x_1)\cdot \omega]d\omega\to \frac{4\pi}{3}(x_2-x_1)g(t,x_0,v) \quad \textup{ as }\varepsilon\to0^+ 
\end{align*}a.e. in $\mathcal{Q}_T$. Thus, we obtain
\begin{align*}
|E_g(t,x_1) - E_g(t,x_2)|&\le \frac{1}{3}|x_2 - x_1| \int_0^1 \int_{\mathbb{R}^3} g(t,x_1 + \tau(x_2-x_1), v)\,dvd\tau.
\end{align*}Since $g(t,x,v)\le C_g\langle v\rangle^{-m}$ for some $m>3$ by \eqref{bound g}, we have
\begin{align*}
|E_g(t,x_1) - E_g(t,x_2)| \le C(m,g)|x_1 - x_2|.
\end{align*}Hence, $E_g$ is Lipschitz continuous with respect to $x$, uniformly in $t$.

\textbf{Step 2: H\"older continuity in time.}
Let $t_1<t_2\in[0,T]$ and fix $x\in\mathbb{R}^3$. Then
\begin{align*}
|E_g(t_2,x) - E_g(t_1,x)| \le \frac{1}{4\pi} \int_{t_1}^{t_2}\int_{\mathbb{R}^3} \int_{\mathbb{R}^3} |\partial_tg(\tau, y,v)|\frac{ 1}{|x-y|^2}\,dvdyd\tau.
\end{align*}To estimate this, apply H\"older's inequality in $(y,v)$ with exponents $q_0>3$, $q_0'$ such that $1/q_0 + 1/q_0'=1$. Then
\begin{multline*}
|E_g(t_2,x) - E_g(t_1,x)|\\ \le \frac{1}{4\pi}\int_{t_1}^{t_2} \left(\int_{\mathbb{R}^3\times\mathbb{R}^3} |\partial_tg(\tau,y,v)|^{q_0} \langle y\rangle \langle v\rangle^8 dvdy\right)^{\frac{1}{q_0}} 
\left(\int_{\mathbb{R}^3\times\mathbb{R}^3} \frac{1}{|x-y|^{2q_0'}\langle y\rangle^{\frac{q_0'}{q_0}}} \frac{1}{\langle v\rangle^{\frac{8q_0'}{q_0}}}\,dvdy\right)^{\frac{1}{q_0'}}\,d\tau\\
=: \frac{1}{4\pi}\int_{t_1}^{t_2} \left(\textup{I}\right)^{\frac{1}{q_0}} \left(\textup{II}\right)^{\frac{1}{q'_0}} d\tau,
\end{multline*}for appropriate weights $\langle x\rangle\langle v\rangle^8$ ensuring the integrability of the second factor (e.g., using Muckenhoupt weight theory). Under the conditions:
\begin{gather}
2q_0'<3,\tag{for local integrability of $\textup{II}$ near $y=x$}\\
\left(2+ \frac{1}{q_0} \right)q'_0 > 3, \,\frac{8}{q_0}q'_0 >3, \tag{for integrability of $\textup{II}$ at infinity}\\
1+8<6(q_0-1), \tag{Muckenhoupt $A_{q_0}$ weight condition for $\langle y\rangle \langle v\rangle^8 $}
\end{gather}the second integral is uniformly bounded in $x$. If $q_0\in (3,11/3)$, all the three conditions above are satisfied. Thus, with $w=\langle y\rangle \langle v\rangle^8 ,$ we have
\begin{align*}
|E_g(t_2,x) - E_g(t_1,x)| &\lesssim \int_{t_1}^{t_2} \lVert \partial_\tau g(\tau)\rVert_{L^{q_0}_w}\,d\tau \lesssim |t_2-t_1|^{1-1/q_0} \lVert \partial_tg\rVert_{L^{q_0}_w(\mathcal{Q}_T)}.
\end{align*}Then we conclude
\begin{multline*}
|J^\varepsilon(t_2)E_g(t_2,x) - J^\varepsilon(t_1)E_g(t_1,x)| \\
\le |J^\varepsilon(t_2)| |E_g(t_2,x) - E_g(t_1,x)| + |J^\varepsilon(t_2) - J^\varepsilon(t_1)| |E_g(t_1,x)|\\
\lesssim_{g,\varepsilon} |t_2 - t_1|^{\alpha/2} (T+1)^{1-1/q_0-\alpha/2} + |t_2 - t_1|^{\alpha/2} \lesssim |t_2-t_1|^{\alpha/2}.
\end{multline*}Here, we used the fact that $1-1/q_0-\alpha/2>0$ for all $\alpha\in(0,1)$, which follows from the assumption $q_0\in(3,11/3)$ in \eqref{Y comp condi}.

\textbf{Step 3: Conclusion.}
Combining both parts, we have
\begin{itemize}
\item $J^\varepsilon(t)\chi(\varepsilon x)E_g(t,x)$ is H\"older continuous in $x$ with exponent $\alpha$,
\item and is H\"older continuous in $t$ with exponent $\alpha/2$.
\end{itemize}Thus, we have
\begin{align*}
|J^\varepsilon(t_1)\chi(\varepsilon x_1)E_g(t_1,x_1) - J^\varepsilon(t_2)\chi(\varepsilon x_2)E_g(t_2,x_2)| &\lesssim |x_1 - x_2|^\alpha + |t_1 - t_2|^{\alpha/2}\\
    &\lesssim (|x_1 - x_2| + |t_1 - t_2|^{1/2})^\alpha
\end{align*}for all $t_1,t_2\in[0,T]$ and $x_1,x_2\in\mathbb{R}^3$; that is, $J^\varepsilon(t)\chi(\varepsilon x)E_g(t,x)$ is parabolically H\"older continuous.
\end{proof}

Now we have sufficient regularity of the coefficients of the linear operator $\mathcal{L}_{\varepsilon, g}$ to solve the corresponding equation between the spaces $Y_{T,\alpha,p_0,q_0,w}$ and $Y'_{T,\alpha,p_0,q_0,w}$, as stated in Lemma~\ref{Linear existence}. 
Before proceeding further, we first recall the solvability of the heat equation between $Y_{T,\alpha,p_0,q_0,w}$ and $Y'_{T,\alpha,p_0,q_0,w}$.

\begin{lemma}[Solvability of the heat operator]\label{Heat_solvability}Let $\alpha,p_0,q_0,w$ be defined as \eqref{Y comp condi}. For each $\lambda>0$, the heat operator $(\partial_t-\Delta_{x,v}+\lambda) : Y_{T,\alpha,p_0,q_0,w} \to Y'_{T,\alpha,p_0,q_0,w}$ is uniquely solvable.
\end{lemma}
\begin{proof}
Given $g\in Y'_{T,\alpha,p_0,q_0,w}$, define 
$
u= G_\lambda\star_{(t,x,v)}g,
$ where
\begin{align*}
G_\lambda(t,x,v) \eqdef \frac{1_{t>0}}{(4\pi t)^{3}} \exp\left(-\frac{|x|^2+|v|^2}{4t} - \lambda t\right),
\end{align*}is the fundamental solution to the heat equation 
$
\partial_tu - \Delta_{x,v}u +\lambda u=0,$ and $\star_{(t,x,v)}$ denotes convolution over $\mathbb{R}^7$ in the $(t,x,v)$ variables. By Lemma 5.1.1(i) in \cite{Ell_Krylov_2008} and its proof, the function $u$ satisfies:
\begin{enumerate}[label=(\roman*)]
\item $u\in W_{p_0}^{1,2}(\Omega_T)$,
\item $\partial_tu -\Delta_{x,v}u + \lambda u = g$ a.e. in $\Omega_T$.
\end{enumerate}Moreover, by Theorem 2.2(B) in \cite{Ell_Ping_2017}, we also have
$
u\in W_{q_0,w}^{1,2}(\Omega_T),
$ where the weight class used in \cite{Ell_Ping_2017} is a parabolically modified Muckenhoupt class (based on the distance $d((t,x), (s,y))=\max\{|t-s|^{1/2}, |x-y|\}$). Since our weight $w(x,v) = \langle x\rangle\langle v\rangle^8$ is independent of time and it is in the Muckenhoupt class using the uniform distance $d((t,x), (s,y)) = \max\{|t-s|, |x-y|\}$, it also belongs to this modified Muckenhoupt class.

From the precious steps, we know all the integrability of the function $u$ and its weak derivatives. To prove that $u\in C^{2,\alpha}_\textup{para}(\Omega_T)$, which is the final ingredient needed to establish the surjectivity of the operator $\partial_t-\Delta_{x,v} + \lambda$, it suffices to show $u$ has parabolic H\"older regularity. Injectivity follows from standard \textit{a priori} estimates, such as those in Theorems \ref{Schauder}-\ref{Weighted mixed L^p estimate}.

To establish H\"older continuity, we first estimate the parabolic H\"older seminorm of $u$. Fix a spatial unit vector $e_1$. Then for small $h\in\mathbb{R}$, using the H\"older continuity of $g$, we have
\begin{multline*}
|u(t,x+he_1,v) - u(t,x,v)| \\
\le \int_{-\infty}^T\int_{\mathbb{R}^6} G_\lambda(s,y,u)|g(t-s,x-y+he_1,u-v) - g(t-s,x-y,u-v)|\,dudyds\\
\le [g]_{C_\textup{para}^\alpha(\Omega_T)}|h|^\alpha \int_{-\infty}^T \int_{\mathbb{R}^6} G_\lambda(s,y,u)\,dudyds
\le \frac{[g]_{C_\textup{para}^\alpha(\Omega_T)}}{\lambda}|h|^\alpha.
\end{multline*}Similar estimates in time and space directions show that $u\in C^\alpha_\textup{para}(\Omega_T)$. To gain the regularity $C^{2,\alpha}_{\textup{para}}$, define mollified functions $u_n\eqdef \xi_n\star_{(t,x,v)}u$, where
\begin{align*}
\xi_n(t,x,v) = n^7\xi(nt,nx,nv),
\end{align*}where $\xi\in C^\infty_c(\mathbb{R}^7)$, $\textup{supp }\xi=\{(t,x,v) : |t|^2 + |x|^2 + |v|^2 \le 1\}= B_1$, $\xi\equiv1$ on $B_{1/2}$, and $0\le\xi\le1$. Then, since $u_n\in C^\infty(\Omega_T)$, the interior Schauder estimate (Theorem 8.11.1 in \cite{Ell_Krylov_1996}) implies that for each $R>0$, we have
\begin{align*}
\| u_n\|_{C^{2,\alpha}_\textup{para}(Q_R)} \le N \left(\| g_n\|_{C^\alpha_\textup{para}(\Omega_T)} + \|u_n\|_{L^\infty(Q_{2R})} \right),
\end{align*}where $Q_R = (-R^2,0)\times \{(x,v): |x|^2 + |v|^2 \le R^2\}$, and $N$ depends only on $R, \alpha$, and the ellipticity constant (equal to 1 in our case). Since $g\in C^\alpha_\textup{para}(\Omega_T)$, the terms on the right-hand side are uniformly bounded in $n$. Moreover, by Young's inequality for convolutions and the uniform continuity of $u$, the $L^\infty$ norm of $u_n$ in $Q_{2R}$ is also uniformly bounded. Thus, $\{u_n\}$ is uniformly bounded in $C^{2,\alpha}_\textup{para}(Q_R)$. By the Arzela-Ascoli theorem, up to a subsequence, $u_n \to \bar{u}\in C^2_1(Q_{2R})$, where $C^2_1$ denotes functions with classical derivatives of order $\ge2$ in $(x,v)$ and $\ge1$ in $t$. 

Meanwhile, we also have $u_n\to u$ in $C^\alpha_\textup{para}(\Omega_T)$, so
\begin{align*}
\| u-\bar{u}\|_{L^\infty(Q_{2R})} \le \|u-u_n\|_{C^\alpha_\textup{para}(\Omega_T)} + \|u_n - \bar{u}\|_{C^{2}_1(Q_{2R})} \to0 \quad \textup{ as }\,\,n\to\infty.
\end{align*}Hence $u=\bar{u}\in C^{2,\alpha}_\textup{para}(Q_{2R})$, and this holds for all $R>0$ with arbitrary center of the cylinder, implying
\begin{align*}
u\in C^{2,\alpha}_{\textup{para, loc}}(\Omega_T).
\end{align*}To obtain global regularity, note that $\| u\|_{C^{2,\alpha}_\textup{para}(Q_R(z_0)\cap \Omega_T)}<\infty$ for all $R>0$ and $z_0\in\Omega_T$. Also, by Theorem 8.7.3 in \cite{Ell_Krylov_1996}, since $g\in C^\alpha_\textup{para}(\Omega_T)$, there exists a unique global solution $u_0\in C^{2,\alpha}_\textup{para}(\Omega_T)$ to
\begin{align*}
\partial_tu_0 - \Delta_{x,v}u_0+\lambda u_0 = g, \quad \textup{ in }\,\,\Omega_T.
\end{align*}By uniqueness, $u=u_0$ in every compact subset of $\Omega_T$, and thus
\begin{align*}
u\in C^{2,\alpha}_\textup{para}(\Omega_T).
\end{align*}This completes the proof.
\end{proof}

We are now in a position to address the solvability of the linear approximate equation \eqref{semi}.
\begin{lemma}[Solution to a linear problem] \label{Linear existence}Let $\alpha,p_0,q_0,w$ be defined as \eqref{Y comp condi}. Let $\varepsilon>0$ and let $g\in Y_{T,\alpha,p_0,q_0,w}$ with $g\ge0$ and let $f^\textup{in}_\varepsilon\in C^\infty_c(\mathbb{R}^6)$ be the approximated initial data defined in \eqref{approx initial}. Then for each $T>0$, there exists a unique function $f\in Y_{T,\alpha,p_0,q_0,w}$ that solves
\begin{align}\label{semi problem}
\mathcal{L}_{\varepsilon,g}(f) = \mathcal{L}_{\varepsilon,g}(f^\textup{in}_\varepsilon)I^\varepsilon \quad \textup{ in }\,\, \mathcal{Q}_T,\textup{ and }f(t,\cdot,\cdot) = f^\textup{in}_\varepsilon(\cdot,\cdot) \quad \textup{ in }\,\, t\in[0,\varepsilon),
\end{align} where $\mathcal{L}_{\varepsilon,g}$ is defined as in \eqref{def.Leg}.
\end{lemma}

\begin{proof}
We first work in the extended domain $\Omega_T = (-\infty, T)\times\mathbb{R}^6$. From Lemmas \ref{L^infty bound of coefficients}, \ref{lem.Hold.cont}, and \ref{Holder cont of E_g}, we obtain global Schauder-type \textit{a priori estimate} for the terminal-value problem associated with $\mathcal{L}_{\varepsilon,g}$. In particular, by Theorem \ref{Schauder}, for each $\lambda>0$, there exists a constant $C_S>0$, depending only on $\varepsilon, T, \lambda, g, C_g, \alpha$ such that
\begin{align}\label{lem4.4eq1}
\lVert f\rVert_{C_\textup{para}^{2,\alpha}(\Omega_T)} \le C_S \lVert \mathcal{L}_{\varepsilon,g}(f) + \lambda f\rVert_{C_\textup{para}^{\alpha}(\Omega_T)}, \quad \forall f\in C_\textup{para}^{2,\alpha}(\Omega_T).
\end{align}Similarly, by Theorem \ref{L^p} and Lemmas \ref{L^infty bound of coefficients} and \ref{lem.Hold.cont}, we have $L^{p_0}$-type estimates; i.e., there exists $\lambda_0\ge1$, such that for all $\lambda\ge \lambda_0$ there exists $C_{p_0}>0$ depending only on $\varepsilon, p_0, T, g, C_g,\lambda$ with
\begin{align}\label{lem4.4eq2}
\lVert f\rVert_{W_{p_0}^{1,2}(\Omega_T)} \le C_{p_0}\lVert \mathcal{L}_{\varepsilon,g}(f) + \lambda f\rVert_{L_{p_0}(\Omega_T)}, \quad \forall f\in W_{p_0}^{1,2}(\Omega_T).
\end{align}Moreover, for weighted mixed Sobolev spaces, by \cite{Ell_Dong_2018} or Theorem \ref{Weighted mixed L^p estimate}, there exists $\lambda\ge1$ such that for all $\lambda\ge \lambda_0$, there exists $C_w>0$ depending only on $\varepsilon, q_0, T, g, C_g, w=\langle x\rangle \langle v\rangle^8$ and $\lambda$ such that
\begin{align}\label{lem4.4eq3}
\lVert f\rVert_{W_{q_0,w}^{1,2}(\Omega_T)} \le C_w\lVert \mathcal{L}_{\varepsilon,g} (f) + \lambda f\rVert_{L_{q_0,w}(\Omega_T)}, \quad \forall f\in W_{q_0,w}^{1,2}(\Omega_T).
\end{align}
Then we are now ready to proceed with the method of continuity (Theorem \ref{MoC}). Define the operators $\mathcal{L}_1 \eqdef \mathcal{L}_{\varepsilon, g}+\lambda$ and $\mathcal{L}_0\eqdef \partial_t - \Delta_{x,v}+\lambda$. Then, we obtain the analogous inequalities to \ref{Schauder}-\ref{Weighted mixed L^p estimate} in terms of the operator $\mathcal{L}_0$ as well. Therefore, by interpolating these endpoint inequalities, we obtain that 
there is some constant $C>0$ independent of $\beta\in[0,1]$ such that 
$$\|f\|_{(C_\textup{para}^{2,\alpha}\cap W_{p_0}^{1,2}\cap W_{q_0,w}^{1,2})(\Omega_T)}\le C\|\mathcal{L}_\beta f\|_{(C^\alpha_\textup{para} \cap L_{p_0} \cap L_{q_0,w})(\Omega_T)},$$ where $\mathcal{L}_\beta\eqdef(1-\beta)\mathcal{L}_0+\beta\mathcal{L}_1$ for each $\beta\in[0,1].$
Then by the method of continuity (Theorem \ref{MoC}) and Lemma \ref{Heat_solvability}, the operator $\mathcal{L}_1=\mathcal{L}_{\varepsilon,g} + \lambda:(C_\textup{para}^{2,\alpha}\cap W_{p_0}^{1,2}\cap W_{q_0,w}^{1,2})(\Omega_T)\to (C^\alpha_\textup{para} \cap L_{p_0} \cap L_{q_0,w})(\Omega_T)$ is onto and hence is invertible. 

Now, define the time-regularized source term:
\begin{align*}
h = (\mathcal{L}_{\varepsilon,g} + \lambda)(e^{-\lambda t} f^\textup{in}_\varepsilon)I^\varepsilon \in (C^\alpha_\textup{para} \cap L_{p_0} \cap L_{q_0,w})(\Omega_T),
\end{align*}which is well-defined since $f^\textup{in}_\varepsilon\in C^\infty_c(\mathbb{R}^6)$. Then by the above solvability result, there exists a unique solution $f'\in (C_\textup{para}^{2,\alpha}\cap W^{1,2}_{p_0} \cap W_{q_0,w}^{1,2})(\Omega_T)$ to the equation:
\begin{align*}
\mathcal{L}_{\varepsilon,g} f' + \lambda f' = h \quad \textup{ in }\,\, \Omega_T.
\end{align*}Note that
\begin{itemize}
\item $I^\varepsilon(t) = 1$ on $[0,\varepsilon]$,
\item $I^\varepsilon(t) = 0$ on $(2\varepsilon, T]$,
\end{itemize}so in particular,
\begin{align*}
\mathcal{L}_{\varepsilon,g} f' + \lambda f' = 0\quad &\textup{ in } \left(2\varepsilon, T\right)\times\mathbb{R}^6,\textup{ and }
f' = e^{-\lambda t}f^\textup{in}_\varepsilon\quad \textup{ in } \left[0, \varepsilon\right]\times\mathbb{R}^6.
\end{align*}Define $f(t,x,v) \eqdef e^{\lambda t}f'(t,x,v)$. Then, by linearity of the operator $\mathcal{L}_{\varepsilon,g}+\lambda$, we obtain
\begin{align*}
\mathcal{L}_{\varepsilon,g}(f) = e^{\lambda t}(\mathcal{L}_{\varepsilon,g} f' + \lambda f') = \mathcal{L}_{\varepsilon,g}(f^\textup{in}_\varepsilon)I^\varepsilon,\textup{ and }
\end{align*}  
\begin{equation}\label{f equals fin}
f(t,x,v) = f^\textup{in}_\varepsilon(x,v) \quad \textup{ for }\,\,t\in[0,\varepsilon).
\end{equation}This completes the construction.
\end{proof}

\label{sequence}By applying Lemma \ref{Linear existence} inductively under the assumption \eqref{bound g}, we construct a sequence $\{f^n\}_{n\ge0}$ $\subset Y_{T,\alpha,p_0,q_0,w}$ satisfying the following equation for each $n\ge1$:
\begin{multline}\label{n approx}
\partial_tf^n + \chi\left(\frac{v}{n}\right)v\cdot\nabla_xf^n + J^{\frac{1}{n}}(t)\chi\left(\frac{x}{n}\right)E_{f^{n-1}}\cdot\nabla_vf^n - (-\Delta_v)_{n^{-1}}^{-1}f^{n-1} \Delta_vf^n - f^{n-1} f^n \\
 =\frac{1}{n}\Delta_{x,v} f^n + \mathcal{L}_{\frac{1}{n}, f^{n-1}}(f^\textup{in}_{n^{-1}})I^{\frac{1}{n}}(t) \quad \textup{ in }\,\, \overline{\mathcal{Q}_T},
\end{multline}with $f^0 \equiv 0$. We will explicitly use this sequence in the discussion following Lemma \ref{local comparison}.

\section{Upper Bounds and Passing to the Limit}
\label{boundedness}In this section, we establish the polynomial decay of the sequence $\{f^{n}\}_{n\ge 0}$ constructed in the previous section.  
We first introduce a lemma on the comparison of weights, which will be used to show that each iterated element $f^{n}$ exhibits the desired decay in $X_T^m$.

\begin{lemma}\label{embed}Let $\varepsilon>0, T>0$, and $m>0$. Then for all $(t,x,v)\in\mathcal{Q}_T$, we have
\begin{align*}
 \langle x\rangle^{-m} \langle v\rangle^{-m}\le (2\langle T\rangle)^m \langle x-\chi(\varepsilon v)vt\rangle^{-m},
\end{align*}where $\chi(\cdot)$ is the cut-off function defined as in \eqref{def cutoff}.
\end{lemma}

\begin{proof}
Fix $t\in[0,T]$. Then by the Schwarz inequality,
\begin{align*}
    1+|x-\chi(\varepsilon v)vt|^2 &\le 1+2|x|^2 + 2t^2|v|^2\\
    &\le (1+2|x|^2)(1+2t^2|v|^2).
\end{align*}Thus,
\begin{align*}
    \langle x-\chi(\varepsilon v)vt\rangle^2 \le 4(1+T^2) \langle x\rangle^2\langle v\rangle^2.
\end{align*}Hence, in $0\le t\le T$, we have
\begin{align*}
    \langle x\rangle^{-m}\langle v\rangle^{-m} \le (2\langle T\rangle)^{m}\langle x-\chi(\varepsilon v)vt\rangle^{-m}.
\end{align*}
\end{proof}

We now introduce a crucial lemma concerning the comparison principle.  
For each fixed $g$, we compare the solution $f$ of the $\varepsilon$–approximated linear problem \eqref{semi problem} with an auxiliary function $\overline{f}_{\varepsilon}$ depending on $\varepsilon>0$.  
This comparison directly yields the same upper and lower bounds for each iterated element $f^{n}$ solving \eqref{n approx}.

\begin{lemma}[Function-wise upper and lower bound]\label{local comparison} Let $g \in Y_{T,\alpha,p_0,q_0,w}$ satisfy $g\ge0$ and
\begin{align*}
 g(t,x,v) \le C_g\langle x-\chi(\varepsilon' v)vt \rangle^{-m}\langle v\rangle^{-m},
\end{align*}for some $C_g\ge 0$ and $0<\varepsilon'\le 1$. Let $f\in Y_{T,\alpha,p_0,q_0,w}$ be a solution to the linear problem \eqref{semi problem} for some $0<\varepsilon\le \varepsilon'$ with the approximated initial data $f^\textup{in}_\varepsilon\in C^\infty_c(\mathbb{R}^6)$ (see the definition \eqref{approx initial}) satisfying $f^\textup{in}_\varepsilon\ge0$. Then, there exists a constant $M=M(m)>0$ such that
\begin{align*}
 f^\textup{in}_\varepsilon I^\varepsilon\le f(t,x,v) \le e^{M(T+1)^3(C_g+1)}\lVert f^\textup{in}_\varepsilon\rVert_{\mathcal{C}^2_{m,0}}(2\langle T\rangle)^m\langle x -\chi(\varepsilon v)vt\rangle^{-m} \langle v\rangle^{-m},
\end{align*} for $(t,x,v)\in\overline{\mathcal{Q}_T}$ where $I^\varepsilon$ is defined as in \eqref{reg interval}.
\end{lemma}

\begin{proof}
\textbf{Upper bound on $t\in[0,\varepsilon]$.} In this interval, we have $f(t)=f^\textup{in}_\varepsilon$ by \eqref{f equals fin}. Then by Lemma \ref{embed}, we have 
\begin{align*}
|f(t,x,v)| \le \langle x\rangle^{-m} \langle v\rangle ^{-2m} |\langle x\rangle^m \langle v\rangle ^{2m} f^\textup{in}_\varepsilon|\le \lVert f^\textup{in}_\varepsilon\rVert_{\mathcal{C}^2_{m,0}}(2\langle T\rangle)^m\langle x-\chi(\varepsilon v)vt\rangle^{-m}\langle v\rangle^{-m}.
\end{align*}This proves the desired upper bound with $M=0$.

\noindent\textbf{Upper bound on }$t\in[\varepsilon,T]$\textbf{.} We define the auxiliary function
\begin{align*}
\overline{f}_\varepsilon(t,x,v) = e^{C'(t-\varepsilon)}(2\langle T\rangle)^m\lVert f^\textup{in}_\varepsilon\rVert_{\mathcal{C}^2_{m,0}(\mathbb{R}^6)} \langle x - \chi(\varepsilon v)vt\rangle^{-m}\langle v\rangle^{-m},
\end{align*}with a sufficiently large $C'>0$ which will be determined as in \eqref{final C'}. Note that by Lemma \ref{embed} again, we have $\overline{f}_\varepsilon(\varepsilon,x,v)\ge f^\textup{in}_\varepsilon(x,v)= f(\varepsilon, x,v)$. Therefore, by the comparison principle, to establish $f\le \overline{f}_\varepsilon$, it suffices to verify the following:
\begin{align}\label{4.2 desired ineq}
\mathcal{L}_{\varepsilon,g}(\overline{f}_\varepsilon) \ge \mathcal{L}_{\varepsilon,g}(f)=\mathcal{L}_{\varepsilon,g}(f^\textup{in}_\varepsilon)I^\varepsilon  \quad \textup{ in }\,\, (\varepsilon, T)\times\mathbb{R}^6,
\end{align}where $\mathcal{L}_{\varepsilon,g}$ denotes the linear differential operator considered in \eqref{semi}.

To this end, we proceed to compute and estimate each term in $\mathcal{L}_{\varepsilon,g}(\overline{f}_\varepsilon)$. 

\noindent \textit{(I) Transport term:} Note that
\begin{align}\label{Trans est}
(\partial_t + \chi(\varepsilon v)v\cdot\nabla_x)\overline{f}_\varepsilon = C'\overline{f}_\varepsilon,
\end{align}since $(\partial_t + \chi(\varepsilon v) v\cdot\nabla_x)\langle x-\chi(\varepsilon v)vt\rangle^{-m} \equiv 0$.

\noindent \textit{(II) Acceleration term:} To estimate the acceleration term $J^\varepsilon(t)\chi(\varepsilon x)E_g\cdot\nabla_v\overline{f}_\varepsilon$, we begin with computing the velocity gradient 
\begin{align*}
\nabla_v\langle x-\chi(\varepsilon v)vt\rangle^{-m}.
\end{align*}Using the chain rule, this term can be written as
\begin{multline*}
\nabla_v\langle \chi(\varepsilon v)vt - x\rangle^{-m} = -\frac{m}{2} \langle \chi(\varepsilon v)vt - x\rangle^{-m-2} \nabla_v|\chi(\varepsilon v)vt - x|^2\\
= -m t\langle \chi(\varepsilon v)vt - x\rangle^{-m-2}\left[\chi(\varepsilon v)(\chi(\varepsilon v)vt - x) + \varepsilon(\nabla_v\chi)(\varepsilon v)\left\{v\cdot(\chi(\varepsilon v)vt - x)\right\}\right].
\end{multline*}Now apply the bounds, $0\le \chi\le1$, $|\nabla_v\chi|<\infty$ and $|v|\le\langle v\rangle$, we obtain the estimate:
\begin{align}\label{v grad}
|\nabla_v\langle x-\chi(\varepsilon v)vt\rangle^{-m}| \le C(m,\chi)T\langle x-\chi(\varepsilon v)vt\rangle^{-m}.
\end{align}
To estimate the electric field $E_g$, we use the bound:
\begin{align*}
|E_g(t,x)| \le C_g\sup_{(t,x)\in[0,T]\times\mathbb{R}^3} \int_{\mathbb{R}^3} \frac{dv}{\langle v\rangle^m} \int_{\mathbb{R}^3} \frac{dy}{|x-\chi(\varepsilon v)vt-y|^2\langle y \rangle^m}.
\end{align*}By using \eqref{int decay}, we conclude 
\begin{align}\label{E bound}
\lVert E_g\rVert_{L^\infty([0,T]\times\mathbb{R}^3)} \lesssim_m C_g.
\end{align}
Therefore, we have
\begin{align}
|J^\varepsilon(t)\chi(\varepsilon x)E_g \cdot\nabla_v\overline{f}_\varepsilon| &\le \lVert E_g\rVert_{L^\infty_{t,x}} |\nabla_v\overline{f}_\varepsilon| \le C(m,\chi)\lVert E_g\rVert_{L^\infty_{t,x}} (T+1)\overline{f}_\varepsilon\nonumber\\
&\le C_1(m,\chi)(T+1)C_g\overline{f}_\varepsilon.\label{Acc est}
\end{align}Here, the constant $C_1$ may be redefined, if necessary, to absorb numerical factors.

\noindent \textit{(III) Nonlocal collision term:} To estimate the term $(-\Delta_v)^{-1}g\Delta_v\overline{f}_\varepsilon$, we first compute the velocity Laplacian of the weight function:
\begin{align*}
\Delta_v\langle x-\chi(\varepsilon v)vt\rangle^{-m}.
\end{align*}By direct computation, this can be written as:
\begin{multline*}
\Delta_v\langle x-\chi(\varepsilon v)vt\rangle^{-m} 
= -mt\left[\nabla_v\langle x-\chi(\varepsilon v)vt\rangle^{-m-2}\right]\cdot [\chi(\varepsilon v)(\chi(\varepsilon v)vt - x) \\
+\varepsilon(\nabla_v\chi)(\varepsilon v) \left\{v\cdot(\chi(\varepsilon v)vt -x)\right\} ] \\
- mt \langle \chi(\varepsilon v)vt - x\rangle^{-m-2}\nabla_v\cdot\left[\chi(\varepsilon v)(\chi(\varepsilon v)vt - x) + \varepsilon(\nabla_v\chi)(\varepsilon v) \left\{v\cdot(\chi(\varepsilon v)vt -x)\right\}\right]\\
=: \Delta_1 + \Delta_2.
\end{multline*}We first estimate $\Delta_1$. Noting that
\begin{multline*}
\nabla_v\langle x-\chi(\varepsilon v)vt\rangle^{-m-2} = -(m+2)t \langle \chi(\varepsilon v) vt - x\rangle^{-m-4}[\chi(\varepsilon v)(\chi(\varepsilon v)vt - x) \\
+ \varepsilon(\nabla_v\chi)(\varepsilon v)\left\{v\cdot(\chi(\varepsilon v)vt - x)\right\}],
\end{multline*}we deduce the bound
\begin{align}\label{Delta 1}
|\nabla_v\langle x-\chi(\varepsilon v)vt\rangle^{-m-2}| \le (1+\|\nabla_v\chi\|_{L^\infty})(m+2)T\langle \chi(\varepsilon v)vt - x\rangle^{-m-3}.
\end{align}Next, we simplify and estimate $\Delta_2$. For the first divergence component:
\begin{multline*}
\nabla_v\cdot\left[\chi(\varepsilon v)(\chi(\varepsilon v)vt - x)\right]
=\nabla_v\left[\chi(\varepsilon v)\right]\cdot(\chi(\varepsilon v)vt - x) + \chi(\varepsilon v)\nabla_v\cdot(\chi(\varepsilon v) vt - x)\\
= \varepsilon(\nabla_v\chi)(\varepsilon v)\cdot (\chi(\varepsilon v)vt - x) + \chi(\varepsilon v)\left[3t\chi(\varepsilon v) + v\cdot\varepsilon (\nabla_v\chi)(\varepsilon v)t\right],
\end{multline*}yielding the bound
\begin{align}\label{Delta 2-1}
|\nabla_v\cdot[\chi(\varepsilon v)(\chi(\varepsilon v)vt - x)]| \le \|\nabla_v\chi\|_{L^\infty}|\chi(\varepsilon v)vt - x| + (3+\|\nabla_v\chi\|_{L^\infty})T.
\end{align}For the second divergence component, we compute:
\begin{multline*}
\nabla_v\cdot \left[\varepsilon (\nabla_v\chi)(\varepsilon v)\left\{v\cdot(\chi(\varepsilon v)vt - x)\right\}\right]\\
=\varepsilon^2(\Delta_v\chi)(\varepsilon v)\left\{v\cdot(\chi(\varepsilon v)vt - x)\right\} + \varepsilon (\nabla_v\chi)(\varepsilon v)\cdot\nabla_v\left\{v\cdot(\chi(\varepsilon v)vt - x) \right\}\\
=\varepsilon^2(\Delta_v\chi)(\varepsilon v)\left\{v\cdot(\chi(\varepsilon v)vt - x)\right\} 
\\
+ \varepsilon(\nabla_v\chi)(\varepsilon v)\cdot\left[3(\chi(\varepsilon v)vt - x) + tv(\varepsilon(\nabla_v\chi)(\varepsilon v)\cdot v) + 3tv\chi(\varepsilon v)\right].
\end{multline*}Then using $0<\varepsilon\le 1$, we obtain
\begin{multline}
    |\nabla_v\cdot \left[\varepsilon (\nabla_v\chi)(\varepsilon v)\left\{v\cdot(\chi(\varepsilon v)vt - x)\right\}\right]| \\
    \le\|\Delta_v\chi\|_{L^\infty}|\chi(\varepsilon v)vt - x| + 3\|\nabla_v\chi\|_{L^\infty}|\chi(\varepsilon v) vt - x| + \|\nabla_v\chi\|_{L^\infty}^2T + 3\|\nabla_v\chi\|_{L^\infty}T\\
= (\|\Delta_v\chi\|_{L^\infty}+ 3\|\nabla_v\chi\|_{L^\infty})|\chi(\varepsilon v) vt-x| + (\|\nabla_v\chi\|_{L^\infty}^2 + 3\|\nabla_v\chi\|_{L^\infty})T.\label{Delta 2-2}
\end{multline}Combining these estimates \eqref{Delta 1}, \eqref{Delta 2-1} and \eqref{Delta 2-2}, we estimate the full Laplacian:
\begin{multline}
|\Delta_v\langle \chi(\varepsilon v) vt - x\rangle^{-m}|\le |\Delta_1| + |\Delta_2| \\
\le mT  (1+\|\nabla_v\chi\|_{L^\infty})^2(m+2)T\langle \chi(\varepsilon v)vt - x\rangle^{-m-3} |\chi(\varepsilon v)vt - x|\\
+ mT \langle \chi(\varepsilon v)vt - x\rangle^{-m-2} (\|\nabla_v\chi\|_{L^\infty}|\chi(\varepsilon v)vt - x| + (3+\|\nabla_v\chi\|_{L^\infty})T \\+ (\|\Delta_v\chi\|_{L^\infty}+ 3\|\nabla_v\chi\|_{L^\infty})|\chi(\varepsilon v)vt - x| + (\|\nabla_v\chi\|_{L^\infty}^2 + 3\|\nabla_v\chi\|_{L^\infty})T)\\
\le C(m,\chi)(T+1)^2 \langle \chi(\varepsilon v)vt - x\rangle^{-m}.\label{Delta}
\end{multline}
To complete the estimate, we compute the Laplacian of the full weight function:
\begin{multline*}
\Delta_v\left(\langle x-\chi(\varepsilon v)vt\rangle^{-m} \langle v\rangle^{-m}\right) = \Delta_v\left(\langle x-\chi(\varepsilon v)vt\rangle^{-m}\right)\langle v\rangle^{-m} \\ +2\nabla_v\langle x-\chi(\varepsilon v)vt\rangle^{-m}\cdot\nabla_v\langle v\rangle^{-m} +
\langle x-\chi(\varepsilon v)vt\rangle^{-m}\Delta_v\left(\langle v\rangle^{-m}\right).
\end{multline*}Applying the previous bounds, \eqref{v grad} and \eqref{Delta}, we obtain
\begin{align}\label{total Delta}
 \Delta_v\left(\langle x-\chi(\varepsilon v)vt\rangle^{-m} \langle v\rangle^{-m}\right)\le C(m,\chi)(T+1)^2\langle x-\chi(\varepsilon v)vt\rangle^{-m}\langle v\rangle^{-m}.
\end{align}

Altogether, we obtain the following upper bound for the localized diffusion term:
\begin{align*}
|(-\Delta_v)_\varepsilon^{-1}g\Delta_v\overline{f}_\varepsilon| \le C_2(m,\chi)(T+1)^2\lVert (-\Delta_v)_\varepsilon^{-1}g\rVert_{L^\infty_{t,x,v}} \overline{f}_\varepsilon.
\end{align*}Moreover, since
\begin{align*}
(-\Delta_v)_\varepsilon^{-1}g \le C_g \int_{\mathbb{R}^3} \frac{du}{|u-v|\langle u\rangle^m} \le C_2(m)C_g,
\end{align*}we conclude that
\begin{align}\label{Coll1 est}
|(-\Delta_v)_\varepsilon^{-1}g\Delta_v\overline{f}_\varepsilon| \le C_2(m,\chi)(T+1)^2C_g\overline{f}_\varepsilon,
\end{align}where $C_2(m,\chi)$ may be redefined to absorb all multiplicative factors, if necessary.

\noindent \textit{(IV) Additional collision term: }The remaining contribution from the collision operator is bounded from above in a straightforward manner. Specifically, we have
\begin{align}\label{Coll2 est}
g\overline{f}_\varepsilon \le C_g \overline{f}_\varepsilon.
\end{align}

\noindent \textit{(V) Artificial diffusion term:} To estimate the diffusion term $\Delta_{x,v}\overline{f}_\varepsilon$, we compute the spacial Laplacian of the weight function $\langle x-\chi(\varepsilon v)vt\rangle^{-m}$. A direct calculation yields that
\begin{multline*}
\Delta_x\langle x-\chi(\varepsilon v)vt\rangle^{-m} = \nabla_x\cdot[-mx\langle x-\chi(\varepsilon v)vt\rangle^{-(m+2)}]\\
= -3m\langle x-\chi(\varepsilon v)vt\rangle^{-(m+2)} +m(m+2)\langle x-\chi(\varepsilon v)vt\rangle^{-(m+4)}(x-\chi(\varepsilon v)vt),
\end{multline*}which implies the bound:
\begin{align*}
\Delta_x\langle x-\chi(\varepsilon v)vt\rangle^{-m} \le m(m+5)\langle x-\chi(\varepsilon v)vt\rangle^{-m}.
\end{align*}Combining this with  \eqref{total Delta}, we obtain
\begin{align}\label{Artifical est}
\Delta_{x,v}\overline{f}_\varepsilon \le C_2(m,\chi)(T+1)^2\overline{f}_\varepsilon,
\end{align}where the constant $C_2$ may be redefined to absorb numerical factors, if necessary. 

\noindent \textit{(VI) Artifical inhomogeneous term associated with initial data:} Since $f^\textup{in}_\varepsilon\in \mathcal{C}^2_{m,0}$, Lemma \ref{embed} yields the following estimate:
\begin{multline}\label{Initial est}
|\mathcal{L}_{\varepsilon,g}(f^\textup{in}_\varepsilon)| \le
\left(\lVert v\cdot\nabla_xf^\textup{in}_\varepsilon\rVert_{L^\infty_{m,2m}} + (C_g+1)(C_1(m) \lVert \nabla_vf^\textup{in}_\varepsilon\rVert_{L^\infty_{m,2m}}\right. \\
\left.+ C_2(m) \lVert \Delta_{v} f^\textup{in}_\varepsilon\rVert_{L^\infty_{m,2m}} + \lVert f^\textup{in}_\varepsilon\rVert_{L^\infty_{m,2m}})\right) (\langle x\rangle^{-m}\langle v\rangle^{-m})\langle v\rangle^{-m}\\
\le C(m,\chi)\lVert f^\textup{in}_\varepsilon\rVert_{\mathcal{C}^2_{m,0}}\left(C_g+1\right) (2\langle T\rangle)^m \langle x-\chi(\varepsilon v)vt\rangle^{-m}\langle v\rangle^{-m},
\end{multline}for some $C(m,\chi)>0.$

Combining \eqref{Trans est}, \eqref{Acc est}, \eqref{Coll1 est}, \eqref{Coll2 est}, \eqref{Artifical est} and \eqref{Initial est}, we assign a sufficiently large constant 
\begin{align}\label{final C'}
C' \eqdef M(m,\chi)(T+1)^2(C_g+1),
\end{align} such that the \textit{transport term} in \textit{(I)} absorbs all the rest of the terms from \textit{(II)} to \textit{(VI)} on the left-hand side of the desired inequality \eqref{4.2 desired ineq}. This allows us to conclude that
\begin{align*}
\mathcal{L}_{\varepsilon,g}(\overline{f}_\varepsilon) \ge \mathcal{L}_{\varepsilon,g}(f^\textup{in}_\varepsilon)I^\varepsilon \quad \textup{ in }\,\, \mathcal{Q}_T.
\end{align*}
By merging the upper bounds over two distinct time intervals, the desired estimate is achieved.

\textbf{Lower bound in $t\in[0,\varepsilon]$.} Since $f=f^\textup{in}_\varepsilon$ in this time interval, the desired lower bound $f\ge f^\textup{in}_\varepsilon I^\varepsilon$ trivially holds.

\textbf{Lower bound in $t\in[\varepsilon, T]$.} By comparison principle, it suffices to show 
\begin{align*}
\mathcal{L}_{\varepsilon,g}(f^\textup{in}_\varepsilon I^\varepsilon) \le \mathcal{L}_{\varepsilon,g}(f)=\mathcal{L}_{\varepsilon,g}(f^\textup{in}_\varepsilon)I^\varepsilon \quad \textup{ in }\,\, [\varepsilon, T]\times\mathbb{R}^6.
\end{align*}This inequality holds, since $$\mathcal{L}_{\varepsilon,g}(f^\textup{in}_\varepsilon I^\varepsilon) =\mathcal{L}_{\varepsilon,g}(f^\textup{in}_\varepsilon)I^\varepsilon+ f^\textup{in}_\varepsilon\partial_t I^\varepsilon, $$ and, by definition of $I^\varepsilon(t)$ in \eqref{reg interval}, we have $\partial_tI^\varepsilon \le0$ for each $t\in[\varepsilon,T]$ and  $\varepsilon>0$. Therefore, we conclude that
\begin{align*}
f^\textup{in}_\varepsilon I^\varepsilon \le f \quad \textup{in}\,\, t\ge\varepsilon.
\end{align*}

By combining the lower bounds on the two distinct time intervals, we obtain the desired estimate.

In the above proof, the constant $M$ was written as if it depends on $\chi$. More precisely, increasing $\|\chi\|_{C^{2}}$ increases the corresponding value $M(\chi)$. However, this dependence on $\chi$ can in fact be removed. We may define
$$
    M(m) \coloneqq \inf\{\, M(m,\chi) : \chi\in C_c^2(\mathbb{R}^3) \textup{ and satisfies }\eqref{def cutoff}\,\},
$$
which yields a choice of $M$ independent of $\chi$. The infimum $M(m)>0$ exists for each fixed $m$. This completes the proof.
\end{proof}

Since our initial sequence element $f^0\equiv0$ for the approximated system \eqref{n approx} is non-negative and bounded from above by a constant multiple of $\langle x - \chi(v)vt\rangle^{-m}\langle v\rangle^{-m}$, we can apply Lemma \ref{local comparison} inductively on each element of the iterated solutions $\{f^n\}_{n\ge 0}$. Then we first obtain that $f^1$ satisfies
\begin{align}\label{C_1}
f^\textup{in}_1 I^{1}\le f^1 \le C_{0} \left\langle x - \chi\left(v\right)vt\right\rangle^{-m} \langle v\rangle^{-m},
\end{align}for some $C_0>0$. Then by choosing $\varepsilon'=\frac{1}{n-1}$ and $\varepsilon=\frac{1}{n}$ on each step for $n\ge 2$, we have 
\begin{align}\label{C_n}
f^\textup{in}_{n^{-1}} I^{\frac{1}{n}}\le f^n \le C_{n-1} \left\langle x - \chi\left(\frac{v}{n}\right)vt\right\rangle^{-m} \langle v\rangle^{-m},
\end{align}where $C_{n-1}$ depends only on $f^{n-1}, T, f^\textup{in}_{n^{-1}}$ and $m$. In Lemma \ref{uniform bound}, we will choose $f^\textup{in}_{n^{-1}}\in \mathcal{C}^2_{m,0}$ sufficiently small so that the sequence $\{C_n\}_{n\in\mathbb{N}\cup \{0\}}$ is uniformly bounded.

\begin{lemma}[Uniform upper bound]\label{uniform bound} Fix $m>3$. Let $M=M(m)$ be as introduced in Lemma \ref{local comparison}. Assume that the initial datum $f^\textup{in}\in \mathcal{C}^2_{m,0}(\mathbb{R}^6)$ satisfies $f^\textup{in}\ge 0$ and
\begin{align}
    \|f^\textup{in}\|_{\mathcal{C}^2_{m,0}} \le \frac{1}{e^{-e}M(2\langle T\rangle)^m(T+1)^3e^{M(T+1)^3}}. \label{init bd}
\end{align}Then each sequence element $f^n$ satisfies
\begin{align*}
f^\textup{in}_{n^{-1}}I^{\frac{1}{n}}\le f^n \le C \left\langle x - \chi\left(\frac{v}{n}\right)vt\right\rangle^{-m} \langle v\rangle^{-m} \quad \textup{ in } \,\, \mathcal{Q}_T,
\end{align*} for some $C=C(f^\textup{in},m)$ given by
$$
    C = \frac{c_0}{M^2  \,\|f^\textup{in}\|_{\mathcal{C}^2_{m,0}} },
$$where $c_0>0$ is an absolute constant. In particular, this lemma specifies the time interval $[0,T]$ associated with the local existence result stated in Theorem~\ref{thm.main1}.
\end{lemma}

\begin{proof}
Define
\begin{align*}
\|f^\textup{in}\| \eqdef \| f^\textup{in}\|_{\mathcal{C}^2_{m,0}}(2\langle T\rangle)^m,
\end{align*}and for each $n\in\mathbb{N}$, define
\begin{align*}
\| f\|_n \eqdef \sup_{(t,x,v)\in \overline{\mathcal{Q}_T}}\langle x-\chi(v/n)vt\rangle^m \langle v\rangle^m|f(t,x,v)|.
\end{align*}Before proceeding, recall the definition of the approximated initial data $f^\textup{in}_{n^{-1}}$ in \eqref{approx initial}, which is obtained by cutting off $f^\textup{in}$. Consequently, $\|f^\textup{in}_{n^{-1}}\|_{\mathcal{C}^2_{m,0}} \uparrow \|f^\textup{in}\|_{\mathcal{C}^2_{m,0}}$ as $n\to\infty$. For $n=0$, we put $f^0\equiv 0$. Also, introduce the function
\begin{align*}
\Phi(s) \eqdef e^{M(T+1)^3(s + 1)}\lVert f^\textup{in}\rVert.
\end{align*}We begin with the base case. We first observe that
\begin{align*}
\lVert f^1\rVert_1 \le \Phi(\lVert f^0\rVert_0)
\end{align*} holds by Lemma \ref{local comparison} with $g=f^0=0$, $C_g=0$, $f=f^1$, $\varepsilon'=\varepsilon=1$.  Recall that $\lVert f^0\rVert_0 =0$. Next, by applying Lemma \ref{local comparison} for $f^2$ with $g=f^1$, $C_g= \|f^1\|_1$, $f=f^2$, $\varepsilon'=1,$ and $\varepsilon=\frac{1}{2}$, we obtain:
\begin{align*}
\lVert f^2 \rVert_2
\le \exp\left[M(T+1)^3(\lVert f^1\rVert_1 + 1)\right]\lVert f^\textup{in}_{\frac{1}{2}}\rVert
 \le \Phi(\lVert f^1\rVert_1)
\le (\Phi\circ \Phi)(0),
\end{align*}since $\Phi$ is increasing. By induction, defining $\Phi^k$ as the $k$-fold composition of $\Phi$, it follows that
\begin{align*}
\lVert f^k \rVert_k \le \Phi^k (0) \quad \forall k\in\mathbb{N}.
\end{align*}To conclude the existence of a uniform-in-$n$ constant $C>0$ in Lemma \ref{uniform bound}, it suffices to show the sequence $\{\Phi^k(0)\}_{k\in\mathbb{N}}$ remains bounded -- i.e., that it admits a finite supremum -- and does not converge to zero. According to Lemma \ref{Phi property}, this can be guaranteed by choosing $T>0$ sufficiently small depending on $\lVert f^\textup{in}\rVert_{\mathcal{C}^2_{m,0}}$ so that the algebraic equation
\begin{align*}
\lVert f^\textup{in}\rVert e^{M(T+1)^3(s+1)} = s,
\end{align*}admits a real root. More precisely, a hypothesis of Lemma \ref{uniform bound} implies
\begin{align*}
    \|f^{\mathrm{in}}\|_{\mathcal{C}^2_{m,0}}
    \,(2\langle T\rangle)^m
    M (T+1)^3 e^{M(T+1)^3}
    \le \frac{1}{e},
\end{align*}
which guarantees the existence of such a root (see Lemma \ref{Phi property} (i), with $c= \|f^{\textup{in}}\|_{\mathcal{C}^2_{m,0}}$ and $A= (2\langle T\rangle)^m
    M (T+1)^3 e^{M(T+1)^3}$ in the notation of Lemma \ref{Phi property}).

Finally, we find a uniform upper bound of the weighted $L^\infty$ norm of $\{f^n\}_{n\ge0}$, that is, $\sup_{k\in\mathbb{N}} \Phi^k(0)$.
By Lemma~\ref{Phi property} (ii), we have
$$
\sup_{t\in[0,T]}\|\langle x-\chi(v/n)vt\rangle^m\langle v\rangle^mf^n(t)\|_{L^\infty(\mathbb{R}^6)}
\le \sup_{n\in\mathbb{N}}\Phi^n(0)
\le s_{\mathrm{large}},
$$
where $s_{\mathrm{large}}>0$ (see Lemma \ref{Phi property} (iii)) denotes the larger real root of the algebraic
equation
$$
\Phi(s) = s,
$$whenever the two real roots exist, and denotes the unique real root in the case that the equation admits only one real solution.
Using the asymptotic behavior of the Lambert $W$-function, the larger real
root of the equation $a e^{b s} = s$ with
$$
a \eqdef \|f^{\mathrm{in}}\|_{\mathcal{C}^2_{m,0}}\, e^{M(T+1)^3}, 
\qquad b \eqdef M(T+1)^3
$$ satisfies the upper bound
$$
s_{\mathrm{large}} \le \frac{c_0}{a b^2},
$$for some absolute constant $c_0>0$. 
Therefore, by choosing $C \eqdef c_0M^{-2}\|f^\textup{in}\|^{-1}_{\mathcal{C}^2_{m,0}}$ and recalling that $a,b>0$, we obtain
the desired local-in-time $L^\infty$ bound.
\end{proof}

\begin{remark}\label{T bd}
    By simplifying \eqref{init bd}, we obtain the following sufficient bound for $T>0$:
\begin{align*}
    T \le 
    \left\{
        \frac{1}{M}
        \log\!\left(
            \frac{1}{
                e^{-e}M
                \|f^{\mathrm{in}}\|_{\mathcal{C}^2_{m,0}}C(m)
            }
        \right)
    \right\}^{\!\frac{1}{3}}
    - 1,
\end{align*}
where $C(m)<\infty$ is an upper bound of the function
\begin{align}\label{C const w.r.t. m}
    \psi(t) = (2\langle t\rangle)^m(t+1)^3 e^{-M(t+1)^3}, \qquad t\ge0
\end{align}for fixed $M=M(m)$, where $M(m)$ is constructed in Lemma \ref{local comparison}.
\end{remark}

We are now ready to prove our main local existence result, Theorem~\ref{thm.main1}.

\begin{proof}[Proof of Theorem \ref{thm.main1}] The proof of Theorem \ref{thm.main1} consists of two steps. 

\textbf{Step 1: Passing to the limit. }By Lemma \ref{uniform bound}, the sequence $\{f^n\}_{n\ge0}$ is uniformly bounded in $L^\infty(\mathcal{Q}_T)$. Hence, by the Banach-Alaoglu theorem, there exists a function $f\in L^\infty(\mathcal{Q}_T)$ such that $f^n$ converges to $f$ in the weak*-topology of $L^\infty(\mathcal{Q}_T)$; i.e., in $\sigma(L^\infty(\mathcal{Q}_T), L^1(\mathcal{Q}_T))$. In addition to this boundedness, we establish the non-negativity and the polynomial decay of $f$ in $x,v$ of order $(m,m)$. Let $\varphi\in C_c^2(\mathcal{Q}_T)$ be a non-negative test function, $\varphi\ge0$. Then, for each $n\in\mathbb{N}$,
\begin{align*}
\int_{\mathcal{Q}_T} f\varphi = \lim_{n\to\infty}\int_{\mathcal{Q}_T} f^n\varphi\le C\limsup_{n\to\infty} \int_{\mathcal{Q}_T}\left\langle x-\chi\left(\frac{v}{n}\right)vt\right\rangle^{-m}\langle v\rangle^{-m} \varphi,
\end{align*}where $C>0$ is the constant from Lemma \ref{uniform bound}. Note that
\begin{align*}
\left\langle x-\chi\left(\frac{v}{n}\right)vt\right\rangle^{-m}\to \langle x-vt\rangle^{-m}, \quad\textup{uniformly on each compact subset of $\mathcal{Q}_T$.}
\end{align*}Hence, by dominated convergence on the compact support of $\varphi$, we have
\begin{align*}
\int_{\mathcal{Q}_T} f\varphi \le C\int_{\mathcal{Q}_T} \langle x - vt\rangle^{-m}\langle v\rangle^{-m}\varphi.
\end{align*}Since this inequality holds for every non-negative $\varphi$, we apply elementary pointwise bounds for weak formulation   \cite[Exercise 4.26]{FA_Brezis_2011} to conclude
\begin{align*}
f\le C\langle x-vt\rangle^{-m}\langle v\rangle^{-m} \quad \textup{almost everywhere in }\,\,\mathcal{Q}_T.
\end{align*}To show non-negativity of $f$, we observe that for any $\varphi\in C_c^2(\mathcal{Q}_T)$ with $\varphi\ge0$:
\begin{align*}
0\le \lim_{n\to\infty}\int_{\mathcal{Q}_T} f^\textup{in}_{n^{-1}}I^{\frac{1}{n}}\varphi \le \int_{\mathcal{Q}_T}f\varphi.
\end{align*} Hence, $f\ge0$ almost everywhere in $\mathcal{Q}_T$. Combining the upper and lower bounds, we obtain $$0\le f\le C \langle x-vt\rangle^{-m}\langle v\rangle^{-m}.$$

Next, we verify that the limit $f$ belongs to either $W_{p_0}^{1,2}(\mathcal{Q}_T)$ or $W_{q_0,w}^{1,2}(\mathcal{Q}_T)$. Recall the global a priori estimates, \eqref{lem4.4eq2} and \eqref{lem4.4eq3}, and note that $f^n$ solves
\begin{align*}
    \mathcal{L}_{\frac{1}{n}, f^{n-1}}(f^n) = \mathcal{L}_{\frac{1}{n}, f^{n-1}}(f^\textup{in}_{n^{-1}})I^{1/n} \quad \textup{ in }\mathcal{Q}_T.
\end{align*}Thus, we have
\begin{align}\label{n Lp}
    \| f^n\|_{W_{p_0}^{1,2}(\mathcal{Q}_T)} &\lesssim \|\mathcal{L}_{1/n, f^{n-1}}(f^\textup{in}_{n^{-1}})I^{1/n}\|_{L_{p_0}(\Omega_T)},\\
    \| f^n\|_{W_{q_0,w}^{1,2}(\mathcal{Q}_T)} &\lesssim \|\mathcal{L}_{1/n, f^{n-1}}(f^\textup{in}_{n^{-1}})I^{1/n}\|_{L_{q_0,w}(\Omega_T)}.\label{n Lqw}
\end{align}From the general parabolic regularity theory, the implied constants in these estimates grow polynomially in $(\Lambda/\lambda)$ and $1/\lambda$, where
\begin{align*}
    0<\lambda \le (\textup{leading coefficient matrix}) \le \Lambda.
\end{align*}For the present sequence of problems, Lemma \ref{uniform bound} gives
\begin{align*}
    \lambda^{-1} = n \quad\textup{and}\quad (\Lambda/\lambda) = 1+n\|(-\Delta_v)^{-1}f^{n-1}\|_{L^\infty(\mathcal{Q}_T)}\lesssim n+1.
\end{align*}Define the growth exponents
\begin{align*}
    \kappa_1/p_0 \eqdef &\textup{ the maximum of the polynomial growth order in }\lambda^{-1} \textup{ and } (\Lambda/\lambda) \\
    &\textup{ arising from the global estimate } W_{p_0}^{1,2}(\Omega_T)\to L_{p_0}(\Omega_T),\\
    \kappa_2/q_0 \eqdef &\textup{ the maximum of the polynomial growth order in }\lambda^{-1} \textup{ and } (\Lambda/\lambda) \\
    &\textup{ arising from the global estimate } W_{q_0,w}^{1,2}(\Omega_T)\to L_{q_0,w}(\Omega_T),
\end{align*}and set $0<\kappa \eqdef \min\{\kappa_1, \kappa_2\}.$ Assume first that $\kappa = \kappa_1$, and recall the definition of the linear problem \eqref{semi}, in particular that the time-regularized characteristic function $I^\varepsilon$ converges at rate $\varepsilon^\kappa$.  Applying \eqref{n Lp} to  $f^n$ gives
\begin{align}\nonumber
    \|f^n\|_{W_{p_0}^{1,2}(\mathcal{Q}_T)} &\le \|f^n\|_{W_{p_0}^{1,2}(\Omega_T)}\\
    &\lesssim (n+1)^{\kappa/p_0} \| \mathcal{L}_{1/n,f^{n-1}}(f^\textup{in}_{n^{-1}})I^{1/(n+1)^{\kappa}}\|_{L_{p_0}(\Omega_T)}\nonumber\\
    &\lesssim (n+1)^{\kappa/ p_0}\left(\int_{-1/(n+1)^\kappa}^{2/(n+1)^{\kappa}} \|\mathcal{L}_{1/n,f^{n-1}}(f^\textup{in}_{n^{-1}})\|^{p_0}_{L_{p_0}(\mathbb{R}^6)}\,dt\right)^{1/p_0}\nonumber\\
    &\lesssim \sup_{n\in\mathbb{N}} \| \mathcal{L}_{1/n,f^{n-1}}(f^\textup{in}_{n^{-1}})\|_{L_{p_0}(\mathbb{R}^6)}.\label{n dep upper bd}
\end{align}The implicit constants here are independent of $n$. The final term in \eqref{n dep upper bd} is finite because $$\|(E_{f^{n-1}}, (-\Delta_v)^{-1}f^{n-1})\|_{L^\infty(\mathcal{Q}_T)}$$ is uniformly bounded in $n$ by Lemma \ref{uniform bound} and $\|f^\textup{in}_{n^{-1}}\|_{L_{p_0}} \lesssim_m \|f^\textup{in}_{n^{-1}}\|_{\mathcal{C}^2_{m,0}}\le\|f^\textup{in}\|_{\mathcal{C}^2_{m,0}}$ with same estimate for its derivatives. Consequently, $\{f^n\}_{n\ge0}$ is bounded in $W_{p_0}^{1,2}(\mathcal{Q}_T)$. By reflexivity, there exists a weakly convergent subsequence whose limit $g$ satisfies
\begin{align*}
    \int_{\mathcal{Q}_T} g \varphi= \lim_{n\to\infty} \int_{\mathcal{Q}_T}  f^n \varphi=\int_{\mathcal{Q}_T}f\varphi, \quad \forall \varphi\in C_c^\infty(\mathcal{Q}_T)\subset L_{p_0}(\mathcal{Q}_T),
\end{align*}so that $g=f$ almost everywhere in $\mathcal{Q}_T$. 

The case $\kappa=\kappa_2$ is treated analogously, but when we obtain the finiteness of the norm of $\mathcal{L}_{n^{-1},f^{n-1}}(f^\textup{in}_{n^{-1}})$ in $L_{q_0,w}$, notice that $L^\infty_{m,2m}(\mathbb{R}^6) \hookrightarrow L_{q_0,w}(\mathbb{R}^6)$ holds for $m>3$, $3<q_0<\frac{11}{3}$ and $w=\langle x\rangle\langle v\rangle^{8}$. Hence, $\{f^n\}_{n\ge0}$ is bounded in $W_{q_0,w}^{1,2}(\mathcal{Q}_T)$ and its weak limit belongs to this weighted Sobolev space.

Combining these two inclusions, we conclude that  the limit function $f$ belongs to the space $X_T^m$.

\textbf{Step 2: Verification of the weak formulation in the limit. }We now show that the limit function $f$, obtained in the previous subsection above, satisfies the weak formulation defined in Definition \ref{Weak formulation}, with dependence on the initial data $f^\textup{in}$. Let $\varphi\in C_c^2([0,T)\times\mathbb{R}^6)$ be a smooth, compactly supported test function. Multiplying the approximate equation \eqref{n approx} by $\varphi$ and integrating over $\mathcal{Q}_T$, we obtain:
\begin{multline*}
\int_{\mathcal{Q}_T}\varphi\left[\partial_tf^n + \chi\left(\frac{v}{n}\right)v\cdot\nabla_xf^n + J^{\frac{1}{n}}(t)\chi(x/n)E_{f^{n-1}}\cdot\nabla_vf^n\right]\,dxdvdt\\
=\int_{\mathcal{Q}_T} \varphi\left[(-\Delta_v)_{n^{-1}}^{-1}f^{n-1}\Delta_vf^n + f^{n-1}f^n +\frac{1}{n}\Delta_{x,v}f^n + \mathcal{L}_{\frac{1}{n}, f^{n-1}}(f^\textup{in}_{n^{-1}})I^{\frac{1}{n}}\right]\,dxdvdt.
\end{multline*}To verify convergence to the weak formulation, we analyze the left-hand and right-hand sides by separating them into linear and nonlinear components. 
\begin{itemize}
    \item Linear terms: 
\begin{align*}
\partial_tf^n, \quad \chi\left(\frac{v}{n}\right)v\cdot\nabla_xf^n, \quad (1/n)\Delta_{x,v}f^n, \quad \textup{and} \quad \mathcal{L}_n(f^\textup{in}_{n^{-1}})I^{\frac{1}{n}}.
\end{align*} \item Nonlinear terms: 
\begin{align*}
\chi(x/n)E_{f^{n-1}}\cdot\nabla_v f^n,\quad (-\Delta_v)_{n^{-1}}^{-1}f^{n-1} \Delta_vf^n, \quad \textup{and}\quad  f^{n-1}f^n.
\end{align*}
\end{itemize}
We now verify that the linear terms in the approximate equations converges to their counterparts in the weak formulation as $n\to\infty$, using the test function $\varphi\in C_c^2([0,T)\times\mathbb{R}^6)\subset L^1(\mathcal{Q}_T)$.

\noindent\textbf{(L1) Time derivative term:} We first examine the time derivative term:
\begin{multline*}
\lim_{n\to\infty}\int_0^T \iint_{\mathbb{R}^3\times\mathbb{R}^3}\varphi \partial_tf^n\,dxdvd\tau 
= \lim_{n\to\infty}\left[ - \iint \varphi(0)f^\textup{in}_{n^{-1}}\,dxdv - \int_0^T \iint \partial_t\varphi f^n\,dxdvd\tau\right]\\
= - \iint \varphi(0)f^\textup{in}\,dxdv - \int_0^T \iint \partial_t\varphi f\,dxdvd\tau.
\end{multline*}Here, in the first equality, we used $f^\textup{in}_{n^{-1}}\to f^\textup{in}$ as $n\to\infty$ in $L^\infty_{x,v}$ and in the second equality, we used the definition of the weak*-convergence $f^n \rightharpoonup f$ in $L^\infty(\mathcal{Q}_T)$.

\noindent\textbf{(L2) Advection term:} Next, consider the transport term:$$
\lim_{n\to\infty} \int_0^T \iint_{\mathbb{R}^3\times\mathbb{R}^3} \varphi \chi(v/n)v\cdot\nabla_xf^n\,dxdvd\tau 
= - \lim_{n\to\infty} \int_0^T \iint f^n \chi(v/n)v\cdot\nabla_x\varphi\,dxdvd\tau.$$Since $\chi(v/n)v\cdot\nabla_x\varphi \equiv v\cdot\nabla_x\varphi$ for large $n$, we have:
\begin{align*}
= -\int_0^T\iint fv\cdot\nabla_x\varphi\,dxdvd\tau.
\end{align*}

\noindent\textbf{(L3) Diffusion term: }The artificial term vanishes in the limit:
\begin{align*}
\lim_{n\to\infty} \int_0^T \iint_{\mathbb{R}^3\times\mathbb{R}^3} \frac{1}{n}\Delta_{x,v}f^n \varphi\,dxdvd\tau &= \lim_{n\to\infty} \int_0^T \iint \frac{1}{n}f^n \Delta_{x,v}\varphi\,dxdvd\tau\\
&=0,
\end{align*}since $\Delta_{x,v}\varphi\in C_c(\mathcal{Q}_T)$ and $\{f^n\}$ is uniformly bounded in $L^\infty(\mathcal{Q}_T)$.

\noindent\textbf{(L4) Inhomogeneous source term: }Finally, we analyze the inhomogeneous source term:
\begin{align*}
\int_0^T \iint_{\mathbb{R}^6} \varphi \mathcal{L}_{\frac{1}{n}, f^{n-1}}(f^\textup{in}_{n^{-1}})I^n\,dxdvd\tau \to 0 \quad \textup{ as }\,\,n\to\infty,
\end{align*}because the integrand converges to zero pointwisely and is uniformly bounded by an integrable function due to the decay of $\mathcal{L}_{\frac{1}{n}, f^{n-1}}(f^\textup{in}_{n^{-1}})$. Hence, the dominated convergence theorem applies.

Thus, all linear terms in the approximate equation converge to their respective terms in the weak formulation satisfied by $f$.

Before passing to the limit in the nonlinear terms, we obtain the following convergence from the dominated convergence theorem:
\begin{align}
(-\Delta_v)_{n^{-1}}^{-1}f^n \to (-\Delta_v)^{-1}f, \quad &\textup{in }\mathcal{Q}_T,\label{inv Lapl lim}\\
\nabla_v(-\Delta_v)_{n^{-1}}^{-1}f^n \to \nabla_v(-\Delta_v)^{-1}f, \quad &\textup{in }\mathcal{Q}_T,\label{d in Lapl limit}\\
E_{f^n} \to E_f, \quad &\textup{in }\mathcal{Q}_T.\label{E lim}
\end{align}To prove these convergences, we use (i) $f^n \to f$ a.e. in $\mathcal{Q}_T$ up to a subsequence, and (ii) $|f^n(t,x,v)|\le C\langle x\rangle^{-m}\langle v\rangle^{-m}$. Since (ii) is established in Step 1, it remains only to prove (i). The latter follows from the Rellich-Kondrachov theorem together with local compactness on bounded cylinders and a diagonal argument.

Now, we inspect the nonlinearity term by term: 

\noindent\textbf{(N1) Collision terms: }We analyze the term 
\begin{align*}
(-\Delta_v)_{n^{-1}}^{-1}f^{n-1}\Delta_vf^n + f^n f^{n-1}.
\end{align*}Testing against $\varphi\in C_c^2([0,T)\times\mathbb{R}^6)$, we compute:
\begin{multline*}
\int_0^T\iint_{\mathbb{R}^6} \varphi[(-\Delta_v)_{n^{-1}}^{-1}f^{n-1}\Delta_vf^n + f^{n-1}f^n]\,dxdvd\tau\\
= \int_0^T\iint_{\mathbb{R}^6} f^n \Delta_v[\varphi(-\Delta_v)_{n^{-1}}^{-1}f^{n-1}] + \varphi f^{n-1}f^n \,dxdvd\tau\\
= \int_0^T\iint_{\mathbb{R}^6} f^n \Delta_v\varphi(-\Delta_v)_{n^{-1}}^{-1}f^{n-1} \,dxdd\tau
+ \int_0^T\iint_{\mathbb{R}^6}2f^n\nabla_v\varphi\cdot\nabla_v(-\Delta_v)_{n^{-1}}^{-1}f^{n-1}\,dxdvd\tau \\
+\int_0^T \varphi f^{n-1}f^n(1-\chi(v/n))\,dxdvd\tau=: I + II + III.
\end{multline*}For the term $I$, we further decompose it as 
\begin{multline*}
I - \int_0^T\iint_{\mathbb{R}^6} f\Delta_v\varphi(-\Delta_v)^{-1}f\,dxdvd\tau\\
= \int_0^T\iint_{\mathbb{R}^6} (f^n - f)\Delta_v\varphi(-\Delta_v)_{n^{-1}}^{-1}f^{n-1}\,dxdvd\tau \\
+ \int_0^T \iint_{\mathbb{R}^6} f\Delta_v\varphi(-\Delta_v)^{-1}(\chi(v/n)f^{n-1} - f)\,dxdvd\tau =: I_1 + I_2.
\end{multline*}For $I_1$, since $(-\Delta_v)_{n^{-1}}^{-1}f^{n-1}$ is bounded in $L^\infty$ uniformly in $n\in\mathbb{N}$ due to the estimate
\begin{align*}
|(-\Delta_v)^{-1}f^n(t,x,v)| \le \frac{C}{4\pi} \int_{\mathbb{R}^3} \frac{du}{|u-v|\langle u\rangle^{m}}<\infty,
\end{align*}(by decay properties from Lemma \ref{uniform bound}), and since $f^n \to f$ a.e. with uniform integrable bound, the dominated convergence theorem provides $I_1 \to 0$ as $n\to\infty$. For $I_2$, by the convergence \eqref{inv Lapl lim} and the compact support of $\Delta_v\varphi$, the dominated convergence again applies, giving $I_2\to0$. Therefore,
\begin{align}\label{Colls1 limit}
I \to \int_0^T \iint_{\mathbb{R}^6} f(-\Delta_v)^{-1}f\Delta_v\varphi\,dxdvd\tau, \quad n\to\infty.
\end{align}

For the term $II$, we again decompose the term as 
\begin{multline*}
\frac{1}{2}\left(II- 2\int_0^T \iint_{\mathbb{R}^6}f\nabla_v\varphi\cdot\nabla_v(-\Delta_v)^{-1}f\,dxdvd\tau\right) \\
= \int_0^T \iint_{\mathbb{R}^6} (f^n - f)\nabla_v\varphi\cdot\nabla_v(-\Delta_v)_{n^{-1}}^{-1}f^{n-1} + f\nabla_v\varphi\cdot\nabla_v(-\Delta_v)^{-1}[\chi(v/n)f^{n-1}-f]\,dxdvd\tau\\
=: II_1 + II_2.
\end{multline*}For the difference $II_1$, using again that $\nabla_v(-\Delta_v)^{-1}f^{n-1}$ is bounded in $L^\infty$ uniformly in $n\in\mathbb{N}$ and $f^n \to f$ a.e., we conclude $II_1\to0$ via dominated convergence. Next, for $II_2$, the convergence \eqref{d in Lapl limit} and the boundedness of $f$ and $\nabla_v\varphi$ on the compact support again ensures $II_2\to0$ via dominated convergence. Hence,
\begin{align}\label{Colls2 limit}
II \to 2\int_0^T\iint_{\mathbb{R}^6} f\nabla_v\varphi\cdot\nabla_v(-\Delta_v)^{-1}f\,dxdvd\tau.
\end{align}

For the term $III$, since $\varphi(1-\chi(v/n))= 0$ for sufficiently large $n$, we conclude that by the dominated convergence again, $III\to0$ as $n\to\infty$.

Combining \eqref{Colls1 limit}, \eqref{Colls2 limit} and the fact that $III\to0$ as $n\to\infty$, we find that the nonlinear collision term converges to its weak limit:
\begin{multline*}
\int_0^T\iint_{\mathbb{R}^6} \varphi\left[(-\Delta_v)_{n^{-1}}^{-1}f^{n-1}\cdot\Delta_vf^n + f^{n-1}f^n\right]\,dxdvd\tau\\
\to\int_0^T \iint_{\mathbb{R}^6} f\cdot(-\Delta_v)^{-1}f\cdot\Delta_v\varphi + 2f\nabla_v\varphi\cdot\nabla_v(-\Delta_v)^{-1}f\,dxdvd\tau.
\end{multline*}Thus it verifies convergence of the nonlinear terms in the weak formulation.

\noindent\textbf{(N2) Acceleration term:} We now analyze the nonlinear acceleration term
\begin{align*}
\chi(x/n)E_{f^{n-1}} \cdot\nabla_vf^n.
\end{align*}Testing this term against $\varphi\in C_c^2([0,T)\times\mathbb{R}^6)$, we consider the limit
\begin{align*}
\int_0^T\iint_{\mathbb{R}^6} \chi(x/n)E_{f^{n-1}}\cdot\nabla_vf^n \varphi\,dxdvd\tau- \int_0^T\iint_{\mathbb{R}^6} E_f\cdot\nabla_vf \varphi\,dxdvd\tau.
\end{align*}Integrating by parts in the velocity variable, and using that for sufficiently large $n$, $\chi(x/n)\varphi \equiv \varphi$, we write:
\begin{align*}
= -\int_0^T \iint_{\mathbb{R}^6} (f^n E_{f^{n-1}} - fE_f)\cdot\nabla_v\varphi\,dxdvd\tau.
\end{align*}Splitting the integrand:
\begin{equation*}
= -\int_0^T \iint_{\mathbb{R}^6} (f^n - f)E_{f^{n-1}}\cdot\nabla_v\varphi + f(E_{f^{n-1}} - E_f)\cdot\nabla_v\varphi\,dxdvd\tau=:\textit{III}_1 + \textit{III}_2.
\end{equation*}For $\textit{III}_1$, since $E_{f^{n-1}}$ is bounded in $L^\infty$ uniformly in $n$ by \eqref{E bound}, we have $\textit{III}_1 \to0$ by dominated convergence. For $\textit{III}_2$, we use the convergence:
\begin{align*}
E_{f^{n-1}}\to E_f \quad \textup{ a.e. in }\,\mathcal{Q}_T,
\end{align*}as established in \eqref{E lim}. Moreover, since $f$ is fixed and bounded (with decay), and $\nabla_v\varphi$ is compactly supported, we again apply the dominated convergence theorem to obtain $\textit{III}_2\to0$.

Combining the two limits, we conclude:
\begin{align*}
\int_0^T\iint_{\mathbb{R}^6} \chi(x/n)E_f^{n-1}\cdot\nabla_vf^n \varphi\,dxdvd\tau \to \int_0^T\iint_{\mathbb{R}^6} fE_f\cdot\nabla_v\varphi\,dxdvd\tau.
\end{align*}

Therefore, we have shown that the limit $f$ is a weak solution to the initial value problem of \eqref{VPiL} for $f^\textup{in}\in \mathcal{C}^2_{m,0}(\mathbb{R}^6)$ with $f^\textup{in}\ge0$.\end{proof}

\begin{remark}The solution $f$ constructed above is never identically equal to zero for any $t>0$, unless the initial data satisfies $f^\textup{in}\equiv0$. This follows from the structure of the weak formulation in Definition \ref{Weak formulation}.
\end{remark}

\begin{remark}
    Although $f^\textup{in}\in C^{2,\alpha}(\mathbb{R}^6)$, the same argument as in \eqref{n dep upper bd} does not yield the $C^{2,\alpha}_{\textup{para}}$ regularity of $f$, since 
$$
\|I^\varepsilon\|_{C^{\alpha/2}_t} \to \infty \quad \textup{as } \varepsilon \to 0 .
$$

\end{remark}

\section{Singularity Formation}
\label{Singularity formation}

In this section, we prove Theorem \ref{thm.main2} and Corollary \ref{collapse}. Throughout, we assume the hypotheses of Theorem \ref{thm.main2} and let $f$ denote the corresponding solution to \eqref{VPiL}. We also recall from Theorem \ref{thm.main2} that gravitational interaction is given by the Coulomb-type potential $K(x) = -\frac{1}{4\pi|x|}$.

In Lemma \ref{basic upper bound}, we establish the $L^\infty$--boundedness of the coefficient appearing in the potential, as well as the boundedness of its derivatives, namely the gravitational field $E_f$.

\begin{lemma}\label{basic upper bound} For each $t\in[0,T]$, we have
\begin{align*}
\nabla_xK\star_x \rho, \quad K\star_x \rho\in L^\infty ([0,T] \times \mathbb{R}^3),
\end{align*}where $\star_x$ denotes the convolution in the spatial variable.
\end{lemma}

\begin{proof}
This follows from Lemma \ref{Linfty.est.sing.int} and the integrability of the local mass $\rho\in L^\infty([0,T]; (L^1\cap L^\infty)(\mathbb{R}^3))$, which is ensured by the decay condition imposed on $f\in X_T^m$.
\end{proof}

Before proceeding, we establish rigorously the differentiability of the relevant real-valued functionals associated with $f$.

\begin{lemma}\label{intch}
    We have
    \begin{align*}
        \frac{d}{dt}\iint_{\mathbb{R}^6} 
        \begin{pmatrix}
        1\\
        |x|^2\\
        |v|^2\\
        x\cdot v
        \end{pmatrix}f(t,x,v)\,dxdv &= \iint_{\mathbb{R}^6} \begin{pmatrix}
        1\\
        |x|^2\\
        |v|^2\\
        x\cdot v
        \end{pmatrix}\partial_tf(t,x,v)\,dxdv, \quad \textup{for a.e. } t\in[0,T],
    \end{align*}and for each fixed $x\in\mathbb{R}^3$ we have for a.e. $t\in[0,T]$ that
    \begin{align*}
        \frac{\partial}{\partial t}\rho(t,x) &= \int_{\mathbb{R}^3} \partial_tf(t,x,v)\,dv,\\
        \nabla_x \cdot j(t,x) &= \int_{\mathbb{R}^3} v\cdot\nabla_xf(t,x,v)\,dv.
    \end{align*}
\end{lemma}

\begin{proof}
    Let $w(x,v)\in \{1, |x|^2, |v|^2, x\cdot v\}$ and let $\bar{f} \eqdef wf$. Since $f\in X_T^m\cap A_2$, we have $|\partial_t\bar{f}| \le wg_1$ for all $(t,x,v)$, where $wg_1\in L^\infty_{x,v}$ with $g_1$ bounded. Then for $(x,v)$, we have that $\partial_t\bar{f}(\cdot,x,v)$ is bounded in $[0,T]$, hence for the mapping $t\mapsto \bar{f}(t,x,v)$, the fundamental theorem of calculus holds
    \begin{align*}
        \bar{f}(t+h,x,v) - \bar{f}(t,x,v) = \int_t^{t+h} \partial_{\tau}\bar{f}(\tau,x,v)\,d\tau \quad \textup{for } t,t+h\in[0,T].
    \end{align*}By integrating the above identity, we obtain for $t,t+h\in[0,T]$,
    \begin{align}\label{ftc barf}
        \frac{1}{h}\iint_{\mathbb{R}^6} \bar{f}(t+h,x,v) - \bar{f}(t,x,v)\,dxdv = \iint_{\mathbb{R}^6}\frac{1}{h}\left(\int_t^{t+h}\partial_{\tau}\bar{f}(\tau,x,v)\,d\tau\right)\,dxdv.
    \end{align}Indeed, by the Lebesgue differentiation theorem and dominated convergence theorem, the right-hand side converges to $\iint_{\mathbb{R}^6} \partial_t\bar{f}(t,x,v)\,dxdv$ as $h\to0$ for a.e. $t\in[0,T]$. The integrability of $\partial_t \bar{f}$ is guaranteed by $|\partial_t\bar{f}|\le |w|g_1\in L^1_{x,v}$. Therefore, this establishes the first claim.

    The remaining assertions follow by the same argument, using the corresponding domination assumptions (e.g. $g_1\in L^\infty_xL^1_v$ and $g_2\in (\dot{L}^\infty_1\cap\dot{L}^1_{1})_v$) for the relevant weighted quantities. This completes the proof.
\end{proof}

In Lemma \ref{mass conservation} and \ref{continuity eq}, we show two fundamental conservation laws.

\begin{lemma}[Mass conservation]\label{mass conservation} It holds that
\begin{align*}
\sup_{t\in[0,T]}\lVert \rho(t)\rVert_{L^1_x} = \iint_{\mathbb{R}^6} f^\textup{in}(x,v)\,dvdx<\infty.
\end{align*}
\end{lemma}

\begin{proof}
    In view of \eqref{ftc barf} with $w=1$, it suffices to show that $\iint_{\mathbb{R}^6}\partial_tf(t,x,v)\,dxdv = 0$ for a.e. $t\in[0,T]$, which implies the conservation of the total mass.
    
    Since $f$ satisfies the equation \eqref{VPiL} for a.e. in $\mathcal{Q}_T$, we may substitute the equation into the integrand $\partial_tf$ to obtain for a.e. $t\in[0,T]$,
    \begin{align}\label{L1 int of eq}
        \iint_{\mathbb{R}^6} \partial_tf(t)\,dvdx = \iint_{\mathbb{R}^6} -v\cdot\nabla_xf - E_f\cdot\nabla_vf + (-\Delta_v)^{-1}f\Delta_vf + f^2\,dvdx.
    \end{align}By the assumptions $f\in X_T^m$, together with \eqref{Dtf dom} and \eqref{vDxf Dvf L1w}, each term on the right-hand side of \eqref{L1 int of eq} belongs to $L^1(\mathbb{R}^6)$. Hence, all the integrals are well-defined, and linearity of integration applies. 
    
    We now evaluate each term. The first two terms vanish by the divergence theorem for Sobolev functions in $x$ and $v$, respectively, since the associated fluxes are integrable and decay at infinity. The third term vanishes by integration by parts for Sobolev functions in $v$, noting that $(-\Delta_v)^{-1}f\Delta_vf\in L^1(\mathbb{R}^6)$ for a.e. $t\in[0,T]$. This integrability follows from the assumption $f\in X_T^m$ together with \eqref{Dtf dom} and \eqref{vDxf Dvf L1w}. Finally, the remaining term $f^2$ is integrable and cancels with the corresponding contribution from the collision operator.

    Consequently, the right-hand side of \eqref{L1 int of eq} vanishes, and therefore $\iint_{\mathbb{R}^6} \partial_tf(t,x,v)\,dxdv = 0$ for a.e. $t\in[0,T]$. This completes the proof.
\end{proof}

\begin{lemma}[Continuity equation]\label{continuity eq} 
It holds that, for each fixed $x\in\mathbb{R}^3$,
\begin{align}\label{cont eq}
    \partial_t \rho(t,x)+\nabla\cdot j(t,x)=0, \quad \textup{ for a.e. } t\in [0,T].
\end{align}
\end{lemma}

\begin{proof}
    Throughout this proof, fix $x\in\mathbb{R}^3$. By Lemma \ref{intch}, integrating equation \eqref{VPiL} with respect to $v$ yields
    \begin{align*}
        \partial_t\rho + \nabla_x\cdot j = \int_{\mathbb{R}^3} -E_f\cdot\nabla_vf + (-\Delta_v)^{-1}f\Delta_vf + f^2\,dv, \quad \text{ for a.e. }t\in[0,T].
    \end{align*}By the same argument as in the proof of Lemma \ref{mass conservation}, each term on the right-hand side is integrable in $v$ and vanishes by applying the divergence theorem to the first term and integration by parts to the second term. Consequently, the right-hand side is identically zero, which proves \eqref{cont eq}.
\end{proof}

Now, recall from either Definition \ref{phys quantity} or Section \ref{notations} that  the functional $\textup{I}(t) $ is defined as
\begin{align*}
\textup{I}(t) = \frac{1}{2}\iint_{\mathbb{R}^6} |x|^2f(t,x,v)\,dxdv.
\end{align*} In Lemma \ref{I' I''}, we derive explicit formulas for $\textup{I}'(t)$ and $\textup{I}''(t)$, respectively.

\begin{lemma}\label{I' I''} The first and second derivatives of $\textup{I}(t)$ are given by
\begin{align*}
\textup{I}'(t) &= \iint_{\mathbb{R}^6} (x\cdot v)f(t,x,v)\,dxdv, \quad \textup{for a.e. }t\in[0,T],\\
\textup{I}''(t) &= \iint_{\mathbb{R}^6} |v|^2f\,dxdv+\frac{1}{2}\iint_{\mathbb{R}^6} (K\star_x \rho) f \,dxdv \quad \textup{for a.e. }t\in[0,T].
\end{align*}
\end{lemma}

\begin{proof}
    By Lemma \ref{intch}, the function $\textup{I}(t)$ is differentiable for a.e. $t\in[0,T]$, and its first derivative is given by $\textup{I}'(t) = \frac{1}{2}\iint_{\mathbb{R}^6} |x|^2\partial_tf(t,x,v)\,dxdv$ for a.e. $t$. Substituting equation \eqref{VPiL} into this expression $\iint_{\mathbb{R}^6} |x|^2\partial_tf\,dxdv$ and applying the divergence theorem together with integration by parts, under the condition $m>6$ and the admissible assumptions \eqref{Dtf dom} and \eqref{vDxf Dvf L1w}, we obtain
    \begin{align*}
    \begin{split}
    \textup{I}'(t)&=\frac{1}{2}\iint_{\mathbb{R}^6} |x|^2\partial_tf(t,x,v)\,dvdx\\
    &= \frac{1}{2}\iint_{\mathbb{R}^6} |x|^2 \left[-v\cdot \nabla_x f +(\nabla_xK\star_x \rho)\cdot \nabla_vf + (-\Delta_v)^{-1}f \Delta_vf+f^2\right]\,dvdx\\
    &= \iint_{\mathbb{R}^6} (x\cdot v)f(t,x,v)\,dvdx.
    \end{split}
    \end{align*}

    We now turn to the second derivative $\textup{I}''(t)$. By Lemma \ref{intch} again, $\textup{I}'(t)$ is differentiable for a.e. $t\in[0,T]$ and its derivative is given by $\iint_{\mathbb{R}^6}(x\cdot v)\partial_tf\,dxdv$. Substituting equation \eqref{VPiL} into this expression and applying integration by parts and divergence theorem under the admissible assumptions \eqref{Dtf dom} and \eqref{vDxf Dvf L1w}, we arrive at a representation of $\textup{I}''(t)$:
    \begin{align}
    \textup{I}''(t)&=\iint_{\mathbb{R}^6} (x\cdot v)\left[-v\cdot \nabla_xf+(\nabla_xK\star_x \rho)\cdot \nabla_vf+(-\Delta_v)^{-1}f\Delta_vf+f^2\right]\,dvdx\nonumber\\ 
    &=\iint_{\mathbb{R}^6} |v|^2f-x\cdot(\nabla_xK\star_x \rho)f+2(x\cdot \nabla_v(-\Delta_v)^{-1}f)f\,dvdx. \label{I'' rep}
    \end{align}To simplify \eqref{I'' rep}, we first establish the following identity
    \begin{align}\label{0 identity}
    \iint_{\mathbb{R}^6} (x\cdot\nabla_v(-\Delta_v)^{-1}f)f\,dxdv=0.
    \end{align}Indeed, the integrand $(x\cdot\nabla_v(-\Delta_v)^{-1}f)f$ is integrable since $\nabla_v(-\Delta_v)^{-1}f\in L^\infty_{t,x,v}$ and $xf(t)\in L^1_{x,v}$. Exchanging the velocity variables in the double integral shows that the expression is antisymmetric, and hence equal to its own negative, which implies \eqref{0 identity}.

    Next, we evaluate the force term. Using the identity $\nabla_xK(x) = +\frac{x}{4\pi|x|^3}$, we rewrite the integral involving $x\cdot(\nabla_xK\star \rho)f$ as a double integral in the spatial variables. By symmetry and an application of Fubini's theorem, the mixed term cancels, leaving the expression:
    \begin{align}\nonumber
    \iint_{\mathbb{R}^6} x\cdot(\nabla_xK\star_x\rho)f\,dvdx&=\frac{1}{8\pi} \iint_{\mathbb{R}^6} \frac{\rho(y)}{|x-y|}\,dy\rho(x)\,dx\\
    &=-\frac{1}{2} \int_{\mathbb{R}^3} (K\star_x\rho)\rho\,dx.\label{second term expression}
    \end{align}

    Substituting the expression \eqref{0 identity} and \eqref{second term expression} into \eqref{I'' rep}, we obtain the desired formula for $\textup{I}''$.
\end{proof}

In Lemma \ref{I'''}, we derive an explicit expression for $\textup{I}'''$. Lemma \ref{field E diff} is then devoted to the evaluation of one of the terms arising in $\textup{I}'''(t)$.

\begin{lemma}\label{field E diff}
    We have for a.e. $t\in[0,T]$
    \begin{align*}
        \frac{d}{dt}\int_{\mathbb{R}^3} \rho(t,x)(K\star\rho)(t,x)\,dx = 2\int_{\mathbb{R}^3} (\nabla K\star\rho)(t,x)\cdot j(t,x)\,dx.
    \end{align*}
\end{lemma}

\begin{proof}
    Firstly, we show
    \begin{align}\label{field energy diff}
        \frac{d}{dt}\int_{\mathbb{R}^3} \rho(K\star\rho)(t,x)\,dx = \int_{\mathbb{R}^3} \partial_t\rho(K\star\rho)(t,x) + \rho(K\star\partial_t\rho)(t,x)\,dx \quad \text{for a.e. }t\in[0,T].
    \end{align}To this end, recall from Lemma \ref{intch} that
    \begin{align*}
        \rho(t,y) - \rho(s,y) = \int_s^t \partial_t\rho(\tau,y)\,d\tau \quad \textup{ for } t,s\in[0,T], \textup{ for }y\in\mathbb{R}^3.
    \end{align*}Multiplying this identity by $K(x-y)$ and integrating with respect to $y$, we obtain
    \begin{align}\label{ftc convol rho}
        (K\star\rho)(t,x) - (K\star\rho)(s,x) = \int_s^t (K\star\partial_t\rho)(\tau,x)\,d\tau \quad \textup{ for }t,s, \textup{ for }x.
    \end{align}Consequently, by the Lebesgue differentiation theorem,
    \begin{align}\nonumber
        \frac{1}{h}\left[(K\star\rho)(t+h,x) - (K\star\rho)(t,x)\right] &= \frac{1}{h}\int_t^{t+h} (K\star\partial_t\rho)(\tau,x)\,d\tau\\
        &\to K\star\partial_t\rho(t,x),\label{convol rho diff}
    \end{align} as $h\to0$ for a.e. $t$ for each fixed $x$. We now compute
    \begin{multline*}
        \frac{1}{h}\int_{\mathbb{R}^3} \rho(t+h)(K\star\rho)(t+h) - \rho(t)(K\star\rho)(t) \,dx\\
        =\int_{\mathbb{R}^3} \rho(t+h)\left[\frac{(K\star\rho)(t+h) - (K\star\rho)(t)}{h}\right]\,dx + \int_{\mathbb{R}^3} \left[\frac{\rho(t+h) - \rho(t)}{h}\right](K\star\rho)(t)\,dx\\
        \eqdef D_1 + D_2.
    \end{multline*}
    By \eqref{ftc convol rho}, \eqref{convol rho diff} and the dominated convergence theorem, we have for a.e. $t\in[0,T]$
    \begin{align*}
        D_1\to \int_{\mathbb{R}^3} \rho(t)(K\star\partial_t\rho)(t)\,dx , \quad D_2\to\int_{\mathbb{R}^3} \partial_t\rho(t)(K\star\rho)(t)\,dx.
    \end{align*}Here the convergence of $D_1$ is justified by the domination $\|\rho(x)\|_{L^\infty_t}\|K\star \|g_1(x)\|_{L^1_v}\|_{L^\infty_x}\in L^1_x$, while the convergence of $D_2$ follows from the domination $\|g(x)\|_{L^1_v}\|K\star\rho\|_{L^\infty_{t,x}}\in L^1_x$ for $D_2 \to\int \partial_t\rho (K\star\rho)\,dx$. This proves \eqref{field energy diff}.

    Next, we simplify the right-hand side of \eqref{field energy diff} using the continuity equation \eqref{cont eq}. By Fubini's theorem and the symmetry of the kernel, the two terms on the right-hand side coincide, so it suffices to consider $\int_{\mathbb{R}^3} \partial_t\rho(K\star\rho)\,dx$. Invoking \eqref{cont eq}, we obtain for a.e. $t\in[0,T]$
    \begin{align*}
        (\textup{RHS of }\eqref{field energy diff}) &= -2\int_{\mathbb{R}^3} \nabla\cdot j(t,x) (K\star\rho)(t,x)\,dx\\
        &= 2\int_{\mathbb{R}^3} j(t,x)\cdot (\nabla K\star\rho)(t,x)\,dx,
    \end{align*}which completes the proof.
\end{proof}

\begin{lemma}The third derivative $\textup{I}'''$ of $\textup{I}(t)$ exists for a.e. $t\in[0,T]$ and admits the representation $$\textup{I}'''(t)=\frac{d}{dt}\textup{I}''(t)=\iint_{\mathbb{R}^6} v\cdot(\nabla_xK\star_x\rho)f\,dvdx+4\iint_{\mathbb{R}^6} (-\Delta_v)^{-1}f\,f\,dvdx,$$for a.e. $t$.
In particular, by invoking Lemma \ref{continuity eq}, we obtain for a.e. $t$,\label{I'''}
\begin{align}\label{Etot id}
\frac{d}{dt}\left(\textup{I}''(t) + \frac{1}{2} \int_{\mathbb{R}^3} \rho K\star_x \rho\,dx\right)&=4\iint_{\mathbb{R}^6} (-\Delta_v)^{-1}f f\,dvdx.
\end{align}
\end{lemma}

\begin{proof}
In view of Lemma \ref{field E diff}, it remains to compute the time derivative of the kinetic energy $\textup{KE}(t)$. By Lemma \ref{intch} and equation \eqref{VPiL}, we obtain
\begin{multline*}
\frac{d}{dt}\left(\iint_{\mathbb{R}^6} |v|^2f(t,x,v)\,dxdv\right)= \iint_{\mathbb{R}^6} |v|^2\partial_tf\,dxdv\\
=\iint_{\mathbb{R}^6} |v|^2[-v\cdot\nabla_xf+(\nabla_xK\star_x\rho)\cdot \nabla_vf+(-\Delta_v)^{-1}f\Delta_vf+f^2]\,dvdx\\
=-2\iint_{\mathbb{R}^6} v\cdot(\nabla_xK\star_x\rho)f\,dvdx+\iint_{\mathbb{R}^6} |v|^2[(-\Delta_v)^{-1}f\Delta_vf+f^2]\,dvdx.
\end{multline*}Consequently,
\begin{align}\label{E'}
\frac{d}{dt}\left(\textup{I}''(t) + \frac{1}{2} \int_{\mathbb{R}^3} \rho K\star_x \rho\,dx\right)=\iint_{\mathbb{R}^6} |v|^2[(-\Delta_v)^{-1}f\Delta_vf+f^2]\,dvdx.
\end{align}

We now evaluate the right-hand side of \eqref{E'}. A direct computation yields
\begin{multline*}
\iint_{\mathbb{R}^6} |v|^2[(-\Delta_v)^{-1}f\Delta_vf+f^2]\,dvdx 
= 6\iint_{\mathbb{R}^6} (-\Delta_v)^{-1}f\,f\,dvdx+4\iint_{\mathbb{R}^6} v\cdot\nabla_v(-\Delta_v)^{-1}f\,f\,dvdx,
\end{multline*}The remaining term can be further simplified as follows:
\begin{multline*}
\int_{\mathbb{R}^3} v\cdot\nabla_v(-\Delta_v)^{-1}f\,f\,dv=\frac{1}{4\pi}\iint_{\mathbb{R}^6} \frac{v\cdot(u-v)}{|u-v|^3}f(u)f(v)\,dudv\\
=\frac{1}{4\pi}\left[-\iint_{\mathbb{R}^6}\frac{f(u)f(v)}{|u-v|}\,dudv+\iint_{\mathbb{R}^6}\frac{u\cdot(u-v)}{|u-v|^3}f(u)f(v)\,dudv\right]\\
=\frac{1}{4\pi}\left[-\iint_{\mathbb{R}^6}\frac{f(u)f(v)}{|u-v|}\,dudv-\iint_{\mathbb{R}^6} u\cdot\nabla_u\left(\frac{1}{|u-v|}\right)f(u)f(v)\,dudv\right]\\
=-\int_{\mathbb{R}^3} (-\Delta_v)^{-1}f\,f\,dv-\int_{\mathbb{R}^3} u\cdot\nabla_u(-\Delta_u)^{-1}f(u)f(u)\,du
=-\frac{1}{2}\int_{\mathbb{R}^3} (-\Delta_v)^{-1}f\,f\,dv.
\end{multline*}Putting all the terms together yields the desired formula.
\end{proof}

We now have all the necessary identities at hand and are ready to establish a
local $L^\infty$ estimate as a corollary of Lemma \ref{uniform bound}.

\begin{lemma}\label{Linfty est f}
    We have \begin{align}\label{global Linfty}
       \sup_{t\in[0,T]} \| \langle x\rangle^m\langle v\rangle^mf(t)\|_{L^\infty(\mathbb{R}^6)}\le \frac{C}{\|f^\textup{in}\|_{\mathcal{C}^2_{m,0}}},
    \end{align}where $C>0$ is a constant depending only on $m$.
\end{lemma}

\begin{proof}
    The lemma follows from Lemma \ref{uniform bound}.
\end{proof}

\begin{remark}\label{E' est}By Lemma \ref{mass conservation} and \ref{Linfty est f}, there exists a constant $C_1>0$ depending only on $m$, $T$, $\|f^\textup{in}\|_{L^1_{x,v}}$ and $\|f^\textup{in}\|_{\mathcal{C}^2_{m,0}}$ such that for a.e. $t\in[0,T]$, for $C_*$ the implied constant in \eqref{3.1 result} and for $C_{**}$ the implied constant in \eqref{global Linfty},
\begin{align*}
\frac{d}{dt}\left(\textup{I}''(t) + \frac{1}{2} \int_{\mathbb{R}^3} \rho K\star_x \rho\,dx\right)&=4\iint_{\mathbb{R}^6} (-\Delta_v)^{-1}f\, f\,dvdx\\
&\le 4C_*\lVert f\rVert_{L^\infty_{t,x,v}}\iint_{\mathbb{R}^6} f\,dvdx\\
&\le 4C_*C_{**}\frac{\lVert f^\textup{in}\rVert_{L^1}}{\|f^\textup{in}\|_{\mathcal{C}^2_{m,0}}}\\
&=: C_1(\|f^\textup{in}\|_{L^1}, \|f^\textup{in}\|_{\mathcal{C}^2_{m,0}}).
\end{align*}
\end{remark}

\begin{proof}[Proof of Theorem \ref{thm.main2}] We aim to show that the spatial moment $\textup{I}(t)$ becomes negative in finite time under the assumptions of Theorem \ref{thm.main2}.

By Lemma \ref{I' I''}, and Remark \ref{E' est}, together with Lemma \ref{I'''}, we obtain the estimate
\begin{align*}
\textup{I}''(t)\le C_1t+kC_1+\textup{I}''(0), \quad \text{for a.e. }t\in[0,T].
\end{align*}Here, by arguments analogous to those in Remark \ref{E' est} and Lemma \ref{Linfty est f}, and recalling the constants $C_*, C_{**}$ introduced in Remark \ref{E' est}, we have
\begin{align*}
\sup_{t\in[0,T]} \int_{\mathbb{R}^3} \rho|K\star_x \rho|\,dx &\le C_*\| \rho\|_{L^\infty([0,T]\times\mathbb{R}^3)} \|f\|_{L^\infty_tL^1(\mathbb{R}^6)}\\
&\le C_*C_{**}\int_{\mathbb{R}^3}\langle v\rangle^{-m}\,dv \|f\|_{X_T^m} \|f^\textup{in}\|_{L^1}\\
&\le C_*C_{**} \left(\frac{4\pi}{3} + \frac{4\pi}{m-3}\right)\frac{\|f^\textup{in}\|_{L^1}}{\|f^\textup{in}\|_{\mathcal{C}^2_{m,0}}}\\
&=: C_2(\|f^\textup{in}\|_{L^1}, \|f^\textup{in}\|_{\mathcal{C}^2_{m,0}}) = k C_1,
\end{align*}for some $k(m) \eqdef \frac{m\pi}{3(m-3)}>0$, provided $m>3$. Integrating the inequality for $\textup{I}''(t)$ twice in $t$ yields
\begin{align*}
\textup{I}(t)\le \frac{C_1}{6}t^3+\left(kC_1+\textup{I}''(0)\right)\frac{t^2}{2}+\textup{I}'(0)t+\textup{I}(0)=:g(t), \quad \text{ for all }t\in[0,T].
\end{align*}We now show that the cubic polynomial $g$ attains a negative local minimum. Since $g'$ is a quadratic polynomial, it has two distinct real roots if and only if its
discriminant is positive. A direct computation gives
$$
\Delta \eqdef \bigl(kC_1 + \textup{I}''(0)\bigr)^2 - 2C_1 \textup{I}'(0).
$$
Because $\textup{I}'(0)<0$ by assumption, the second term on the right-hand side is strictly
positive, and hence $\Delta > 0$. 
Therefore, $g'(t)$ has two distinct real roots. Moreover, since $\textup{I}'(0)<0$ again, the larger
root $t_2$ satisfies $t_2>0$. As $g$ is a cubic polynomial, the identity $$g'(t_2)= \frac{C_1}{2}t_2^2 + (kC_1+\textup{I}''(0))t_2 + \textup{I}'(0)=0$$ allows us to
eliminate the quadratic term in $g(t_2)$, thereby rewriting $g(t_2)$ as a linear function of $t_2$:
\begin{align*}
g(t_2)
&= \left(\frac{2}{3}\textup{I}'(0)
      - \frac{(kC_1 + \textup{I}''(0))^2}{3C_1}\right)t_2
   + \left(\textup{I}(0) - \frac{(kC_1 + \textup{I}''(0))\textup{I}'(0)}{3C_1}\right) \\
&=: M_2 + M_3.
\end{align*}
By the positivity of the discriminant and the assumption $\textup{I}'(0)<0$, we have
$M_2<0$. Thus, to ensure that $g(t_2)<0$, it suffices to verify that
$M_3\le0$. A direct computation shows that this condition is equivalent to the hypothesis \eqref{inertia < energy}. Consequently,
$$
g(t_2) < 0 \qquad \text{for some } t_2>0,
$$
and hence $g$ attains a negative local minimum in finite time.

Finally, since $\textup{I}(t_2)<0$, the non-negative solution cannot be extended beyond $t_2>0$. This completes the proof of Theorem \ref{thm.main2}.
\end{proof}

\begin{proof}[Proof of Corollary \ref{collapse}]
In Corollary \ref{collapse}, recall that we assumed the existence of an $L^1_{x,v}$-valued measure $g_t: [0,t_{\max}]\to \mathcal{M}^1_2$, which represents a weak solution $f\in X_T^m\cap A_2$ on the time interval $[0,t_{\max})$. By mass conservation (Lemma \ref{mass conservation}), we have $$g_t(\mathbb{R}^6)=M>0$$ for all $t\in[0,t_{\max}].$ Also, note that for all $t<t_{\max}$, $$I(t)=\frac{1}{2}\iint_{\mathbb{R}^6}|x|^2 dg_t\ge 0.$$ Moreover, by the singularity mechanism established in Theorem \ref{thm.main2} together with continuity of the spatial moment $\textup{I}(t)$ up to $t_{\max}$ established in Lemma \ref{intch}, we obtain 
$$
\lim_{t\to t_{\max}^-} \textup{I}(t)=0.
$$
Fix $\varepsilon>0$. Then for every $t<t_{\max}$,
$$
\textup{I}(t)\ge \frac{\varepsilon^2}{2}\,g_t\big(\{|x|\ge\varepsilon\}\big),
$$
and therefore $g_t\big(\{|x|\ge\varepsilon\}\big)\to 0$ as $t\to t_{\max}^-$.
Consequently,
$$
g_{t_{\max}}\big(\{|x|\ge\varepsilon\}\big)=0,\qquad \forall\,\varepsilon>0.
$$
Since $\varepsilon>0$ is arbitrary, we conclude that
$$
\operatorname{supp}g_{t_{\max}}\subset\{x=0\}\times\mathbb{R}^3_v,
$$
which is a set of Lebesgue measure zero in $\mathbb{R}^6$.

In particular, the measure $g_{t_{\max}}$ is singular with respect to Lebesgue measure and therefore cannot be represented by an $L^1$-density. This shows that, at the maximal non-negativity time $t_{\max}$, the mass collapses onto a set of measure zero in phase space, and no continuation of the weak solution within our $L^1$-framework is possible.
\end{proof}

\section*{Acknowledgments}
\sloppypar{
 Both authors are supported by the National Research Foundation of Korea (NRF) grants RS-2023-00210484 and RS-2023-00219980. J. W. J. is also supported by the NRF grants 2022R1G1A1009044 and
 2021R1A6A1A10042944.}

\appendix
\section{Preliminary Lemmas}

In this section, we collect two lemmas that are used in the main sections.

We first establish an embedding property of the singular integrals, which will be used to estimate $(-\Delta_v)^{-1}f$ and $E_f$ in $C^\alpha_\textup{para}$.

\begin{lemma}[Boundedness of the singular integrals]\label{Linfty.est.sing.int}Let $1< p<\infty$ and let $f\in L^p(\mathbb{R}^3)\cap L^\infty(\mathbb{R}^3)$. Suppose $0< \beta <3$. Then, if $p<\frac{3}{3-\beta}$, there exists a constant $C>0$, depending only on $\beta,p$ such that
\begin{align}\label{3.1 result}
\lVert f\star |\cdot|^{-\beta}\rVert_{L^\infty(\mathbb{R}^3)}\le C  \big(\|f\|_{L^p} + \lVert f\rVert_{L^\infty(\mathbb{R}^3)}\big).
\end{align}
\end{lemma}

\begin{proof}Without loss of generality, suppose $f\ge0$. We estimate the integral  $f\star |\cdot|^{-\beta}$ into two regions:

\textbf{Local region }$|x-y|<1$: Using the assumption $f\in L^\infty$ and applying H\"older's inequality with exponent 1, we have
\begin{align*}
\int_{|x-y|<1}\frac{f(y)}{|x-y|^\beta}\,dy \le \lVert f\rVert_{L^\infty(\mathbb{R}^3)}\left(\int_{|x-y|<1} \frac{dy}{|x-y|^{\beta }}\right).
\end{align*}We perform a change of variable $z=x-y$, yielding:
\begin{align}\label{p>1 est}
C(\beta) \eqdef\int_{|x-y|<1}\frac{dy}{|x-y|^\beta} =  \int_{|z|<1} \frac{dz}{|z|^{\beta }}<\infty,
\end{align}since $\beta<3$. Hence,
\begin{align*}
\int_{|x-y|<1}\frac{f(y)}{|x-y|^\beta}\,dy \le C(\beta) \lVert f\rVert_{L^\infty(\mathbb{R}^3)}.
\end{align*}

\textbf{Far region }$|x-y|\ge1$: In this region, we have
\begin{align*}
\int_{|x-y|\ge1} \frac{f(y)}{|x-y|^\beta}\,dy \le \lVert f\rVert_{L^p(\mathbb{R}^3)} \left(\int_{|x-y|\ge1} \frac{dy}{|x-y|^{\beta p'}}\right)^{\frac{1}{p'}}.
\end{align*}As before, we define
\begin{align*}
C(\beta,p) \eqdef  \left(\int_{|z|\ge1} \frac{dz}{|z|^{\beta p'}}\right)^{\frac{1}{p'}} = \left(4\pi\int_1^\infty r^{2-\beta p'}\,dr\right)^{1-\frac{1}{p}} <\infty,
\end{align*}so that
\begin{align*}
\int_{|x-y|\ge1} \frac{f(y)}{|x-y|^\beta}\,dy \le C(\beta,p) \lVert f\rVert_{L^p(\mathbb{R}^3)}.
\end{align*}

Combining the estimates in all cases, we have the desired bound with the implied constant depending only on $\beta$ and $p$ and complete the proof.
\end{proof}

Next, we present a result that is used to establish the uniform boundedness of $\{C_n\}_{n\in\mathbb{N}}$ appearing in \eqref{C_n}.

\begin{lemma}[Boundedness of the compositions of exponential functions]\label{Phi property}Define
\begin{align*}
    \Phi(s) \eqdef ce^{A(s+1)},\quad \textup{ for } s>0,
\end{align*}for constants $c,A>0$, and denote the compositions of $\Phi$ by $\Phi^k =\underbrace{\Phi\circ\Phi\circ\cdots\circ\Phi}_{k\textup{ times}} $. Then the following statements hold:
\begin{enumerate}[label=(\roman*)]
\item \textit{Existence of fixed points: }The equation $\Phi(t) = t$ has two positive roots if and only if 
\begin{align}\label{small c}
 0<c < \frac{\exp\{-A\}}{eA}.
\end{align}
\item \textit{Boundedness of iterates: }Under the conditions in (i), suppose that equation $\Phi(t)=t$ has two positive roots, say $0<N_1<N_2$. Then, the sequence $\{\Phi^k(s)\}_{k\in\mathbb{N}}$ is uniformly bounded if and only if $s\le N_2$.
\item \textit{Positivity of iterates: }Again, under the conditions in (i), we have that for all $s\in(0,N_2)$, $\textup{inf}_{k\in\mathbb{N}}\Phi^k(s)>0$.
\end{enumerate}
\end{lemma}

\begin{proof}
(i) Notice that $\Phi(s)=s$ for some positive number $s>0$ if and only if $te^{-t} = cAe^{A}$ for $t\eqdef As$. 
But since $te^{-t}\le \frac{1}{e}$ for all $t\in\mathbb{R}$, the first statement is equivalent to \eqref{small c}.

(ii) and (iii) Suppose $\Phi(s)<s$ for some $s>0$ by letting $c$ be sufficiently small. Then, the graph of $y=\Phi(s)$ intersects the identity line $y=s$ at exactly two points, yielding two fixed points
\begin{align*}
0<N_1<N_2<\infty, \quad \textup{such that}\quad \Phi(N_1) = N_1, \quad \Phi(N_2) = N_2.
\end{align*}We analyze three cases based on the initial value $s\in(0,\infty)$.
\begin{enumerate}[label=Case \arabic*. , leftmargin=4em]
\item $0<s \le N_1$: Since $\Phi(s)>s$ for $0\le s<N_1$, and $\Phi$ is strictly increasing, the sequence $\{\Phi^k(s)\}_{k\in\mathbb{N}}$ is strictly increasing. Moreover, for all such $s$, since $\Phi(s)<N_1$, the sequence is bounded above by $N_1$. We have proved (ii) and (iii) in this case.
\item $N_1<s<N_2$: In this interval, $\Phi(s)<s$, so the sequence $\{\Phi^k(s)\}_{k\in\mathbb{N}}$ is strictly decreasing. Again, because $\Phi(s)>N_1$ for $s\in(N_1,N_2)$, the sequence is bounded below by $N_1$. We have proved (ii) and (iii) in this interval.
\item $s>N_2$: Here $\Phi(s)>s$, and since $\Phi$ is strictly convex on $(0,\infty)$ (i.e., $\Phi''>0$), and $\Phi'(N_2)>1$, the sequence $\{\Phi^k(s)\}$ grows faster with each iteration. More precisely, for each $s>N_2$, applying the mean value theorem yields
\begin{align*}
\Phi^2(s) - \Phi(s) = \Phi'(s')(\Phi(s) - s)>\Phi(s) - s,
\end{align*}for some $s'\in (s,\Phi(s))$, using $\Phi'(s')>1$. Inductively, the sequence $\{\Phi^k(s)\}$ satisfies
\begin{align*}
\Phi^{k+1}(s) - \Phi^k(s) > \Phi^k(s) - \Phi^{k-1}(s),
\end{align*}and diverges to $\infty$. This shows $\{\Phi^k(s)\}$ is unbounded for $s>N_2$.
\end{enumerate}
Consequently, for any case, (ii) and (iii) hold, completing the proof.
\end{proof}

\section{Auxiliary Theorems on Parabolic Equations}

In this section, we state a bijection theorem for bounded linear operators, particularly applicable to uniformly elliptic and parabolic equations. Furthermore, we introduce fundamental and advanced \textit{a priori} estimates for general uniformly parabolic operators. 

The bijection theorem for bounded linear operators follows. 
\begin{theorem}[Method of continuity, {\cite[Theorem 5.2]{Ell_Gilbarg_1998}}]\label{MoC}Let $E,F$ be the two Banach spaces with norms $\| \cdot\|_E$ and $\| \cdot\|_F$, respectively and let $\mathcal{L}_0$, $\mathcal{L}_1$ be bounded linear operators from $E$ into $F$. For each $t\in[0,1]$, set
\begin{align*}
\mathcal{L}_t = (1-t)\mathcal{L}_0 + t\mathcal{L}_1
\end{align*}and suppose that there is a constant $C$ such that
\begin{align*}
\lVert x\rVert_E \le C \lVert \mathcal{L}_tx\rVert_F
\end{align*}for $t\in[0,1]$. Then $\mathcal{L}_1$ maps $E$ onto $F$ if and only if $\mathcal{L}_0$ maps $E$ onto $F$.
\end{theorem}

Then, we provide several \textit{a priori} estimates for general uniformly parabolic operators. Recall the relevant function spaces introduced in Section \ref{notations}.
\begin{theorem}[Schauder estimate, {\cite[Theorem 8.10.1]{Ell_Krylov_1996}}]\label{Schauder} Assume that $L$ is a uniformly elliptic operator and that the zeroth-order coefficient satisfies $c\le-\lambda$ for some constant $\lambda>0$. Let $T\in(-\infty,\infty]$ (including $\infty$) and let $\alpha\in(0,1)$. Suppose that for some constant $K>0$,
\begin{align*}
\lVert b^i\rVert_{C_\textup{para}^\alpha(\Omega_T)}, \lVert c\rVert_{C_\textup{para}^\alpha(\Omega_T)} \le K \quad \textup{ for }\,\, 1\le i\le n.
\end{align*}Then for each $u\in C^{2,\alpha}_\textup{para}(\mathcal{Q}_T)$, there exists a constant $N$ depending only on $n, \kappa,\alpha, K,\lambda$ such that
\begin{align*}
\lVert u\rVert_{C_\textup{para}^{2,\alpha}(\Omega_T)} \le N\lVert \partial_tu - Lu\rVert_{C_\textup{para}^\alpha(\Omega_T)}.
\end{align*}
\end{theorem}

\begin{theorem}[$L^p$ estimate, Theorem 5.2.1 of {\cite{Ell_Krylov_2008}}]\label{L^p}  Assume that $L$ is a uniformly elliptic operator and that the zeroth-order coefficient satisfies $c\le -\lambda$ for some constant $\lambda\ge1$. Suppose further that for some $K>0$,
\begin{align*}
\lVert (b^1,\cdots, b^n)\rVert_{L^\infty(\Omega_T)} + \lVert c\rVert_{L^\infty(\Omega_T)} \le K.
\end{align*}In addition, assume that there exists an increasing modulus of continuity $\omega(\varepsilon)$, $\varepsilon\ge0$, such that $\omega(\varepsilon)\downarrow 0$ as $\varepsilon\downarrow0$ and that for all $t<T$, $x,y\in\mathbb{R}^n$ and $i,j=1,\cdots,n$
\begin{align*}
|a^{ij}(t,x) - a^{ij}(t,y)| \le \omega(|x-y|).
\end{align*}Then, for any $1<p<\infty$, there exists a constant $N$ depending only on $K,\kappa,\omega, p,$ and $n$ such that
\begin{align*}
\lVert u\rVert_{W^{1,2}_p(\Omega_T)} \le N \lVert \partial_tu - Lu \rVert_{W^{1,2}_p(\Omega_T)}
\end{align*}for all $u\in W^{1,2}_p(\Omega_T)$.
\end{theorem}

\begin{theorem}[Weighted mixed $L^p$ estimate, Theorem 6.3 of {\cite{Ell_Dong_2018}}]\label{Weighted mixed L^p estimate} Let $p,q\in(1,\infty)$ with $p\le q$, and let $K_0\ge1$ be a constant. Let the weight $w=w_1(x')w_2(t,x'')$ be of product form, where
\begin{gather*}
w_1(x')\in A_p(\mathbb{R}^{d_1}, dx'), \quad w_2(t,x'') \in A_q(\mathbb{R}^{1+d_2}, dx''\,dt),\\
d_1 + d_2 = d, \quad [w_1]_{A_p}\le K_0, \quad [w_2]_{A_q} \le K_0.
\end{gather*}Let $L = a^{ij}D_{ij} + b^i D_i + c$ be a uniformly elliptic operator satisfying the following conditions:
\begin{enumerate}[label=(\roman*)]
\item The coefficients $a^{ij}$ are continuous in the same sense as in Theorem \ref{L^p} with modulus of continuity $\omega$.
\item The coefficients $b^i$ and $c$ are measurable and bounded. In particular, there exists a constant $K\in(0,\infty)$ such that
\begin{align*}
|b^i|\le K, \quad |c|\le K.
\end{align*}
\end{enumerate}Then there exist a constant
\begin{gather*}
\lambda_0 = \lambda_0(d,\kappa,p,q,d_1,d_2, K_0,K,\omega)\ge1,
\end{gather*}such that the following holds. If $u\in W_{p,q,w}^{1,2}(\mathbb{R}^{1+d})$ satisfies
\begin{align*}
-u_t + Lu - \lambda u = f
\end{align*}for some  $f\in L_{p,q,w}(\mathbb{R}^{1+d})$, then for all $\lambda\ge\lambda_0$,
\begin{align*}
\lVert u_t\rVert_{L_{p,w}} + \lambda\lVert u\rVert_{L_{p,q,w}} + \sqrt{\lambda}\lVert Du\rVert_{L_{p,q,w}} + \lVert D^2u\rVert_{L_{p,q,w}} \le N\lVert f\rVert_{L_{p,q,w}},
\end{align*}where
$N=N(d,\kappa,p,q,d_1,d_2,K_0,K,\omega).$
\end{theorem}

\bibliographystyle{amsplaindoi}
\bibliography{bibliography.bib}{}
\end{document}